\documentclass[12pt]{amsart}
\usepackage{amsmath,amsthm,amsfonts,amssymb,eucal}

\renewcommand {\a}{ \alpha }
\renewcommand{\b}{\beta}
\newcommand{\e}{\epsilon}

\newcommand{\g}{\gamma}

\newcommand{\vark}{\varkappa}
\renewcommand{\d}{\delta}
\newcommand{\s}{\sigma}
\renewcommand{\l}{\lambda}
\renewcommand{\L}{\Lambda}
\newcommand{\z}{\zeta}
\renewcommand{\t}{\theta}

\newcommand{\T}{\Theta}

\newcommand{\om}{\omega}
\newcommand{\Om}{\Omega}

\newcommand{\oq}{\ {\raise 7pt\hbox{${\scriptstyle\circ}$}}
\kern -7pt{
\hbox{$Q$}}}

\newcommand{\R}{ \mathbb R}

\newcommand{\Rd}{ \mathbb R^d}

\newcommand{\Torus}{\mathbb T}
\newcommand{\Td}{\mathbb T^d}

\newcommand {\GA}{\mathfrak A}

\newcommand {\GG}{\mathfrak G}
\newcommand {\GH}{\mathfrak H}

\newcommand {\GV}{\mathfrak V}
\newcommand {\GW}{\mathfrak W}
\newcommand {\GU}{\mathfrak U}

\newcommand {\GX}{\mathfrak X}
\newcommand {\ba}{\mathbf a}

\newcommand {\bb}{\mathbf b}
\newcommand{\bv}{\mathbf v}

\newcommand {\BD}{\mathbf D}

\newcommand {\BS}{\mathbf S}

\newcommand {\bx}{\mathbf x}

\newcommand {\be}{\mathbf e}

\newcommand {\bk}{\mathbf k}
\newcommand {\bp}{\mathbf p}

\newcommand {\bm}{\mathbf m}

\newcommand {\bz}{\mathbf z}

\newcommand {\bt}{\mathbf t}

\newcommand {\bs}{\mathbf s}

\newcommand {\bn}{\mathbf n}

\newcommand{\SG}{{\sf{\Gamma}}}

\newcommand{\SN}{{\sf{N}}}
\newcommand {\bnu}{\boldsymbol\nu}

\newcommand {\bmu}{\boldsymbol\mu}
\newcommand {\bth}{\boldsymbol\theta}
\newcommand {\boldeta}{\boldsymbol\eta}

\newcommand {\bphi}{\boldsymbol\phi}
\newcommand {\bxi}{\boldsymbol\xi}
\newcommand {\bg}{\boldsymbol\gamma}

\newcommand{\bchi}{\boldsymbol\chi}

\newcommand {\normal}{\mathbf e}

\newcommand{\lra}{\leftrightarrow}

\newcommand{\lu}{\langle}
\newcommand{\ru}{\rangle}

\newcommand{\CV}{\mathcal V}

\newcommand{\CB}{\mathcal B}

\newcommand{\CT}{\mathcal T}
\newcommand{\CX}{\mathcal X}
\newcommand{\CH}{\mathcal H}
\newcommand{\CO}{\mathcal O}
\newcommand{\CP}{\mathcal P}

\newcommand{\CA}{\mathcal A}

\newcommand{\CC}{\mathcal C}

\newcommand{\CD}{\mathcal D}

\newcommand{\plainC}[1]{\textup{{\textsf{C}}}^{#1}}

\newcommand{\plainS}{\textup{{\textsf{S}}}}

\newcommand{\plainH}[1]{\textup{{\textsf{H}}}^{#1}}

\newcommand{\plainL}[1]{\textup{{\textsf{L}}}^{#1}}

\DeclareMathOperator{\card}{{card}}

\newcommand{\1}
{{\,\vrule depth3pt height9pt}{\vrule depth3pt height9pt}
{\vrule depth3pt height9pt}{\vrule depth3pt height9pt}\,}

\DeclareMathOperator {\dist} {{dist}}

\DeclareMathOperator \volume {{vol}}
\DeclareMathOperator \area {{vol}}

\DeclareMathOperator{\op}{{Op}}

\DeclareMathOperator {\ad}{{ad}}
\DeclareMathOperator{\supp}{{supp}}

\DeclareMathOperator{\w}{w}
\DeclareMathOperator{\dc}{d}

\newtheorem{thm}{Theorem}[section]
\newtheorem{cor}[thm]{Corollary}
\newtheorem{lem}[thm]{Lemma}
\newtheorem{prop}[thm]{Proposition}

\theoremstyle{definition}
\newtheorem{defn}[thm]{Definition}

\newtheorem{convention}[thm]{Convention}

\newtheorem{rem}[thm]{Remark}

\numberwithin{equation}{section}

\newcommand{\bee}{\begin{equation}}
\newcommand{\ene}{\end{equation}}
\newcommand{\bees}{\begin{equation*}}
\newcommand{\enes}{\end{equation*}}
\newcommand{\bes}{\begin{split}}
\newcommand{\ens}{\end{split}}

\newcommand{\bet}{\begin{thm}}
\newcommand{\ent}{\end{thm}}
\newcommand{\bel}{\begin{lem}}
\newcommand{\enl}{\end{lem}}
\newcommand{\bec}{\begin{cor}}
\newcommand{\enc}{\end{cor}}
\newcommand{\bep}{\begin{proof}}
\newcommand{\enp}{\end{proof}}
\newcommand{\ber}{\begin{rem}}
\newcommand{\enr}{\end{rem}}
\newcommand{\ep}{\varepsilon}
\newcommand{\la}{\lambda}

\newcommand{\de}{\delta}
\newcommand{\al}{\alpha}
\newcommand{\Z}{\mathbb Z}

\newcommand {\BUps}{\boldsymbol\Upsilon}

 \renewcommand{\uparrow}{{\mathcal {L F}}}
\renewcommand{\sharp}{{\mathcal {L E}}}
\renewcommand{\natural}{{\mathcal {NR}}}
\renewcommand{\flat}{{\mathcal R}}
\renewcommand{\downarrow}{{\mathcal {S E}}}

\begin{document}

\hoffset -4pc

\title[Bethe-Sommerfeld conjecture]
{Bethe-Sommerfeld conjecture for periodic operators with strong perturbations}
\author[L. Parnovski \& A.V. Sobolev]
{Leonid Parnovski \& Alexander V. Sobolev}
\address{Department of Mathematics\\ University College London\\
Gower Street\\ London\\ WC1E 6BT UK}
\email{Leonid@math.ucl.ac.uk}
\email{asobolev@math.ucl.ac.uk}

\keywords{ Bethe-Sommerfeld conjecture, periodic problems, pseudo-differential operators,
spectral gaps}
\subjclass[2000]{Primary 35P20, 47G30, 47A55; Secondary 81Q10}


\begin{abstract}
We consider a periodic self-adjoint pseudo-differential
operator $H=(-\Delta)^m+B$, $m>0$, in $\R^d$ which satisfies the
following conditions: (i) the symbol of $B$ is smooth in $\bx$, and (ii) the perturbation
$B$ has order less than $2m$. Under these assumptions, we prove that the spectrum
of $H$ contains a half-line. This, in particular implies the Bethe-Sommerfeld
Conjecture for the Schr\"odinger operator with a periodic
magnetic potential in all dimensions.
\end{abstract}

\maketitle
\vskip 0.5cm

 
\section{Introduction}

Under very broad  conditions, spectra of elliptic differential
operators with periodic coefficients in $\plainL2(\R^d)$, $d\ge 1$,
have a band structure, i.e.
they represent a union of closed intervals (\textsl{bands}),
possibly separated by spectrum-free intervals
(\textsl{gaps}) (see \cite{RS} and \cite{Kuch}).
Since the 30's it has been a general belief among the physicists
that the number of gaps in the spectrum of the
Schr\"odinger operator $H_V = -\Delta + V$ with a periodic electric
potential $V$ in dimension three must be finite. After the classical monograph \cite{BS}
this belief is known as the Bethe-Sommerfeld conjecture.
It is relatively straightforward
to see that this conjecture holds for potentials which admit a
partial separation of variables, as shown in \cite{E}, p.121.
For general potentials
this problem turned out to be quite difficult, and the first rigorous results
appeared only in the beginning  of the 80's.
We do not intend to discuss these and more recent results in details, but
refer to \cite{Sob1} for a more comprehensive survey and further references. Here
we content ourselves with a very short description.
 
In the case of the Schr\"odinger operator $H_V$
it is known that the number of gaps is generically infinite
if $d=1$ (see \cite{RS}). For $d\ge 2$  there has been a large number
of publications proving the
conjecture for $H_V$ under various conditions on the potential and
the periodicity lattice.
First rigorous results for the Schr\"odinger operator relied on number-theoretic ideas,
 and they appeared
in \cite{PopSkr}, \cite{DahlTru} ($d=2$) and
\cite{Skr0} -\cite{Skr2} ($d\ge 2$).
At that time it was found out that the complexity of the problem
increases together with the dimension: the validity of the conjecture for dimensions
$d\ge 4$ was established only for rational lattices, see \cite{Skr1}.
Later  the conjecture for arbitrary lattices was extended to $d=4$ in \cite{HelMoh}.
The definitive  result was obtained
in the recent paper \cite{P} where the Bethe-Sommerfeld
conjecture was proved for the Schr\"odinger operator for any periodicity lattice in
all dimensions $d\ge 2$,
with an arbitrary smooth potential $V$ (see \cite{Vel} for an alternative approach).

The study of the polyharmonic operator  $(-\Delta)^m + V$, $m >0$
in \cite{Skr1}, \cite{Kar} and \cite{PS1}, \cite{PS2}
revealed that large values of $m$ facilitate the overlap of the spectral bands.
Precisely, it was found that for $4m > d+1$ the Bethe-Sommerfeld conjecture holds for
arbitrary bounded perturbations $V$ (see \cite{PS1}), and if
$8m > d+3$, then it holds for arbitrary smooth potentials $V$ (see \cite{PS2}).

Returning to the case of the Schr\"odinger operator, we observe
that the complexity of the problem increases dramatically
when instead of the bounded potential perturbation one
introduces in the Schr\"odinger operator a periodic magnetic potential
$\ba = (a_1, a_2, \dots, a_d)$:
$(-i\nabla - \ba)^2 + V$.
Until recently the Bethe-Sommerfeld conjecture for this operator was known to hold only
for $d=2$, see \cite{Moh}, \cite{Karp2}.
A new step towards the study of higher order perturbations
was made in \cite{BarPar}, where the methods of \cite{P} were extended to
the operator
\begin{equation}\label{operator:eq}
H=(-\Delta)^m + B, \ \ m >0,
\end{equation}
with a pseudo-differential
perturbation $B$ of order  $a<2m-1$ and arbitrary $d\ge 2$.

In the present paper we prove the Bethe-Sommerfeld conjecture for the operator
\eqref{operator:eq} for arbitrary $B$ of order $a <2m$,
see Theorems \ref{main:thm} and \ref{main1:thm}.
In particular, our result covers the magnetic
Schr\"odinger operator with a smooth periodic vector potential in any dimension $d\ge 2$.

Our  proof is based on a subtle analysis of the Floquet eigenvalues of the operator $H$.
It is known that the Floquet eigenvalues are divided in two groups: stable (or non-resonant)
and unstable (or resonant).
>From the perturbation-theoretic point of view, the stable eigenvalues are generated by the
isolated non-degenerate Floquet eigenvalues of the free operator $H_0 = (-\Delta)^m$, and
hence they can be described
using standard methods of the theory.
For the Schr\"odinger operator in dimensions $d=2,3$ it was done in \cite{FKT0}.
The unstable eigenvalues, on the contrary, are produced by the clusters of close
(or even degenerate) eigenvalues of $H_0$, and their detailed
description is not that simple.
However, under appropriate conditions on the parameters of the problem, e.g.
the orders $m$, $a$ and dimension $d$,
only a crude estimate on the unstable eigenvalues suffices to
show that their contribution is negligible.
For example, as shown in papers \cite{PS1}, \cite{PS2}
for the polyharmonic operator $(-\Delta)^m +V$ (i.e. when $a = 0$) under
the condition $8m > d+3$
the contribution of the stable, ``controllable" eigenvalues is dominant, and
for the unstable ones it suffices to obtain an appropriate
upper bound on their quantity.
Another example is the result of \cite{Skr1} (see also \cite{SkrSob1})
where, under
the condition that the lattice is rational and $d\ge 4$,
an elementary estimate on the unstable eigenvalues
guarantees that their contribution can be ignored.
In the present paper we go beyond all these restrictions, and hence we are forced to study
the unstable eigenvalues in detail.

Associated with the partition of the eigenvalues into two groups,
is a partition of the phase space
into resonant and non-resonant zones (sets).
In fact, the main focus of the present
paper is the precise construction of these zones and
understanding of the eigenvalues associated with them.
Technically, our approach is a combination of methods of \cite{P} and \cite{Sob}.
Our construction of the resonant zones is a simplified variant of that suggested in
\cite{P}. However, in spite of the simplification, these are
rather complicated geometrical objects, and the study of their properties  is not
straightforward.
The reduction of the operator to the resonant and non-resonant parts is done
using the ``near-similarity" approach of \cite{Sob}. It consists in
finding a unitary operator $U$ such that $A=U^* H U$ is ``almost" an
operator with constant coefficients.
The operator $U$ is sought in the form $e^{i\Psi}$ where $\Psi$ is a self-adjoint periodic
PDO, and hence we sometimes call this similarity transformation
a ``gauge transformation".
For $d=1$ such a reduction to constant coefficients can be done
(see \cite{R}, \cite{Sob0}), but
for $d\ge 2$ only a partial reduction is possible.
Namely,
we explicitly describe the procedure of finding a pseudo-differential operator $\Psi$
with symbol $\psi(\bx, \bxi)$ such that the operator $A = e^{i\Psi} H e^{-i\Psi}$
has constant coefficients in the non-resonant zone.
Thus the Floquet eigenvalues in the non-resonant subspace
can be found explicitly, which leads
to relatively straightforward estimates for the band overlap.
As far as the resonant zones are concerned, our
construction ensures that on each of them the new operator
$A$ admits a partial separation of variables (see \cite{FKT} for a similar observation
for the Schr\"odinger operator in dimensions $d = 2, 3$). This fact enables us to
show that the volume (more precisely, the angular measure) of the resonant sets
is negligibly small compared to the non-resonant one. Having established this
fact, we apply the combinatorial-geometric argument of \cite{P}, which
allows us to deduce that the resonant zones do not destroy the band overlap obtained for the
non-resonant one.

To conclude the introduction, we give a brief outline of the paper.
In the next section, we introduce necessary notation, discuss the classes of
pseudo-differential operators we will be using throughout and formulate
the main result of the paper. In Section 3,
we provide necessary information about the classes of pseudo-differential
operators introduced in Section 2.
The ``gauge transformation" is studied in Section 4.
In Section 5, we describe the partition of the phase space
into resonant and non-resonant zones.
This section has a purely combinatorial-geometric character, and
can be read separately from the rest of the paper.
In Section 6, we construct the decomposition of
$A$ into an orthogonal sum over the resonant and non-resonant subspaces.
On the basis of this decomposition
we study the Floquet eigenvalues of $A$ in Section 7.
Sections 8, 9 are concerned with estimates for the volumes of resonant and non-resonant sets.
These estimates
become the central ingredient of the proof, completed in Section 10.

{\bf Acknowledgment.} This work was supported by the EPSRC grants EP/F029721/1 and
EP/D00022X/2.
The first author
was partially supported by the Leverhulme fellowship.
We would like to thank R.Shterenberg and S. Morozov
for reading the preliminary version of this manuscript and making useful comments.

\section{Periodic
pseudo-differential operators. Main result}\label{setting:sect}

\subsection{Classes of PDO's}\label{classes:subsect}
Before we define the pseudo-differential operators (PDO's), we
introduce the relevant classes of symbols. Let $\SG\in\Rd$
be a lattice. Denote by $\CO$ its fundamental domain. For example,
for $\CO$ one can choose a parallelepiped spanned by a basis of
$\SG$. The dual lattice and its fundamental domain are denoted by
$\SG^\dagger$ and $\CO^\dagger$ respectively. Sometimes we reflect
the dependence on the lattice and write $\CO_{\SG}$ and
$\CO^\dagger_\SG$. In particular, in the case $\SG = (2\pi\Z)^d$
one has $\SG^\dagger = \Z^d$ and it is natural to take $\CO = [0,
2\pi)^d$, $\CO^\dagger = [0, 1)^d$. For any measurable set
$\CC\subset\R^d$ we denote by $|\CC|$ or $\volume (\CC)$ its
Lebesgue measure (volume). The volume of the fundamental domain
does not depend on its choice, it is called the
\textsl{determinant of the lattice} $\SG$ and denoted $\dc(\SG) =
|\CO|$. By $\be_1, \be_2, \dots, \be_d$ we denote the standard
orthonormal basis in $\R^d$.

For any $u\in \plainL2(\CO)$ and $f\in\plainL2(\R^d)$ define the
Fourier coefficients and Fourier transform respectively:
\begin{equation*}
\hat u(\bth) = \frac{1}{\sqrt{\dc(\SG)}} \int_{\CO} e^{-i \lu\bth,
\bx\ru} u(\bx) d\bx,\ \bth\in\SG^\dagger,\ \ (\mathcal F f)(\bxi)
= \frac{1}{(2\pi)^{\frac{d}{2}}} \int_{\R^d} e^{-i\lu\bxi,
\bx\ru}f(\bx) d\bx,\ \bxi\in\R^d.
\end{equation*}
Let us now define the periodic symbols and PDO's associated with
them. Let $b = b(\bx, \bxi)$, $\bx, \bxi\in\R^d$, be a
$\SG$-periodic complex-valued function, i.e.
\begin{equation*}
b(\bx+\bg, \bxi) = b(\bx, \bxi),\ \forall \bg\in\SG.
\end{equation*}
Let $w:\R^d\to \R$ be a locally bounded function such that
$w(\bxi)\ge 1\ \forall \bxi\in\Rd$ and
\begin{equation}\label{weight:eq}
w(\bxi + \boldeta)\le C w(\bxi) \lu\boldeta\ru^\kappa, \ \forall
\bxi, \boldeta\in \R^d,
\end{equation}
for some $\kappa\ge 0$. Here we
have used the standard notation $\lu \bt \ru = \sqrt{1+|\bt|^2},\
\forall \bt\in\R^d$.  We say that the symbol $b$ belongs to the
class $\BS_{\a} = \BS_{\a}(w) = \BS_{\a}(w, \SG)$,\ $\a\in\R$, if
for any $l\ge 0$ and any non-negative $s\in\Z$ the condition
\begin{equation}\label{1b1:eq}
\1 b \1^{(\a)}_{l, s} :=
\max_{|\bs| \le s}
\sup_{\bxi, \bth} \lu \bth\ru^{l}\ w(\bxi)^{-\a + |\bs|}
|\BD_{\bxi}^{\bs} \hat b(\bth, \bxi)|<\infty, \ \ |\bs| = s_1+ s_2 + \dots + s_d,
\end{equation}
is fulfilled.
Here, of course, $\hat b_{\boldeta}(\bth, \bxi)$ is the Fourier coefficient
of the symbol $b(\cdot,\bxi)$ with respect to the first variable.
The quantities \eqref{1b1:eq} define norms on the class $\BS_\a$.
In the situations when it is not important for us to know the exact
values of $l, s$, we denote the above norm
by $\1 b\1^{(\a)}$. In this case
the inequality $A \le C\1 b\1^{(\a)}$ means that there exist values of $l$ and $s$,
and a constant $C>0$, possibly depending on $l, s$,
such that $A \le C\1 b\1^{(\a)}_{l, s}$. Similarly, when we write
$\1 b\1^{(\g)} \le C\1 g\1^{(\a)}$ for some symbols $b\in\BS_\g, g\in\BS_\a$, we mean that
for any $l$ and  $s$ the norm $\1 b\1^{(\g)}_{l, s}$
is bounded by $\1 g\1^{(\a)}_{p, n}$ with some $p$ and $n$ depending on $l, s$, and
some constant $C = C_{l, s}$. In general,
by $C, c$(with or without indices) we denote various positive constants, whose precise value
is unimportant.
Throughout
the entire paper we adopt the following convention. An
estimate (or an assertion) is said to be uniform in a symbol
$b\in\BS_\a$ if the constants in the estimate (or
assertion) at hand depend only on the constants $C_{l, s}$
in the bounds $\1 b \1^{(\a)}_{l, s}\le C_{l, s}$.
This is sometimes expressed by saying that an estimate (or assertion) is
uniform in the symbol $b$ satisfying $\1 b\1^{(\a)} \le C$.

We use the classes
$\BS_\a$ mainly with the weight $w(\bxi) = \lu\bxi\ru^\b$, $\b\in (0, 1]$, which
satisfies \eqref{weight:eq} for $\kappa = \b$.
 Note that $\BS_\a$ is an increasing function of $\a$,
i.e. $\BS_{\a}\subset\BS_{\g}$ for $\a < \g$. For later reference we
write here the following convenient bounds that follow from definition
\eqref{1b1:eq} and property \eqref{weight:eq}:
\begin{gather}
|\BD_{\xi}^{\bs} \hat b(\bth, \bxi)|\le \1 b \1^{(\a)}_{l, s} \lu
\bth\ru^{-l} w(\bxi)^{\a-s},
\label{decay:eq}\\
|\BD^{\bs}_{\bxi}\hat b(\bth, \bxi+ \boldeta) -
\BD^{\bs}_{\bxi}\hat b(\bth, \bxi)|\le C\1 b\1^{(\a)}_{l, s+1} \lu
\bth\ru^{-l} w(\bxi)^{\a-s-1} \lu\boldeta\ru^{\kappa|\a-s-1|}
|\boldeta|, \ s = |\bs|, \label{differ:eq}
\end{gather}
with a constant $C$ depending only on $\a, s$. For a vector $\boldeta\in\R^d$ introduce
the symbol
\begin{equation}\label{bboldeta:eq}
b_{\boldeta}(\bx, \bxi) = b(\bx, \bxi+\boldeta), \boldeta\in\R^d,
\end{equation}
so that $\hat b_{\boldeta}(\bth, \bxi) = \hat b(\bth, \bxi+\boldeta)$ .
The bound \eqref{differ:eq} implies that for all $|\boldeta|\le C$ we have
\begin{equation}\label{differ1:eq}
\1 b - b_{\boldeta}\1^{(\a-1)}_{l, s}\le C_s \1 b\1^{(\a)}_{l, s+1}|\boldeta|,\
\end{equation}
uniformly in $\boldeta$: $|\boldeta|\le C$.

Now we define the PDO $\op(b)$ in the usual way:
\begin{equation*}
\op(b)u(\bx) = \frac{1}{(2\pi)^{\frac{d}{2}}} \int  b(\bx, \bxi)
e^{i\lu\bxi, \bx\ru} (\mathcal Fu)(\bxi) d\bxi,
\end{equation*}
the integrals being over $\Rd$. Under the condition $b\in\BS_\a$
 the integral in the r.h.s. is clearly finite for any
$u$ from the Schwarz
class $\plainS(\Rd)$. Moreover, the condition $b\in \BS_0$
guarantees the boundedness of $\op(b)$ in $\plainL2(\Rd)$, see
Proposition \ref{bound:prop}. Unless otherwise stated, from now on
$\plainS(\Rd)$ is taken as a natural domain for all PDO's at hand. 
Observe that the operator $\op(b)$ is
symmetric if its symbol satisfies the condition
\begin{equation}\label{selfadj:eq}
\hat b(\bth, \bxi) = \overline{\hat b(-\bth, \bxi+\bth)}.
\end{equation}
We shall call such symbols \textsl{symmetric}.

Our aim is to study the spectrum of the operator
 \begin{equation}\label{h:eq}
\begin{cases}
H = \op(h),\ h(\bx, \bxi) = h_0(\bxi) + b(\bx, \bxi),\\[0.2cm]
h_0(\bxi) =  |\bxi|^{2m},\  m >0,\\[0.2cm]
\ b\in\BS_{\a}(\lu\bxi\ru^\b), \
\a\b < 2m,
\end{cases}
\end{equation}
with a symmetric symbol $b$.
The operator $\op(b)$ is infinitesimally
$H_0$-bounded, see Lemma \ref{formbound:lem},
so that $H$
is self-adjoint on the domain $D(H) = D(H_0) = \plainH{2m}(\Rd)$.
Due to the $\SG$-periodicity of the
symbol $b$, the operator $H$ commutes with the shifts along the
lattice vectors, i.e.
\begin{equation*}
H \CT_{\bg} = \CT_{\bg} H, \ \bg\in\SG.
\end{equation*}
with $(\CT_{\bg} u)(\bx) = u(\bx+\bg)$. This allows us to use the
\textsl{Floquet decomposition}.

\subsection{Floquet decomposition}\label{floquet:subsect}
We identify the underlying Hilbert space $\CH = \plainL2(\R^d)$
with the direct integral
\begin{equation*}
\GG = \int_{\CO^\dagger} \GH  d\bk,\ \GH = \plainL2 (\CO).
\end{equation*}
This identification is implemented  by the Gelfand transform
\begin{equation}\label{gelfand:eq}
(U u)(\bx, \bk) = \frac{1}{\sqrt{\dc(\SG^\dagger)}} e^{-i\lu\bk,
\bx\ru} \sum_{\bg\in \SG} e^{-i \lu\bk, \bg\ru } u(\bx +  \bg),\
\bk\in\R^d,
\end{equation}
which is initially defined on $u\in \plainS(\Rd)$ and extends by
continuity  to a unitary mapping from $\CH$ onto $\GG$. In terms
of the Fourier transform the Gelfand transform is defined as
follows: $\widehat{(Uu)}(\bth, \bk) = (\mathcal F u)(\bth + \bk),\
\bth\in\SG^\dagger$. The unitary operator $U$ reduces $\CT_{\bg}$
to the diagonal form:
\begin{equation*}
(U \CT_{\bg}U^{-1} f)(\ \cdot\ , \bk) = e^{i \bk\cdot\bg} f (\
\cdot\ , \bk),\  \forall \bg\in \SG.
\end{equation*}
Let us consider a self-adjoint operator $A$ in $\CH$ which
commutes with $\CT_{\bg}$ for all $\bg\in\SG$, i.e. $A\CT_{\bg} =
\CT_{\bg}A$. We call such operators ($\SG$-)periodic.
Then $A$ is partially diagonalised by $U$ (see
\cite{RS}), that is, there exists a measurable family of
self-adjoint operators (fibres) $A(\bk), \bk\in\CO^\dagger$
acting in $\GH$, such that
\begin{equation}\label{direct:eq}
U A U^*  = \int_{\CO^{\dagger}} A(\bk)d\bk.
\end{equation}
It is easy to show that any periodic  
operator $T$, which
is $A$-bounded with relative bound $\e<1$, can be also decomposed
into a measurable set of fibers $T(\bk)$ in the sense that
\begin{equation*}
(UT f)(\ \cdot\ , \bk) = T(\bk) (U f)(\ \cdot\ , \bk),\
\textup{a.e.}\  \bk\in\CO^{\dagger},
\end{equation*}
for all $f\in D(A)$. Moreover, the fibers $T(\bk)$ are
$A(\bk)$-bounded with the bound $\e$, and if $T$ is symmetric,
then the operator $A(\bk) + T(\bk)$ is self-adjoint on
$D(A(\bk))$.

Suppose that the operator $A$ (and hence $A(\bk)$) is bounded from
below and that the spectrum of each $A(\bk)$ is discrete. Denote
by $\l_j\bigl(A(\bk)\bigr), j = 1, 2, \dots,$ the eigenvalues of
$A(\bk)$ labeled in the ascending order. Define the counting function in the usual way:
\begin{equation*}
N\bigl(\l, A(\bk )\bigr) = \#\{j: \l_j\bigl(A(\bk)\bigr)\le \l\},\
\l\in\R.
\end{equation*}
If $A = \op(a)$ with a real-valued symbol
$a\in \plainL\infty_{\textup{\tiny loc}}(\Rd)$ depending only on $\bxi$,
then $A(\bk)$ is a self-adjoint PDO in $\GH$ defined as follows:
\begin{equation*}
A(\bk) u(\bx) = \frac{1}{\sqrt{\dc(\SG)}}
\sum_{\bm\in\SG^\dagger} e^{i\bm\cdot\bx}
a(\bm +\bk) \hat u(\bm).
\end{equation*}
If $a(\bxi)\to \infty$ as $|\bxi|\to \infty$, then the spectrum of
each $A(\bk)$ is purely discrete with eigenvalues given by
$\l^{(\bm)}(\bk) = a(\bm+\bk), \bm\in\SG^\dagger$. Consequently,
the number of eigenvalues below each $\l \in \R$ is essentially
bounded from above uniformly in $\bk\in\CO^{\dagger}$. If $T$ is a
periodic symmetric operator which is $A$-bounded with a bound
$\e<1$, then the spectrum of $A(\bk) + T(\bk)$ is also purely
discrete and the counting function is also bounded uniformly in
$\bk$.
In particular, the above applies to the elliptic operator $H$ defined in
\eqref{h:eq}.
In fact, applying the Gelfand transform
\eqref{gelfand:eq} to $\op(b)$, one finds that, similarly to $A$
considered above,
 the operator $H(\bk)$ is a PDO in  $\GH$ of the form
\begin{equation}\label{floquet:eq}
H(\bk) u(\bx) = \frac{1}{\sqrt{\dc(\SG)}}
\sum_{\bm\in\SG^\dagger} e^{i\bm\cdot\bx}
h(\bx, \bm + \bk) \hat u(\bm),\ \bk\in\R^d.
\end{equation}
The values $H(\bk)$ for $\bk\in\CO^\dagger$ determine $H(\bk)$
for all $\bk\in\R^d$ due to the following unitary equivalence:
\begin{equation*}
H(\bk + \bm) = e^{-i\bm\bx}H(\bk) e^{i\bm\bx},\ \bm\in\SG^\dagger.
\end{equation*}
This implies, in particular, that
\begin{equation}\label{periodicity:eq}
\l_j (H(\bk+\bm)) = \l_j(H(\bk)),\ j = 1, 2, \dots,
\end{equation}
for all $\bm\in\SG^\dagger$.
The images
\begin{equation*}
\s_j = \bigcup_{\bk\in\overline{\CO^\dagger}} \l_j(H(\bk)),
\end{equation*}
are called \textsl{spectral bands of} $H$.
The spectrum of $H$ is the union
\begin{equation*}
\s(H) = \bigcup_{j} \s_j.
\end{equation*}
Due to the mentioned boundedness of the counting function $N(\l, H(\bk))$,
each interval $(-\infty, \l]$ has non-empty intersection with finitely many spectral
bands.
When proving the Bethe-Sommerfeld conjecture we study the band overlap, which is
characterized by the  \textsl{the overlap function}
$\z(\l)$, $\l\in\R$, defined as the maximal number $t$ such that the symmetric
interval $[\l - t, \l+t]$ is entirely contained in one band, i.e.
\begin{equation}\label{zeta:eq}
\z(\l; H) =
\begin{cases}
\max_j \max\{t: [\l - t, \l+t]\subset \s_j\}, \ \l\in\s(H);\\
0,\ \l\notin\s(H).
\end{cases}
\end{equation}
It is easy to see that $\z$ is continuous in $\l$.
An equivalent definition of $\z(\l)$ is
\begin{equation}\label{zeta1:eq}
\z(\l; H) = \sup\{t: \min_{\bk} N(\l+t, H(\bk)) < \max_{\bk} N(\l-t, H(\bk)\}.
\end{equation}
The function $\z(\l; H)$ was first introduced by M. Skriganov, see e.g. \cite{Skr1}.

The main result of the paper is 
the following Theorem:

\begin{thm} \label{main:thm}
Let $H = H_0 + \op(b)$ where $H_0 = (-\Delta)^{m}$ with some $m >0$, and
$b\in\BS_{\a}(w)$, $w = \lu\bxi\ru^{\b}$, with some $\a\in\R$ and
$\b\in (0, 1)$ satisfying the condition
\begin{equation}\label{alm1:eq}
2m-2 > \b(\a-2).
\end{equation}
Then the spectrum of the operator $H$ contains a half-line, i.e.
there exists a number $\l_0\in\R$ such that $[\l_0, \infty)\subset \s(H)$.
Moreover,
there is a number $S \in\R$ and a constant $c>0$ such that for each $\l\ge \l_0$
we have $\z(\l; H)\ge c\l^S$. The constant $c$ and parameter $\l_0$ are uniform
in $b$ satisfying $\1 b\1^{(\a)} \le C$.
\end{thm}

If one prefers stating the conditions on $b$ in terms
of the ``standard"
classes $\BS_a(\lu\bxi\ru)$, one can re-write  Theorem \ref{main:thm} as follows:

\begin{thm} \label{main1:thm}
Let $H = H_0 + \op(b)$ where $H_0 = (-\Delta)^{m}$ with some $m >0$, and
$b\in\BS_{a}(w)$, $w = \lu\bxi\ru$, with some $ a < 2m$.
Then the spectrum of the operator $H$ contains a half-line, i.e.
there exists a number $\l_0\in\R$ such that $[\l_0, \infty)\subset \s(H)$.
Moreover there is a number $S \in\R$ and a constant $c>0$ such that for each $\l\ge \l_0$
we have $\z(\l; H)\ge c\l^S$. The constant $c$ and parameter $\l_0$ are uniform
in $b$ satisfying $\1 b\1^{(\a)} \le C$.
\end{thm}

To deduce Theorem \ref{main1:thm} from \ref{main:thm} it suffices to note that
 $S_{a}(\lu\bxi\ru)\subset S_\a(\lu\bxi\ru^\b)$ for any $\b\in (0, 1)$
and $\a = a \b^{-1}$, and that  for this $\a$ the condition \eqref{alm1:eq}
is equivalent to
\begin{equation}\label{alm2:eq}
\b > \frac{a}{2} - m+1.
\end{equation}

\begin{rem}
The magnetic Schr\"odinger operator $H = (-i\nabla - \ba)^2 + V$
with a smooth $\SG$-periodic vector-potential $\ba: \R^d\to\R^d$ and
electric potential $V:\R^d\to\R$,
is a special case of the operator \eqref{h:eq}
with $h_0(\bxi) = |\bxi|^2$ and
$b(\bx, \bxi) = -2\ba(\bx)\cdot\bxi + i (\nabla\ba(\bx)) + \ba^2(\bx) + V(\bx)$.
Thus defined symbols satisfy the conditions of Theorem \ref{main1:thm} with $m=1$
and $a = 1$. In this case any $\b > 1/2$ satisfies \eqref{alm2:eq}.
\end{rem}

\begin{rem}
In \cite{BarPar}
the Bethe-Sommerfeld conjecture was proved for symmetric symbols $b$
satisfying the conditions
$\1 b\1^{(a)}_{l, 1} < \infty$ (here $w(\bxi) = \lu \bxi\ru$)
for some $a < 2m-1$, and all $l\ge 1$.
Although the restriction on the order of $b$ is stronger than in Theorem \ref{main1:thm},
the paper \cite{BarPar} does not impose any
conditions on derivatives w.r.t. $\bxi$ of order higher than one.
In general, an interesting question is to find out how the smoothness of the perturbation
in $\bxi$ affects the band overlap. We hope to address this issue in a further publication.
\end{rem}

We conclude the Introduction by fixing some notations which will be used
throughout the paper.

\subsection{Some notational conventions}\label{fibre:subsect}
For any measurable set $\CC\subset\R^d$ we denote by $\CP(\CC)$ the operator
$\op(\chi( \ \cdot\ ; \CC))$, where $\chi(\ \cdot\ ; \CC)$ is the characteristic function
of the set $\CC$. We denote $\CH(\CC) = \CP(\CC)\CH$, $\CH = \plainL2(\R^d)$.
Accordingly, the fibres $\CP(\bk, \CC), \bk\in\CO^\dagger$,
of $\CP(\CC)$, which act in $\GH$, are PDO's with symbols 
$\sum_{\bm\in\SG^\dagger}\chi(\bm +\bk; \CC)$.
In other words, each $\CP(\bk; \CC)$ is a projection
in $\GH$ on the linear span of the exponentials 
\begin{equation}\label{exp:eq}
E_\bm(\bx) := \frac{1}{\sqrt{\dc(\SG)}}
e^{i\bm\cdot \bx},\ \ \bm\in\SG^\dagger: \bm+\bk\in\CC.
\end{equation}
The subspace $\CP(\bk; \CC)\GH$ of $\GH$ is denoted by $\GH(\bk; \CC)$.

Suppose that  
$\CH(\CC)$ is an invariant subspace of the operator
$H$ defined in \eqref{h:eq},  
that is  $(H-iI)^{-1}\CH(\CC)\subset \CH(\CC)$.
Then the subspace $\GH(\bk; \CC)$,
$\bk\in\CO^\dagger$, is invariant for $H(\bk)$.
We denote by $H(\bk; \CC)$ the part of $H(\bk)$ in $\GH(\bk; \CC)$, so that
\begin{equation*}
H(\bk) = H(\bk; \CC)\oplus H(\bk; \R^d\setminus\CC),\ \bk\in\CO^\dagger,
\end{equation*}
where $\oplus$ denotes the orthogonal sum. If $\CH(\CC)$ is invariant for $H$,
then we denote by $N(\l, H(\bk); \CC)$ the counting function of $H(\bk; \CC)$
on the subspace $\GH(\bk; \CC)$.

Each $\bxi\in\R^d$ can be uniquely represented as the
sum $\bxi = \bm+\bk$, where $\bm\in\SG^\dagger$ and $\bk\in\CO^\dagger$.
We say that $\bm=:[\bxi]$ is the integer part of $\bxi$ and
$\bk=:\{\bxi\}$ is the fractional part of $\bxi$.
 
The notation $B(\bx_0, R)$ is used for the open
ball in $\R^d$ of radius $R>0$, centered  at $\bx_0\in\R^d$. We also write
$B(R)$ for the open ball of radius $R$ centered at $0$.

For the reference convenience we copy here the conventions about symbol classes made
earlier in this section.
In the situations when it  is not important for us to know the exact
values of $l, s$ in the norm $\1 b\1^{(\a)}_{l, s}$, we denote the above norm
by $\1 b\1^{(\a)}$. In this case
the inequality $A \le C\1 b\1^{(\a)}$ means that there exist values of $l$ and $s$,
and a constant $C>0$, possibly depending on $l, s$,
such that $A \le C\1 b\1^{(\a)}_{l, s}$. Similarly, when we write
$\1 b\1^{(\g)} \le C\1 g\1^{(\a)}$ for some symbols $b\in\BS_\g, g\in\BS_\a$, we mean that
for any $l$ and  $s$ the norm $\1 b\1^{(\g)}_{l, s}$
is bounded by $\1 g\1^{(\a)}_{p, n}$ with some $p$ and $n$ depending on $l, s$, and
some constant $C = C_{l, s}$. In general,
by $C, c$(with or without indices) we denote various positive constants, whose precise value
is unimportant.
Throughout
the entire paper we adopt the following convention. An
estimate (or an assertion) is said to be uniform in a symbol
$b\in\BS_\a$ if the constants in the estimate (or
assertion) at hand depend only on the constants $C_{l, s}$
in the bounds $\1 b \1^{(\a)}_{l, s}\le C_{l, s}$.
This is sometimes expressed by saying that an estimate (or assertion) is
uniform in the symbol $b$ satisfying $\1 b\1^{(\a)} \le C$.

We sometimes use notation $f\ll g$
or $g\gg f$ for two positive functions $f, g$,
if there is a constant $C>0$ independent of $f, g$ such that $f\le C g$.
If $f\ll g$ and $g\ll f$, then we write $f\asymp g$.


\section{Properties of periodic PDO's}\label{calc:sect}

In this section we collect various properties of periodic PDO's
to be used in what follows.

\subsection{Some basic results on the calculus of periodic PDO's}
We begin by listing some elementary results for periodic PDO's, some of which can be found in
\cite{Sob0}.

Recall that $\plainS(\Rd)$ is taken as a natural domain of
$\op(b)$. Unless otherwise stated, all the symbols are supposed to
belong to the class $\BS_\a = \BS_\a(w; \SG),\ \a\in\R,$ with an arbitrary
function $w$ satisfying \eqref{weight:eq}
 and a lattice $\SG$.
The functions $w$ and the lattice $\SG$ are
usually omitted from the notation.

\begin{prop}\label{bound:prop}(See e.g. \cite{Sob0})
Suppose that $\1 b\1^{(0)}_{l, 0}<\infty$ with some $l >d$. Then $B = \op(b)$ is
bounded in $\CH$ and $\|B\|\le C \1 b \1^{(0)}_{l, 0}$,
with a constant $C$ independent of $b$.
\end{prop}

Since $\op(b) u\in\plainS(\Rd)$ for any $b\in\BS_{\a}$ and
$u\in \plainS(\Rd)$,
the product $\op(b) \op(g)$, $b\in \BS_{\a}, g\in \BS_{\g}$,
is well defined on $\plainS(\Rd)$. A straightforward calculation gives
the following formula for the symbol
$b\circ g $ of the product $\op(b)\op(g)$:
\begin{equation*}
(b\circ g)(\bx, \bxi) = \frac{1}{ {\dc(\SG)}} \sum_{\bth, \bphi}
\hat b(\bth, \bxi +\bphi) \hat g(\bphi, \bxi)
e^{i(\bth+\bphi)\bx},
\end{equation*}
and hence
\begin{equation}\label{prodsymb:eq}
\widehat{(b\circ g)}(\bchi, \bxi) = \frac{1}{ \sqrt{\dc(\SG)}}
\sum_{\bth +\bphi = \bchi} \hat b (\bth, \bxi +\bphi) \hat g(\bphi,
\bxi),\ \bchi\in\SG^\dagger,\ \bxi\in \Rd.
\end{equation}
Here and below $\bth, \bphi\in\SG^\dagger$. In particular, one sees that
$\op(b) \op(w^{\d}) = \op(b w^\d)$ for any $\d \in\R$.
This observation leads to the following Lemma.
We remind that the symbol $b_{\boldeta}$ is defined in \eqref{bboldeta:eq}.

\begin{lem}\label{formbound:lem}
Let $b\in\BS_{\a}(w)$ with $w(\bxi) = \lu\bxi\ru^\b, \b\in (0, 1]$.
Then for any $u\in\plainS(\R^d)$ and any $l > d$, we have
\begin{equation}\label{formbound1:eq}
\| \op(b) u\|\le C \1 b\1^{(\a)}_{l, 0}
\|(H_0 + I)^{\tilde\g} u\|,
\tilde\g  = \frac{\a\b}{2m},
\end{equation}
with a constant $C$ independent of $b, u$.
In particular, if $\a\b  < 2m$, then
$\op(b)$ is $H_0$-bounded with an arbitrarily small relative bound.

Moreover, for any $\boldeta\in\R^d$ and any $l > d$,
\begin{equation}\label{formdiffer:eq}
\|(\op(b) - \op(b_{\boldeta})) u\|
\le C |\boldeta| \1 b\1^{(\a)}_{l, 1}\|(H_0 + I)^{\g} u\|,\
\g=\frac{\b(\a-1)}{2m},
\end{equation}
where the constant $C$ does not depend on $b, u$, and
is uniform in $\boldeta$:
$|\boldeta|\le \tilde C$.
\end{lem}

\begin{proof}
Define $G =  B \op(w^{-\a })$.
As we have observed earlier, $G = \op(g)$ with $g = b w^{-\a}$, so that
$g\in\BS_0(w)$ and $\1 g\1^{(0)}_{l, 0}=\1 b\1^{(\a)}_{l, 0}$. Hence,
by Lemma \ref{bound:prop}, $\|G\|\le C\1 b\1^{(\a)}_{l, 0}$ and
\begin{equation}\label{form:eq}
\|\op(b) u\| = \|G \op(w^\a) u \| \le C
\1 b\1^{(\a)}_{l, 0} \| \op(w^{\a})u\|.
\end{equation}
As $\op(w^\a)\le C(H_0+I)^{\tilde\g}$, $\tilde\g = \a\b(2m)^{-1}$,
we get \eqref{formbound1:eq}.

The bound \eqref{formdiffer:eq} follows from \eqref{formbound1:eq} when applied to
the symbol $b-b_{\boldeta}$, and from the estimate \eqref{differ1:eq}.
\end{proof}

The bound \eqref{formbound1:eq} allows one to give a proper meaning
to the operator \eqref{h:eq}, since $b$ is infinitesimally $H_0$-bounded.
The bound \eqref{formdiffer:eq} will be useful in the study of the Floquet eigenvalues
as functions of the quasi-momentum $\bk$.

For general symbols $b, g$ we have the following proposition (see e.g. \cite{Sob0}).

\begin{prop}\label{product:prop}
Let $b\in\BS_{\a}$,\ $g\in\BS_{\g}$. Then
$b\circ g\in\BS_{\a+\g}$ and
\begin{equation*}
\1 b\circ g\1^{(\a+\g)}
\le C \1 b\1^{(\a)} \1 g\1^{(\g)},
\end{equation*}
with a constant $C$ independent of $b, g$.
\end{prop}

We are also interested in the estimates for symbols of
commutators. For PDO's $A, \Psi_l, \ l = 1, 2, \dots ,N$,
denote
\begin{gather*}
\ad(A; \Psi_1, \Psi_2, \dots, \Psi_N)
= i\bigl[\ad(A; \Psi_1, \Psi_2, \dots, \Psi_{N-1}), \Psi_N\bigr],\\
\ad(A; \Psi) = i[A, \Psi],\ \ \ad^N(A; \Psi) =
\ad(A; \Psi, \Psi, \dots, \Psi),\ \ad^0(A; \Psi) = A.
\end{gather*}
For the sake of convenience we use the notation
$\ad(a;  \psi_1, \psi_2, \dots, \psi_N)$
and $\ad^N(a, \psi)$ for the symbols of multiple commutators.
It follows from \eqref{prodsymb:eq} that the
Fourier coefficients of the
symbol $\ad(b,g)$ are given by
\begin{multline}\label{comm:eq}
\widehat{\ad(b, g)}(\bchi, \bxi) = \frac{i}{\sqrt{\dc(\SG)}}
\sum_{\bth +\bphi = \bchi} \bigl[\hat b(\bth, \bxi +\bphi) \hat
g(\bphi, \bxi) - \hat b(\bth, \bxi)
\hat g(\bphi, \bxi + \bth)\bigr],\\
\bchi\in\SG^\dagger,\ \bxi\in \Rd.
\end{multline}

\begin{prop}\label{commut0:prop}(See e.g. \cite{Sob0})
Let $b\in \BS_{\a}$ and $g_j\in\BS_{\g_j}$,\
$j = 1, 2, \dots, N$.
Then $\ad(b; g_1, \dots, g_N) \in\BS_{\g}$ with
$$
\g = \a+\sum_{j=1}^N(\g_j-1),
$$
and
\begin{equation}\label{commutator:eq}
\1 \ad(b; g_1, \dots, g_N)\1^{(\g)}
\le C \1 b\1^{(\a)}
\prod_{j=1}^N \1 g_j\1^{(\g_j)},
\end{equation}
with a constant $C$ independent of $b, g_j$.
\end{prop}

\subsection{Partition of the perturbation}
>From now on the weights in the definition
of classes $\BS_{\a} = \BS_{\a}(w)$
are assumed to be $w(\bxi) = \lu\bxi\ru^\b$
with some $\b\in (0, 1]$. Here we partition every symbol $b\in\BS_\a$ into the sum
of several symbols, restricted to different parts of the phase space.
These symbols depend on the parameter
$\rho\ge 1$, but this dependence is usually omitted from the notation.
Later on, we will put $\rho= \l^{\frac{1}{2m}}$.

Let $\iota\in \plainC\infty(\R)$ be a non-negative function such that
\begin{equation}\label{eta:eq}
0\le\iota\le 1,\ \ \iota(z) =
\begin{cases}
& 1,\  z \le 1/4;\\
& 0,\  z \ge 1/2.
\end{cases}
\end{equation}
For $L \ge 1$ and $\bth\in \SG^\dagger, \bth\not = \mathbf 0$, define the following
$\plainC\infty$-cut-off functions:
\begin{equation}\label{el:eq}
\begin{cases}
e_{\bth}(\bxi) =&\ \iota\biggl({\biggl|\dfrac{|\bxi +\bth/2|}{\rho}-1\biggr|}
 \biggr),\\[0.5cm]
\ell^{>}_{\bth}(\bxi) = &\ 1 -
\iota\biggl(\dfrac{|\bxi + \bth/2|}{\rho}-1
\biggr),\\[0.5cm]
\ell^{<}_{\bth}(\bxi) = &\ 1 -
\iota\biggl(1 - \dfrac{|\bxi + \bth/2|}{\rho}\biggr),
\end{cases}
\end{equation}
and
\begin{equation}\label{phizeta:eq}
\begin{cases}
\z_{\t}(\bxi; L) =&\
\iota\biggl(\dfrac{|\bth(\bxi + \bth/2)|}
{L|\bth|}\biggr),\\[0.5cm]
\varphi_{\t}(\bxi; L)
= &\  1 - \z_{\bth}(\bxi; L).
\end{cases}
\end{equation}
Note that $e_{\bth}+\ell^{>}_{\bth} + \ell^{<}_{\bth} = 1$.
The function $\ell^{>}_{\bth}$ is supported on the
set $|\bxi+\bth/2|>5\rho/4$, and
$\ell^{<}_{\bth}$ is supported on the
set $|\bxi+\bth/2|< 3\rho/4$. The function $e_{\bth}$ is supported in the
shell $\rho/2\le |\bxi|\le 3\rho/2$.
Omitting the parameter $L$ and using the notation
$\ell_{\bth}$ for any of the functions $\ell^{>}_{\bth}$ or $\ell^{<}_{\bth}$,
we point out that
\begin{equation}\label{symmetry:eq}
\begin{cases}
e_{\bth}(\bxi)
= e_{-\bth}(\bxi + \bth), \ \ell_{\bth}(\bxi)
= \ell_{-\bth}(\bxi + \bth),\\[0.2cm]
\varphi_{\bth}(\bxi)
=  \varphi_{-\bth}(\bxi + \bth),\ \
\z_{\bth}(\bxi) = \z_{-\bth}(\bxi + \bth).
\end{cases}
\end{equation}
Note that the above functions satisfy the estimates
\begin{equation}\label{varphi:eq}
\begin{cases}
|\BD^{\bs}_{\bxi} e_{\bth}(\bxi)|
+ |\BD^{\bs}_{\bxi}\ell_{\bth}(\bxi)|\ll \rho^{-|\mathbf s|},\\[0.2cm]
|\BD^{\bs}_{\bxi}\varphi_{\bth}(\bxi; L)|
+ |\BD^{\bs}_{\bxi} \z_{\bth}(\bxi; L)|
\ll L^{-|\bs|}.
\end{cases}
\end{equation}
Let
\begin{equation}\label{T:eq}
\T_r = \T_r(\SG) = \{\bth\in\SG^\dagger: 0<|\bth|\le r\},\
\T_r^{0} = \T_r\cup\{\mathbf 0\},
\end{equation}
with some $r >0$.
We always assume that $1\le r \le \rho^{\varkappa}$, where $\varkappa <\b$ is
a fixed (small) positive number the precise value of which will be chosen later.
Using the above cut-off functions, for any symbol
$b\in\BS_{\a }(w)$ we introduce six new symbols
$b^{\uparrow}, b^{\downarrow}, b^o, b^{\sharp}, b^{\natural}, b^{\flat}$
in the following way:
\begin{gather}
b^{\uparrow}(\bx, \bxi; \rho) = \frac{1}{\sqrt{\dc(\SG)}}
\sum_{\bth\notin\T_r^0}
\hat b(\bth, \bxi)
e^{i\bth \bx},\label{uparrow:eq}\\
b^{\sharp}(\bx, \bxi; \rho)
=  \frac{1}{\sqrt{\dc(\SG)}}\sum_{\bth\in\T_r}
\hat b(\bth, \bxi)
\ell^{>}_{\bth}(\bxi)
 e^{i\bth \bx},\label{sharp:eq}\\
b^{\natural}(\bx, \bxi; \rho) = \frac{1}{\sqrt{\dc(\SG)}}
\sum_{\bth\in\T_r}
\hat b(\bth, \bxi) \varphi_{\bth}(\bxi; \rho^\b)
e_{\bth}(\bxi)  e^{i\bth \bx},
\label{natural:eq}\\
b^{\flat}(\bx, \bxi; \rho) = \frac{1}{\sqrt{\dc(\SG)}}
\sum_{\bth\in\T_r}
\hat b(\bth, \bxi)\z_{\bth}(\bxi; \rho^\b)
e_{\bth}(\bxi)
  e^{i\bth \bx},\label{flat:eq}\\
b^{\downarrow}(\bx, \bxi; \rho) = \frac{1}{\sqrt{\dc(\SG)}}
\sum_{\bth\in\T_r}
\hat b(\bth, \bxi) \ell^{<}_{\bth}(\bxi)
  e^{i\bth \bx},\label{downarrow:eq}\\
b^o(\bx, \bxi; \rho) = b^o(\bxi; \rho) =
\frac{1}{\sqrt{\dc(\SG)}}\hat b(0, \bxi).\label{o:eq}
\end{gather}
The superscripts here are chosen to mean
correspondingly: $\uparrow$ =`large Fourier' (coefficients),
$\sharp$ = `large energy', $\natural$ = `non-resonance',
$\flat$ = `resonance', $\downarrow$ =`small energy', $o$ =$0$-th Fourier coefficient.
Sometimes the dependence of the introduced symbols
on the parameter $\rho$ is omitted from the notation.
The corresponding operators are denoted by
\begin{equation*}
\bes
B^{\uparrow} &= \op(b^{\uparrow}),\
B^{\sharp} = \op(b^{\sharp}),\ B^{\natural} = \op(b^{\natural}),\\
B^{\flat} &= \op(b^{\flat}),\ B^{\downarrow} = \op(b^{\downarrow}),\ B^o = \op (b^o).
\end{split}
\end{equation*}
By definition \eqref{eta:eq},
\begin{equation*}
b = b^o + b^{\downarrow}+ b^{\flat} + b^{\natural} + b^{\sharp} + b^{\uparrow}.
\end{equation*}
The role of each of these operator is easy to explain. The symbol
$b^{\uparrow}$ contains only Fourier coefficients with $|\bth|>r$,
and the remaining symbols contain the Fourier coefficients
with $|\bth|\le r$.
Note that on the support of the functions
$\hat b^{\natural}(\bth, \ \cdot\ ; \rho)$ and
$\hat b^{\flat}(\bth, \ \cdot\ ; \rho)$ we have
\begin{equation}\label{ds:eq}
|\bth|\le \rho^\b,\
\frac{1}{2}\rho \le |\bxi+\bth/2|\le \frac{3}{2}\rho, \
\frac{1}{2}\rho - \frac{1}{2}\rho^\varkappa\le |\bxi|
\le \frac{3}{2}\rho + \frac{1}{2}\rho^\varkappa.
\end{equation}
On the support of $b^{\downarrow}(\bth, \ \cdot\ ; \rho)$ we have
\begin{equation}\label{supportell<:eq}
\biggl|\bxi+\frac{\bth}{2}\biggr|\le \frac{3}{4}\rho,\
|\bxi|\le \frac{3}{4}\rho + \frac{1}{2}\rho^{\varkappa}.
\end{equation}
On the support of
$b^{\sharp}(\bx, \ \cdot\ ; \rho)$
we have
\begin{equation}\label{supportell>:eq}
\biggl|\bxi+\frac{\bth}{2}\biggr|\ge \frac{5}{4}\rho,\
|\bxi|\ge \frac{5}{4}\rho - \frac{1}{2}\rho^{\varkappa}.
\end{equation}
The introduced symbols play a central role in the proof
of the Main Theorem \ref{main:thm}. As we show in the course of the proof,
due to \eqref{supportell<:eq} and \eqref{supportell>:eq} the symbols
$b^{\uparrow}$, $b^\downarrow$  and $b^\sharp$ make only a negligible contribution to
the spectrum of the operator \eqref{h:eq} near the point $\l = \rho^{2m}$.
The only significant
components of $b$ are the symbols $b^\natural, b^{\flat}$
and $b^o$. The symbol $b^\natural$ will be
transformed in the next Section into another symbol, independent of
$\bx$.

We will often combine $B^{\flat}$, $B^{\sharp}$ and $B^{\uparrow}, B^{\downarrow}$:
for instance
$B^{\flat, \sharp} = B^{\flat} + B^{\sharp}$,
$B^{\flat, \sharp, \uparrow} = B^{\flat, \sharp} + B^{\uparrow}$.
A similar convention applies to the symbols.
Under the condition
$b\in\BS_{\a}(w)$ the above symbols
belong to the same class $\BS_{\a}(w)$
and the following bounds hold:
\begin{equation}\label{subord:eq}
\1 b^{\flat}\1^{(\a)}_{l, s}
+ \1 b^{\natural}\1^{(\a)}_{l, s} +
\1 b^{\sharp}\1^{(\a)}_{l, s}
+ \1 b^{o}\1^{(\a)}_{l, s}
+ \1 b^{\downarrow}\1^{(\a)}_{l, s}
+ \1 b^{\uparrow}\1^{(\a)}_{l, s}
\ll \1 b\1^{(\a)}_{l, s}.
\end{equation}
Indeed, let us check this  for the symbol $b^{\natural}$, for instance.
According to \eqref{ds:eq}
and \eqref{varphi:eq}, on the support of
the function $\hat b^{\natural}(\bth, \ \cdot\ ; \rho)$ we have
\begin{gather*}
|\BD^{\bs}\varphi_{\bth}(\bxi, \rho^\b)|\ll \rho^{-\b|\bs|}
\ll  w^{-|\bs|},\\[0.2cm]
|\BD^{\bs}\ell^{>}_{\bth}(\bxi)|
+ |\BD^{\bs}\ell^{<}_{\bth}(\bxi)|
+ |\BD^{\bs}e_{\bth}(\bxi)|\ll  \rho^{-|\bs|}\ll w^{-|\bs|}.
\end{gather*}
This immediately leads to the bound of the form \eqref{subord:eq}
for the symbol $b^{\natural}$.

The introduced operations also
preserve symmetry. Precisely, calculate using
\eqref{symmetry:eq}:
\begin{align*}
\overline{\hat b^{\flat}(-\bth, \bxi + \bth)}
= & \ \overline{\hat b(-\bth, \bxi + \bth)}
\z_{-\bth}(\bxi+\bth; \rho^\b)
e_{-\bth}(\bxi+\bth)\\
= &\ \hat b(\bth, \bxi) \z_{\bth}(\bxi; \rho^\b)
e_{\bth}(\bxi)
= \hat b^{\flat}(\bth, \bxi).
\end{align*}
Therefore, by \eqref{selfadj:eq} the operator $B^{\flat}$ is
symmetric if so is $B$. The proof is similar for the rest of the
operators introduced above.

Let us list some other elementary properties
of the introduced operators. In the Lemma below we
use the projection $\CP(\CC), \CC\subset\R$ whose definition
was given in Subsection \ref{fibre:subsect}.

\begin{lem}\label{smallorthog:lem}
Let $b\in \BS_{\a}(w)$,
$w = \lu\bxi\ru^\b$, $\b\in (0, 1]$, with some
$\a\in\R$.
Then the following hold:
\begin{itemize}
\item[(i)]
The operator $\op(b^{\downarrow})$ is bounded and
\begin{equation*}
\|\op(b^{\downarrow})\|\ll \1 b \1^{(\a)}_{l, 0}
\rho^{\b\max(\a, 0)}.
\end{equation*}
Moreover,
\begin{equation*}
\bigl(I - \CP (B(7\rho/8) )\bigr) \op(b^{\downarrow})
= \op(b^{\downarrow}) \bigl(I - \CP (B(7\rho/8))\bigr)  = 0.
\end{equation*}

\item[(ii)] The operator $B^\flat$ satisfies the following relations
\begin{align}\label{bflatorthog:eq}
\CP(B(3\rho/8)) B^{\flat} =  &\ B^{\flat} \CP(B(3\rho/8)) \notag\\[0.2cm]
= &\ \bigl(I - \CP(B(11\rho/8))\bigr) B^{\flat}
=  B^{\flat} \bigl(I - \CP(B(11\rho/8))\bigr) = 0,
\end{align}
and similar relations hold for the operator $B^{\natural}$ as well.

Moreover, for any $\g \in\R$ one has
$b^{\natural}, b^{\flat} \in\BS_{\g}$ and
\begin{equation}\label{nat:eq}
\1 b^{\natural}\1^{(\g)}_{l, s} + \1 b^{\flat}\1^{(\g)}_{l, s}
\ll \rho^{\b(\a - \g)}\1 b\1^{(\a)}_{l, s},
\end{equation}
for all $l$ and $s$,
with an implied constant independent of $b$ and $\rho\ge 1$.
In particular, the operators $B^{\natural}, B^{\flat}$ are bounded and
\begin{equation*}
\| B^{\natural} \| + \|B^{\flat}\| \ll \rho^{\b \a } \1 b \1^{(\a)}_{l, 0},
\end{equation*}
for any $l >d$.

\item[(iii)]
\begin{equation*}
\CP\bigl(B(9\rho/8)\bigr)B^{\sharp} = B^{\sharp}\CP\bigl(B(9\rho/8)\bigr) = 0.
\end{equation*}

\item[(iv)] If $R\le 2\rho$, then
\begin{equation}\label{shar:eq}
\| \CP(B(R))B^{\uparrow}\| + \| B^{\uparrow} \CP(B(R))\|
\ll \1 b \1^{(\a)}_{l, 0}r^{p-l} \rho^{\b \max(\a, 0) },
\end{equation}
for any  $p>d$ and any $l \ge p$.
\end{itemize}
\end{lem}

\begin{proof}
\underline{Proof of (i).}
By  \eqref{decay:eq},
\begin{equation*}
|\hat b(\bth, \bxi; \rho)|\le
\1 b \1^{(\a)}_{l, 0}  \lu \bth\ru^{-l} \lu\bxi\ru^{\b\a},
\end{equation*}
for any $l >0$. It follows from \eqref{supportell<:eq}  that
\begin{equation}\label{flatdecay:eq}
|\hat b^{\downarrow} (\bth, \bxi; \rho)|
\ll \1 b \1^{(\a)}_{l, 0} \rho^{\b\max(\a, 0)}
\lu \bth\ru^{-l},\ \forall l >0.
\end{equation}
By Proposition \ref{bound:prop} this implies the
sought bound for the norm $\|\op(b^{\downarrow})\|$.

In view of \eqref{supportell<:eq},
the second part of statement (i) follows from
\eqref{downarrow:eq} by inspection.

\underline{Proof of (ii).}
 The relations \eqref{bflatorthog:eq} follow from the definitions
\eqref{flat:eq} and \eqref{natural:eq} in view of \eqref{ds:eq}.

Furthermore, by  \eqref{decay:eq} and \eqref{subord:eq},
\begin{equation*}
|\BD^{\bs}_{\bxi} \hat b^{\natural}(\bth, \bxi; \rho)|
+ |\BD^{\bs}_{\bxi} \hat b^{\flat}(\bth, \bxi; \rho)|
\ll \1 b \1^{(\a)}_{l, s}\lu \bxi\ru^{\b(\a - \g)}  \lu\bxi\ru^{ \b(\g - s)}
\lu\bth\ru^{-l}.
\end{equation*}
Thus, using again \eqref{ds:eq}, we obtain:
\begin{equation*}
|\BD^{\bs}_{\bxi} \hat b^{\natural}(\bth, \bxi; \rho)|
+ |\BD^{\bs}_{\bxi} \hat b^{\flat}(\bth, \bxi; \rho)|
\ll \rho^{\b(\a-\g)}
\1 b \1^{(\a)}_{l, s} w^{ \g-s} \lu\bth\ru^{-l}.
\end{equation*}
This means that $b^{\natural}, b^{\flat}\in\BS_\g$ for any $\g\in\R$  and
\eqref{nat:eq} holds. The bounds for the norms follow
from \eqref{nat:eq} with $\g = 0$, and Proposition \ref{bound:prop}.

\underline{Proof of (iii)} is similar to (i). The required result follows from
\eqref{supportell>:eq}.

\underline{Proof of (iv).} By definition \eqref{T:eq}
the sum \eqref{uparrow:eq} contains only those values of $\bth$
for which  $|\bth|\ge r $.
Thus, in view of \eqref{decay:eq} and \eqref{subord:eq},
for any $l \ge p$ we have
\begin{equation*}
| \hat b^{\uparrow}(\bth, \bxi; \rho)|
\le
\1 b \1^{(\a )}_{l, 0}\  w^{\a}
\lu\bth\ru^{-l}
\ll r^{p - l}
\1 b \1^{(\a)}_{l, 0}\ w^{\a}
\lu\bth\ru^{-p},
\end{equation*}
Thus the symbol of $ B^{\uparrow} \CP\bigl(B(R)\bigr)$ is bounded by
\[
  C r^{p - l}
\1 b \1^{(\a)}_{l, 0}\ R^{\b\max(\a, 0)}
\lu\bth\ru^{-p},
\]
so that the sought estimate follows by Proposition \ref{bound:prop}.
The same argument leads to
the same bound for $\CP\bigl(B(R)\bigr) B^{\uparrow}$.
\end{proof}

In what follows a central role is played by the operator of the form
\bee\label{newA}
A := H_0 + B^o + B^{\flat}
\ene
with some symmetric symbol $b\in\BS_\a$.
In the next Lemma we study the continuity
of the Floquet eigenvalues $\l_j(A(\bk))$, $j=1, 2, \dots$,
as functions of the quasi-momentum
$\bk\in\R$. Here, $A(\bk)$ are the fibers of the operator \eqref{newA}.
To state the result, we introduce
for any vector $\boldeta\in\R^d$ the \textit{distance on the torus:}
\begin{equation}\label{distorus:eq}
|\boldeta|_{\Torus} = \min_{\bm\in\SG^\dagger}|\boldeta-\bm|.
\end{equation}

\begin{thm}\label{unperturbed:thm}
Suppose that $\rho\ge 1$ and
\begin{equation}\label{beta:eq}
\b(\a-1)< 2m-1.
\end{equation}
If for some $j$
\bee\label{newA1}
\l_j(A(\bk))\asymp \rho^{2m},
\ene
then
for any $l > d$ we have
\begin{equation}\label{unperturbed:eq}
|\l_j(A(\bk+\boldeta)) - \l_j(A(\bk))|
\ll\bigl(1+\1 b\1^{(\a)}_{l, 1}\bigr)|\boldeta|_{\Torus} \rho^{2m-1}.
\end{equation}
The implied constant in \eqref{unperturbed:eq} depends
on the constants in \eqref{newA1}.
\end{thm}

\begin{proof}
By Lemma \ref{smallorthog:lem} (ii),
\[
\CP\bigl(\bk, B(3\rho/8)\bigr) B^{\flat}(\bk)
=  B^{\flat}(\bk) \CP\bigl(\bk, B(3\rho/8)\bigr) = 0,
\]
so that
\[
\CP\bigl(\mathbf 0, B(3\rho/8-|\bk|)\bigr) B^{\flat}(\bk)
=  B^{\flat}(\bk) \CP\bigl(\mathbf 0, B(3\rho/8-|\bk|)\bigr) = 0.
\]
Similarly, \eqref{bflatorthog:eq} implies that
\[
\biggl(I - \CP\bigl(\mathbf 0, B(11\rho/8+|\bk|)\bigr)\biggr) B^{\flat}(\bk)
=  B^{\flat}(\bk) \biggl(I - \CP\bigl(\mathbf 0, B(11\rho/8+|\bk|)\bigr)\biggr)  = 0.
\]
Thus, if one assumes that $|\bk|\le R$ with some $R$, then the operator
$A(\bk)$ can be represented in the form
\begin{equation}\label{orthog:eq}
A(\bk) = A_-(\bk) \oplus A_c(\bk) \oplus A_+(\bk),
\end{equation}
with
\begin{equation*}
A_\pm(\bk) = \CP_\pm (H_0(\bk)+B^o(\bk))\CP_\pm,\
A_c(\bk) = \CP_c A(\bk) \CP_c,
\end{equation*}
where by $\CP_\pm$ and $\CP_c$ we have denoted the following
projections in $\GH$:
\begin{align*}
\CP_-= &\ \CP\bigl(\mathbf 0, B(3\rho/8- R )\bigr),\\[0.2cm]
\CP_c =  &\ \CP\bigl(\mathbf 0, B(11\rho/8+R)\bigr)
- \CP\bigl(\mathbf 0, B(3\rho/8- R )\bigr),\\[0.2cm]
\CP_+ = &\ I - \CP\bigl(\mathbf 0, B(11\rho/8+ R)\bigr).
\end{align*}
Thanks to \eqref{orthog:eq}, in order to establish \eqref{unperturbed:eq}
it suffices to prove this inequality for eigenvalues
(labeled in the standard ascending order)
of each of the operators $A_-, A_c, A_+$.

Suppose first that $|\boldeta|_{\Torus} = |\boldeta|$, so that
\begin{equation*}
|\boldeta|\le R:=\min_{\mathbf 0\not=\bm\in\SG^\dagger} |\bm|.
\end{equation*}
The operator
$A_0(\bk) = H_0(\bk) + B^o(\bk)$ has constant coefficients and its
eigenvalues are found explicitly:
\begin{equation*}
\mu_{\bm}(\bk) = (\bm+\bk)^{2m} + b^o(\bm+\bk), \bm\in\SG^\dagger.
\end{equation*}
Assuming that $|\bm+\bk|\ge 2R$, from \eqref{differ:eq}
and \eqref{subord:eq} we obtain that
\begin{equation*}
|\mu_{\bm}(\bk+\boldeta) - \mu_{\bm}(\bk)|
\ll \bigl(|\bm+\bk|^{2m-1}  + \1 b\1^{(\a)}_{0, 1}
|\bm+\bk|^{\b(\a-1)}\bigr) |\boldeta|.
\end{equation*}
Due to conditions \eqref{beta:eq},
\begin{equation}\label{seig0:eq}
|\mu_{\bm}(\bk+\boldeta) - \mu_{\bm}(\bk)|
\ll \bigl(1 + \1 b\1^{(\a)}_{0, 1}
 \bigr)|\bm+\bk|^{2m-1}  |\boldeta|.
\end{equation}
If we assume that the considered eigenvalue $\mu_\bm(\bk)$ satisfies \eqref{newA1},
then $|\bm+\bk|\asymp \rho$, and hence the inequality \eqref{seig0:eq} implies that
\begin{equation}\label{seig1:eq}
|\mu_{\bm}(\bk+\boldeta) - \mu_{\bm}(\bk)|
\ll \bigl(1 + \1 b\1^{(\a)}_{0, 1}
 \bigr)\rho^{2m-1}  |\boldeta|.
\end{equation}
In order to rewrite this bound for the eigenvalues $\l_j(A_\pm(\bk))$ arranged in the
ascending order, note that for $\bm\in B(3\rho/8-R)$ or $\bm\notin B(11\rho/8+R)$,
the previous inequality is equivalent
to the following one:
\begin{equation*}
\begin{split}
&N(\la-C \bigl(1 + \1 b\1^{(\a)}_{0, 1}
 \bigr)\rho^{2m-1}  |\boldeta|; A_\pm(\bk))\le
 N(\la; A_\pm(\bk+\boldeta))\le\\
 &
 N(\la+C \bigl(1 + \1 b\1^{(\a)}_{0, 1}
 \bigr)\rho^{2m-1}  |\boldeta|; A_\pm(\bk)),
\end{split}
\end{equation*}
which, in turn means that
\begin{equation}\label{seig:eq}
|\l_j(A_\pm(\bk+\boldeta)) - \l_j(A_\pm(\bk))|
\ll \bigl(1 + \1 b\1^{(\a)}_{0, 1}
 \bigr)\rho^{2m-1}|\boldeta|
\end{equation}
for all eigenvalues satisfying
the condition $|\l_j(A_\pm(\bk))|\asymp\rho^{2m}$.

Let us study the eigenvalues of $A_c$.
This operator is bounded, and hence
it suffices to find an upper bound for the sum of the norms
\[
\|\CP_c \bigl(A_0(\bk+\boldeta) - A_0(\bk)\bigr)\CP_c\|
+ \| B^\flat(\bk+\boldeta) - B^\flat(\bk)\|.
\]
Using \eqref{seig1:eq} for $3\rho/8-R\le |\bm|\le 11\rho/8 + R$, we get
\begin{equation}\label{a0:eq}
\|\CP_c \bigl(A_0(\bk+\boldeta) - A_0(\bk)\bigr)\CP_c\|
\ll  \bigl(1 + \1 b\1^{(\a)}_{0, 1}\bigr)\rho^{2m-1}  |\boldeta|.
\end{equation}
(note that $\CP_c A_0(\bk+\boldeta)\CP_c$. 
By Lemma \ref{smallorthog:lem}(ii),
$b^\flat\in\BS_{\g}$ with any $\g\in\R$, so that \eqref{formdiffer:eq} and
\eqref{nat:eq} give:
\begin{equation*}
\| B^\flat(\bk+\boldeta) - B^\flat(\bk)\|\ll |\boldeta| \1 b^\flat\1^{(1)}_{l, 1}
\ll |\boldeta| \1 b\1^{(\a)}_{l, 1} \rho^{\b(\a-1)},
\end{equation*}
for any $l >d$.
By \eqref{beta:eq} $\b(\a-1)< 2m-1$, so that the above estimate
in combination with \eqref{a0:eq}, proves that
\[
|\l_j(A_c(\bk+\boldeta) - \l_j(A_c(\bk))|
\ll |\boldeta| (1+\1 b\1^{(\a)}_{l, 1}) \rho^{2m-1},
\]
uniformly in $j$.
In its turn, this estimate together with \eqref{seig:eq} leads to \eqref{unperturbed:eq}.

For a general $\boldeta$, note that according to
\eqref{periodicity:eq},
$\l_j(A(\bk+\boldeta)) = \l_j(A(\bk+\bm+\boldeta))$ for any $\bm\in\SG^\dagger$.
Choose $\bm$ in such a way that
$$
|\boldeta|_{\Torus} = |\boldeta+\bm|.
$$
Denote $\boldeta_1 = \boldeta+\bm$ and use the first part of the proof
for $\boldeta_1$.
\end{proof}


\section{A ``gauge transformation''}\label{gauge:sect}

In this and all the subsequent sections we assume that
$\BS_\a = \BS_\a(w)$ with $w(\bxi) = \lu \bxi\ru^\b$,
$\b\in (0, 1]$.
Recall that we study spectral properties of the operator
$H$ defined in \eqref{h:eq}. Our ultimate goal is to prove
that each sufficiently large $\la$ belongs to the spectrum of $H$.
We are going to use the notation from the previous section with
the parameter $\rho=\la^{\frac{1}{2m}}\ge 1$.

\subsection{Preparation}
Our strategy is to find a unitary
operator which reduces $H = H_0+\op(b)$ to another PDO, whose
symbol, up to some controllable small errors,
depends only on $\bxi$.
The sought unitary operator is constructed in the
form $U = e^{i\Psi}$ with a suitable bounded
self-adjoint $\SG$-periodic PDO $\Psi$.
This is why we sometimes call it a ``gauge transformation''.
It is useful to consider $e^{i\Psi}$ as an element of the group
\begin{equation*}
U(t) = \exp\{ i \Psi t\},\ \ \forall
t\in\R.
\end{equation*}
\textbf{We assume that the operator
$\ad(H_0, \Psi)$ is bounded, so that
$U(t) D(H_0) = D(H_0)$}. This assumption will be justified
later on. Let us express the operator
\begin{equation*}
A_t := U(-t)H U(t)
\end{equation*}
via its (weak) derivative with respect to $t$:
\begin{equation*}
A_t = H +  \int_0^t U(-t') \ad(H; \Psi) U(t') dt'.
\end{equation*}
By induction it is easy to show that
\begin{gather}
A_1 = H + \sum_{j=1}^M \frac{1}{j!}
\ad^j(H; \Psi) +  R^{(1)}_{M+1},\label{decomp:eq}\\
R^{(1)}_{M+1} :=  \int_0^1 d t_1
\int_0^{t_1} d t_2\dots \int_0^{t_M}
U(-t_{M+1})
\ad^{M+1}(H; \Psi)
U(t_{M+1}) dt_{M+1}.\notag
\end{gather}
The operator $\Psi$ is sought in the form
\begin{equation}\label{psik:eq}
\Psi = \sum_{k=1}^M \Psi_k,\ \Psi_k = \op(\psi_k),
\end{equation}
with symbols $\psi_k$ from some suitable
classes $\BS_\s, \s = \s_k$ to be specified later on.
Substitute this formula in \eqref{decomp:eq}
and rewrite, regrouping the terms:
\begin{gather}
A_1 = H_0 + B + \sum_{j=1}^M \frac{1}{j!}
\sum_{l=j}^M \sum_{k_1+ k_2+\dots + k_j = l}
\ad(H; \Psi_{k_1}, \Psi_{k_2}, \dots, \Psi_{k_j})\notag\\
 +
 R^{(1)}_{M+1} + R^{(2)}_{M+1},\notag\\
R^{(2)}_{M+1}: =
\sum_{j=1}^M \frac{1}{j!}\sum_{k_1+ k_2+\dots + k_j \ge  M+1}
\ad(H; \Psi_{k_1}, \Psi_{k_2}, \dots, \Psi_{k_j}).
\label{rtilde:eq}
\end{gather}
Changing this expression yet again produces
\begin{gather*}
A_1 = H_0 + B + \sum_{l=1}^M \ad(H_0; \Psi_l)
+ \sum_{j=2}^M \frac{1}{j!}
\sum_{l=j}^M \sum_{k_1+ k_2+\dots + k_j = l}
\ad(H_0; \Psi_{k_1}, \Psi_{k_2}, \dots, \Psi_{k_j})\\
+ \sum_{j=1}^M \frac{1}{j!}
\sum_{l=j}^M \sum_{k_1+ k_2+\dots + k_j = l}
\ad(B; \Psi_{k_1}, \Psi_{k_2}, \dots, \Psi_{k_j}) +
R^{(1)}_{M+1} + R^{(2)}_{M+1}.
\end{gather*}
Next, we switch the summation signs
and decrease $l$ by one in the second summation:
\begin{gather*}
A_1 = H_0 + B + \sum_{l=1}^M \ad(H_0; \Psi_l)
+ \sum_{l=2}^M
\sum_{j=2}^l \frac{1}{j!}\sum_{k_1+ k_2+\dots + k_j = l}
\ad(H_0; \Psi_{k_1}, \Psi_{k_2}, \dots, \Psi_{k_j})\\
+\sum_{l=2}^{M+1}
\sum_{j=1}^{l-1} \frac{1}{j!}\sum_{k_1+ k_2+\dots + k_j = l-1}
\ad(B; \Psi_{k_1}, \Psi_{k_2}, \dots, \Psi_{k_j}) +
R^{(1)}_{M+1} + R^{(2)}_{M+1}.
\end{gather*}
Now we introduce the notation
\begin{gather}
B_1 := B,\notag\\
B_l := \sum_{j=1}^{l-1} \frac{1}{j!}
 \sum_{k_1+ k_2+\dots + k_j = l-1}
\ad(B; \Psi_{k_1}, \Psi_{k_2}, \dots, \Psi_{k_j}),
\ l\ge 2,\label{bl:eq}\\
T_l := \sum_{j=2}^l \frac{1}{j!}
\sum_{k_1+ k_2+\dots + k_j = l}
\ad(H_0; \Psi_{k_1}, \Psi_{k_2}, \dots, \Psi_{k_j}),\  l\ge 2.
\label{tl:eq}
\end{gather}
We emphasise that the operators $B_l$ and $T_l$ depend
only on $\Psi_1, \Psi_2, \dots, \Psi_{l-1}$.
Let us make one more rearrangement:
\begin{gather}
A_1 =  H_0 + B + \sum_{l=1}^M \ad(H_0, \Psi_l)
+ \sum_{l=2}^M B_l
+   \sum_{l=2}^{M} T_{l} + R_{M+1},\notag\\
R_{M+1} =  B_{M+1} +
R^{(1)}_{M+1} + R^{(2)}_{M+1}.\label{r:eq}
\end{gather}
Now we can specify our algorithm for finding $\Psi_k$'s.
The symbols $\psi_k$
will be found from the following system of commutator equations:
\begin{gather}
\ad(H_0; \Psi_1) + B_1^{\natural} = 0,\label{psi1:eq}\\
\ad(H_0; \Psi_l) + B_l^{\natural} + T_l^{\natural}
= 0,\ l\ge 2,\label{psil:eq}
\end{gather}
and hence
\begin{equation}\label{lm:eq}
\begin{cases}
A_1 = A_0 + X_M^{\flat}
+ X_{M}^{\downarrow, \sharp, \uparrow} + R_{M+1},\\[0.3cm]
X_M =  \sum_{l=1}^{M} B_l + \sum_{l=2}^{M}
T_l,\\[0.3cm]
A_0 = H_0
+ X^{(o)}_M.
\end{cases}
\end{equation}
Below we denote by $x_M$ the symbol of the PDO $X_M$.
Recall that by Lemma \ref{smallorthog:lem}(ii),
the operators $B_l^{\natural}, T_l^{\natural}$ are bounded,
and therefore, in view of \eqref{psi1:eq}, \eqref{psil:eq},
so is  the commutator $\ad(H_0; \Psi)$. This justifies
the assumption made in the beginning of the formal calculations
in this Section.

\subsection{Commutator equations}
Recall that $h_0(\bxi) = |\bxi|^{2m}$ with $m >0$. Before
proceeding to the study of the commutator equations
\eqref{psi1:eq}, \eqref{psil:eq} note that for
$\bxi$ in the support of the function
$\hat b^{\natural}(\bth, \ \cdot\ ; \rho)$
the symbol
\begin{equation}\label{taut:eq}
\tau_{\bth}(\bxi) = h_0(\bxi+\bth) - h_0(\bxi) = \bigl(|\bxi|^2+
2\bth\cdot(\bxi+\bth/2)\bigr)^m - |\bxi|^{2m}
\end{equation}
satisfies the bound
\begin{equation*}
\tau_{\bth}(\bxi)\asymp \rho^{2m-2}\
|\bth\cdot(\bxi+\bth/2)|,
\end{equation*}
which easily follows from \eqref{ds:eq}. Using
\eqref{ds:eq} again, we conclude that
\begin{equation*}
  |\bth| \rho^{2m-2 + \b}\ \ll \tau_{\bth}(\bxi)\ll
|\bth| \rho^{2m-1}.
\end{equation*}
Note also that
\begin{equation*}
|\BD_{\bxi}^{\bs} \tau_{\bth}(\bxi)|\ll |\bth| \rho^{2m-1-s}, \ |\bs| = s.
\end{equation*}
Therefore,
\begin{equation}\label{tau:eq}
|\BD_{\bxi}^{\bs} \tau_{\t}^{-1}| \ll |\bth|^{-1}
\rho^{-2m+2 -\b(1+s)}
\ll |\bth|^{-1} w^{-(2m-2)\b^{-1}-1-s},\ \bth\not=0,
\end{equation}
for all $\bxi$ in the support of the function
$\hat b^{\natural}(\bth, \ \cdot\ ; \rho)$.
This estimate will come in handy in the next lemma.

\begin{lem} \label{commut:lem}
Let $A = \op(a)$ be a symmetric PDO with $a\in\BS_{\om }$.
Then the PDO $\Psi$ with the Fourier coefficients of the symbol
$\psi(\bx, \bxi; \rho)$ given by
\begin{equation}\label{psihat:eq}
\begin{cases}
\hat\psi(\bth, \bxi; \rho) = i \dfrac{\hat a^{\natural}(\bth, \bxi; \rho)}
{\tau_{\bth}(\bxi)},\ \bth\not = 0,\\[0.3cm]
\hat\psi(\mathbf 0, \bxi; \rho) = 0,
\end{cases}
\end{equation}
solves the equation
\begin{equation}\label{adb:eq}
\ad(H_0; \Psi) + \op(a^{\natural})= 0.
\end{equation}
Moreover, the operator $\Psi$ is bounded and self-adjoint, its symbol $\psi$
belongs to $\BS_{ \g}$ with any $\g \in\R$
and the following bound holds:
\begin{equation}\label{psitau:eq}
\1 \psi\1^{(\g)}_{l, s}
\ll \rho^{\b(\s-\g)}\1 a\1^{(\om)}_{l-1, s},
\end{equation}
where
\begin{equation}\label{sigma:eq}
\s = \om - (2m - 2)\b^{-1}-1.
\end{equation}
\end{lem}

\begin{proof}
For brevity we omit $\rho$ from the notation.
Let $t$ be the symbol of $\ad(H_0; \Psi)$.
The Fourier transform $\hat t(\bth, \bxi)$
is easy to find using \eqref{prodsymb:eq}:
\begin{equation*}
\hat t(\bth, \bxi) =
i \bigl(h_0(\bxi+\bth) - h_0(\bxi)\bigr) \hat \psi(\bth, \bxi)
= i\tau_{\bth}(\bxi)\hat\psi(\bth, \bxi).
\end{equation*}
Therefore, by definition \eqref{natural:eq},
the equation \eqref{adb:eq} amounts to
\begin{equation*}
i \tau_{\bth}(\bxi)\hat \psi(\bth, \bxi)
= - \hat a^{\natural}(\bth, \bxi; \rho)
= - \hat a(\bth, \bxi; \rho)\varphi_{\bth}(\bxi; \rho^{\b})
e_{\bth}(\bxi), \ |\bth|\le r.
\end{equation*}
By  definition of the functions $\varphi_{\bth}, e_{\bth}$,
the function $\hat\psi$ given by \eqref{psihat:eq} is defined for all $\bxi$.
Moreover, the symbol $\hat\psi$ satisfies the condition
\eqref{selfadj:eq}, so that $\Psi$
is a symmetric operator.

In order to prove that $\psi\in \BS_\g$ for all $\g\in\R$,
note that according to \eqref{nat:eq} and \eqref{decay:eq},
\begin{equation*}
|\BD^{\bs}_{\bxi} \hat a^{\natural}(\bth, \bxi; \rho)|
\ll \rho^{\b(\om-\g)} \1 a\1^{(\om)}_{l, s} w^{\g - s} |\bth|^{-l}.
\end{equation*}
Together with \eqref{tau:eq} this implies that
\begin{equation*}
|\BD^{\bs}_{\bxi} \hat\psi(\bth, \bxi; \rho)|
\ll \rho^{-\b\g} \1 a\1^{(\om)}_{l, s} w^{\s+\g - s} |\bth|^{-l-1},
\end{equation*}
so that $\psi\in \BS_\g$ and it satisfies \eqref{psitau:eq}.

The estimate \eqref{psitau:eq} with $ \g = 0, s = 0$, and Proposition
\ref{bound:prop} ensure the boundedness of $\Psi$.
\end{proof}

Let us apply Lemma \ref{commut:lem} to equations
\eqref{psi1:eq} and \eqref{psil:eq}.

\begin{lem}\label{gauge:lem}
Let $b\in \BS_{\a}$ be a symmetric symbol, $\rho\ge 1$, and let
\begin{equation}\label{sigmaeps:eq}
\begin{cases}
\s = &\ \a - (2m-2)\b^{-1} - 1,\\
\s_j = &\ j(\s-1)+1,\\
\e_j = &\ j(\s-1)+(2m-2)\b^{-1}+2,
\end{cases}
\end{equation}
$j=1, 2, \dots$.
Then there exists a sequence of
self-adjoint bounded PDO's  $\Psi_j$, $j = 1, 2, \dots$
with the symbols
$\psi_j$ such that $\psi_j\in\BS_{ \g}$ for any $\g \in\R$,
\eqref{psi1:eq} and \eqref{psil:eq} hold,
and
\begin{equation}\label{psik1:eq}
\1 \psi_j\1^{( \g)}
\ll \rho^{\b(\s_j-\g)} \bigl(\1 b\1^{(\a)} )^j,\ j\ge 1.
\end{equation}
The symbols $b_j$, $t_j$ of the
corresponding operators $B_j$, $T_j$ belong to
$\BS_{\g}$ for any $\g\in\R$ and
\begin{equation}\label{bt:eq}
\1 b_j\1^{(\g)}
+ \1 t_j\1^{(\g)}
\ll \rho^{\b(\e_j - \g)}(\1 b\1^{(\a )} )^j,\  j\ge 2.
\end{equation}

If $\rho^{\b(\s-1)} \1 b\1^{(\a)}\ll 1$,
then for any $M$ and $\psi = \sum_{j=1}^M \psi_j$ the following bounds hold:
\begin{gather}
\1\psi \1^{(\g)} \ll \rho^{\b(\s-\g)}\1 b\1^{(\a)},\ \forall \g\in\R;\ \ \
\ \ \1 x_M\1^{(\a)} \ll\1 b\1^{(\a)}, \label{xm:eq}\\[0.2cm]
\| R_{M+1}\|
 \ll (\1 b\1^{(\a)} )^{M+1} \rho^{\b\e_{M+1}};
 \label{rm:eq}
\end{gather}
uniformly in $b$ satisfying $\rho^{\b(\s-1)} \1 b\1^{(\a)}\ll 1$.
\end{lem}

\begin{proof}

The existence of $\psi_1\in\BS_{\g}$ with required properties
follows from Lemma \ref{commut:lem}. Further proof is by induction.

Suppose that $\psi_k$ with $k = 1, 2, \dots, K-1$ satisfy
\eqref{psik1:eq}.
In order to conclude that $\psi_K$
also satisfies \eqref{psik1:eq}, first
we need to check that $b_K$ and $t_K$ satisfy \eqref{bt:eq}.

\underline{Step I. Estimates for $b_j$.}
Note that
\begin{equation}\label{sigmatoeps:eq}
\begin{split}
\e_j = &\ \s_{j-1} + \a - 1 \\
 = &\ (j-1)\s + \a - (j-1),\ j \ge 2.
 \end{split}
\end{equation}
To begin with, we prove that all the symbols
$b_j$  with $j \le K$, satisfy the estimate
\eqref{bt:eq}.
We first obtain
a bound for $\ad(b; \psi_{k_1}, \psi_{k_2}, \dots, \psi_{k_q})$
with $k_1+k_2+\dots+k_q = j-1$.
To this end we use \eqref{psik1:eq} with $\g = (\om-\a)q^{-1}+1$
for each $\psi_{k_n}$
and Proposition \ref{commut0:prop} to conclude that
\begin{equation}\label{b:eq}
\1\ad(b; \psi_{k_1}, \psi_{k_2},
\dots, \psi_{k_q})\1^{(\om)}
\ll \1 b\1^{(\a)}
\prod_{n=1}^q \1 \psi_{k_n}\1^{(\g)} \ll (\1 b\1^{(\a)})^j\prod_{n=1}^q
\rho^{\b(\s_{k_n}-\g)},
\end{equation}
for any $\om\in\R$. Obviously, we have:
\begin{align*}
\sum_{n=1}^q \b(\s_{k_n} - \g)
= &\ \b \bigl(q(1-\gamma)+\sum_{n=1}^qk_n(\s-1)\bigr)
=  \b \bigl((j-1)(\s-1) - \om + \a \bigr)\\
= &\ \b(\s_{j-1} + \a - 1 -\om) = \b(\e_j - \om).
\end{align*}
Now, \eqref{bl:eq} implies
\begin{equation*}
\1 b_j\1^{(\om)} \ll  \rho^{\b(\e_j-\om)}(\1 b\1^{(\a)})^j,
\end{equation*}
for all $\om\in\R$, i.e. $b_j$ satisfies \eqref{bt:eq} for all $j\le K$.

\underline{Step II. Estimates for $t_j$.}
For the symbols $t_j$ the proof is by induction.
First of all, note that
\begin{equation*}
\ad(h_0; \psi_1, \psi_1) = - \ad(b^{\natural}, \psi_1),
\end{equation*}
so that, using \eqref{subord:eq}, Proposition \ref{commut0:prop}, and
\eqref{psik1:eq} with $\g = \om - \a + 1$,
we obtain
\begin{equation*}
\1\ad(h_0; \psi_1, \psi_1)\1^{(\om)}
\ll \1 b\1^{(\a)} \1 \psi_1\1^{(\g)}
\ll \rho^{\b(\s_1-\g)}(\1 b\1^{(\a)})^2.
\end{equation*}
By \eqref{sigmaeps:eq}, we have $\s_1 - \g = \s_1 +\a  -1 - \om = \e_2 - \om$,
and hence $t_2$ satisfies \eqref{bt:eq}.
Suppose that all $t_k$ with $k\le j-1 \le K-1$ satisfy \eqref{bt:eq}.
As we have already established that
all $b_k, k\le K$ satisfy \eqref{bt:eq},
by definition \eqref{psil:eq} and \eqref{subord:eq}
all $\ad(h_0; \psi_k)$, $k \le j-1$,
satisfy the same bound, i.e.
\begin{equation*}
\1 \ad(h_0, \psi_k)\1^{(\g)}\ll \rho^{\b(\e_k - \g)}(\1 b\1^{(\a)})^k, k\le j-1.
\end{equation*}
Using the bound \eqref{psik1:eq} for $\psi_k$, $k\le j-1$, with $\g = (\om-1)q^{-1}+1$,
and applying Proposition \ref{commut0:prop},
we obtain for $k_1+k_2+\dots+k_q = j$, $q\ge 2$ the following estimate:
\begin{align}\label{h0:eq}
\1 \ad(h_0; \psi_{k_1}, \psi_{k_2}, &\ \dots, \psi_{k_q})\1^{(\om)}
= \1 \ad\bigl(  \ad(h_0; \psi_{k_1});
\psi_{k_2}, \dots, \psi_{k_q}\bigr)\1^{(\om)}\notag\\[0.2cm]
\ll &\  \1 \ad(h_0; \psi_{k_1})\1^{(\g)} \prod_{n=2}^q
\1\psi_{k_n}\1^{(\g)}
\ll  (\1 b \1^{(\a)})^j \rho^{\b(\e_{k_1}-\g)}\prod_{n=2}^q \rho^{\b(\s_{k_n}-\g)},
\end{align}
for any $\om\in\R$. Obviously, we have:
\begin{align*}
\b(\e_{k_1}-\g)+\sum_{n=2}^q \b(\s_{k_n}-\g)
= &\ \b\bigl(
(\s-1)\sum_{n=1}^q k_n  + 2 + (2m-2)\b^{-1} +  q-1 -q\g
\bigr)\\
= &\ \b\bigl( j(\s-1) + 2 + (2m-2)\b^{-1}  + q(1 -\g) - 1
\bigr)\\
= &\ \b\bigl(\e_j - \om\bigr).
\end{align*}
This leads to \eqref{bt:eq} for all $t_j$, $j\le K$.

\underline{Step III.} In order to handle $\Psi_K$, we use the solution $\Psi$
of the equation \eqref{adb:eq}
constructed in Lemma \ref{commut:lem}.
Then from definition
\eqref{psil:eq} and steps I, II
we immediately conclude that
$\psi_{K}\in \BS_{\g }$ for any $\g\in \R$. 
Moreover, estimate \eqref{psitau:eq} with $\omega=\gamma$ implies
\begin{align*}
\1 \psi_{K}\1^{( \g )}
\ll  &\  \rho^{\b( - (2m-2)\b^{-1} - 1)}\bigl(\1 b_K\1^{(\g )} + \1 t_K\1^{(\g )}\bigr)\\
\ll &\ \rho^{\b(\e_K - \g- (2m-2)\b^{-1} - 1)} (\1 b\1^{(\a)})^K.
\end{align*}
Since
$\e_K - (2m-2)\b^{-1} - 1 = \s_K$,
the required result follows.

\underline{Step IV. Proof of \eqref{rm:eq} and \eqref{xm:eq}.}
We assume that $\rho^{\b(\s-1)} \1 b\1^{(\a)}\ll 1$ throughout.
Before treating the remainder $R_{M+1}$ (see \eqref{r:eq} for its definition), let us prove
estimates \eqref{xm:eq} for the symbols
$\psi = \sum_{j=1}^M \psi_j$ and $x_M = b+\sum_{j=2}^M (b_j + t_j)$.
Using \eqref{psik1:eq}, we get
\begin{equation*}
\1 \psi\1^{(\g)}\ll \sum_{j=1}^M\rho^{\b(\s_j-\g)}(\1 b\1^{(\a)})^j
= \rho^{\b(1-\g)} \sum_{j=1}^M \bigl(\rho^{\b(\s-1)}\1 b\1^{(\a)}\bigr)^j,
\end{equation*}
which implies that
\begin{equation}\label{globalpsi:eq}
\1\psi\1^{(\g)}\ll \rho^{\b(\s-\g)} \1 b\1^{(\a)} \biggl( 1
+ \bigl(\rho^{\b(\s-1)}\1 b\1^{(\a )}\bigr)^{M-1}\biggr)
\ll \rho^{\b(\s-\g)} \1 b\1^{(\a)}
\end{equation}
with an arbitrary $\g\in\R$. Similarly,
in view of \eqref{bt:eq} and \eqref{sigmatoeps:eq},
\begin{align}\label{globalx:eq}
\1 x_M\1^{(\a)}\ll &\
\1 b\1^{(\a)} + \sum_{j=2}^{M} \rho^{\b(\e_j-\a)} (\1 b\1^{(\a)})^j
= \1 b\1^{(\a)} + \sum_{j=2}^M \rho^{\b(j-1)(\s-1)} (\1 b\1^{(\a)})^j \notag\\[0.2cm]
\ll &\ \1 b\1^{(\a)}\biggl(1 + (\rho^{\b(\s-1)}\1 b\1^{(\a)})^{M-1}\biggr)
\ll \1 b\1^{(\a)}.
\end{align}
Now we apply these estimates to find upper bounds for $R_{M+1}$.

The remainder $R_{M+1}$ consists of three components.
To estimate the first one, $B_{M+1}$, note, that
in view of \eqref{bt:eq},
$b_{M+1}\in \BS_0$ and $\1 b_{M+1}\1^{(0)} \ll \rho^{\b\e_{M+1}}(\1 b\1^{(\a)})^{M+1}$.
By
Proposition \ref{bound:prop}, we conclude that the norm of $B_{M+1}$ is bounded
by $ (\1 b\1^{(\a)})^{M+1}\rho^{\b\e_{M+1}}$ as required.

Consider now $R^{(1)}_{M+1}$ defined in \eqref{decomp:eq}.   
Using \eqref{globalpsi:eq} with $\g = -(M+1)^{-1} \a + 1$ and applying
 Proposition \ref{commut0:prop},we conclude that
\begin{equation*}
\1
\ad^{M+1}(b; \psi)\1^{(0)}
\ll \1 b\1^{(\a)} \bigl(\1 \psi\1^{(\g)}\bigr)^{M+1}
\ll
\rho^{\b (M+1)(\s-\g)}\bigl(\1 b\1^{(\a)}\bigr)^{M+2}.
\end{equation*}
In view of \eqref{sigmatoeps:eq},
\begin{equation*}
(M+1)(\s-\g) =  M\s + \a - M + \s -1  = \e_{M+1}  + \s-1,
\end{equation*}
so that Proposition \ref{bound:prop} yields:
\begin{equation}\label{big:eq} 
\|\ad^{M+1}(B; \Psi)\|\ll  
\rho^{\b\e_{M+1}}
\bigl(\1 b\1^{(\a)}\bigr)^{M+1} 
\rho^{\b(\s-1)}\1 b\1^{(\a)} 
\ll  \rho^{\b\e_{M+1}}
\bigl(\1 b\1^{(\a)}\bigr)^{M+1}.
\end{equation}
To estimate the norm of $\ad^{M+1}(H_0; \Psi)$, note that
by \eqref{globalx:eq} and \eqref{nat:eq},
\begin{equation*}
\1 x_M^{\natural}\1^{(\g)}\ll \rho^{\b(\a-\g)}\1 b\1^{(\a)}
\end{equation*}
for any $\g\in\R$.
It follows from \eqref{psi1:eq} and \eqref{psil:eq} that
$\ad(h_0, \psi) + x^{\natural}_M = 0$.
Thus, using the above bound and \eqref{globalpsi:eq} with  
$\g = M(1+M)^{-1}$,
we obtain with the help of Proposition \ref{commut0:prop} that
\begin{align*}
\1 \ad^{M+1}(h_0; \psi)\1^{(0)} = &\ \1 \ad^M(x_M^{\natural}; \psi)\1^{(0)}
\ll \1 x_M^\natural\1^{(\g)} \bigl(\1 \psi\1^{(\g)}\bigr)^{M}\\
\ll &\  \rho^{\b M(\s-\g) + \b(\a-\g)}\bigl(\1 b\1^{(\a)}\bigr)^{M+1}.
\end{align*}
According to \eqref{sigmatoeps:eq},
\begin{equation*}
M(\s-\g) +  \a-\g = M\s + \a - M -(M+1)\g + M = \e_{M+1},
\end{equation*}
so that by Proposition \ref{bound:prop},
\begin{equation*}
\|\ad^{M+1}(H_0; \Psi)\|\ll \rho^{\b\e_{M+1}}\bigl(\1 b\1^{(\a)}\bigr)^{M+1}.
\end{equation*}
Together with \eqref{big:eq} this leads to the estimate of the form \eqref{rm:eq}
for $R^{(1)}_{M+1}$.

Following the same strategy, one easily obtains the sought estimate for
the norm of the $R^{(2)}_{M+1}$ defined in \eqref{rtilde:eq}.
This completes the proof of \eqref{rm:eq}.
\end{proof}


Let us now summarize the results of this section in the following
Theorem:
the implications of the above Lemma for the operator $H = H_0 + \op(b)$, defined in
\eqref{h:eq}.

\begin{thm}\label{reduction:thm}
Let $b\in\BS_{\a}(w)$, $w(\bxi) = \lu \bxi\ru^\b, \b\in (0, 1]$, $\a\in\R$
be a symmetric symbol, and let $H$ be the operator defined in \eqref{h:eq}.
Suppose that the condition \eqref{alm1:eq} is satisfied.
Then for any positive integer $M$
there exist symmetric
symbols $\psi=\psi_M$, $x = x_M$, and a self-adjoint bounded operator $R_{M+1}$
satisfying the following properties:
\begin{enumerate}
\item $\psi\in\BS_\g$ for all $\g\in\R$, $x\in\BS_\a$,
and
\begin{align*}
\1 \psi\1^{(\g)}\ll &\ \rho^{\b(\s-\g)}\1 b\1^{(\a)},\\
\1 x\1^{(\a)}\ll &\ \1 b\1^{(\a)},\\
\|R_{M+1}\|\ll &\ (\1b\1^{(\a)})^{M+1} \rho^{\b \e_{M+1}},
\end{align*}
uniformly in $b$ satisfying $\1 b\1^{(\a)}\ll 1$;
\item
The operator
$A_1 = e^{-i\Psi} H e^{i\Psi}, \Psi = \op(\psi)$, has the form
\begin{equation}\label{a1:eq}
A_1 = A_0 + X^{\flat} + X^{\downarrow, \sharp, \uparrow} + R_{M+1},\
A_0 = H_0 + X^o.
\end{equation}
\end{enumerate}
\end{thm}

\begin{proof}
Note that the condition \eqref{alm1:eq} is equivalent to $\s< 1$. Thus the
existence of symbols $\psi$, $x$ and the operator $R_{M+1}$ with required properties
follows from Lemma \ref{gauge:lem}. In particular, the claimed upper bounds
are direct consequences
of \eqref{xm:eq} and \eqref{rm:eq}.
\end{proof}


\section{Geometry of congruent points: resonant sets}\label{ezhiki:sect}

In the course of the proof we need a substantial number of
certain lattice-geometric constructions. They are discussed in this section.
First, we fix the notation.

For any vector $\bxi\in\R^d$, $\bxi\not = 0$, we denote
$\normal(\bxi) = \bxi|\bxi|^{-1}$.
For any  $\boldeta\in\R^d$ the vector $\bxi_{\boldeta}$ is the projection of $\bxi$
onto the one-dimensional subspace spanned by $\boldeta$, i.e.
$\bxi_{\boldeta} = (\bxi\cdot\normal(\boldeta)) \normal(\boldeta)$.
Let $\T_r\subset\SG^\dagger$ and $\T_r^0\subset\SG^\dagger$
be the sets defined in \eqref{T:eq}. We always assume that
$$
\underline{r\ge r_0,
\ \ \textup{where
$r_0$ is such that $\T_{r_0}$ contains $d$ linearly independent lattice vectors.}}
$$
We say that a subspace $\GV\subset \R^d$ is \textsl{a lattice
$r$-subspace} if $\GV$ is spanned by some linearly independent lattice vectors
from the set $\T_r^0$. A one dimensional lattice subspace
spanned by a vector $\bth\in\T_r$ is denoted by $\GV(\bth)$.
The set of all lattice $r$-subspaces of dimension $n$ is denoted
by $\CV(n)=\CV(r,n),\ n = 0, 1, \dots, d$, and
$\mathcal W = \mathcal W(r)= \cup_{n=0}^d \CV(r, n)$.
We have included in this list the zero dimensional
subspace $\GX = \{0\}$.
For any vector $\bxi\in\R^d$ we denote by $\bxi_{\GV}$
its orthogonal projection on $\GV\in\CV(r, n)$.
In particular, $\bxi_{\GX} = \mathbf 0$.

The notation $\w_n$ is used for the volume of the unit ball in $\R^n$.

\subsection{Elementary geometrical estimates}

We begin with estimates for distances between lattice subspaces and lattice points.

\begin{lem} \label{cassels:lem}
\begin{enumerate}
\item
For any  $\GV\in \CV(r, d-1)$ there exists a vector
$\bg\in \SG$ such that $\bg_{\GV} = 0$, and
$ |\bg|\le 2\dc(\SG)\w_{d-1}^{-1}  \pi^{1-d} r^{d-1}$.
\item
For any $\GW\in \CV(r, n), n\le d-1$, and any
$\bnu\notin \GW$, $\bnu\in\SG^{\dagger}$,
one has
\begin{equation*}
\dist(\bnu, \GW)\ge \dc(\SG)^{-1}\w_{d-1}  \pi^d r^{1-d}.
\end{equation*}
\end{enumerate}
\end{lem}

\begin{proof}  Denote  by $\be\in \R^d$ a unit vector, orthogonal to $\GV$.
For $t\ge 1$, let $\CA_t\subset \R^d$ be the cylindrical set
\begin{equation*}
\CA_t = \{\bxi\in \R^d: |\bxi_{\GV}|< 2\pi r^{-1},  |\bxi_{\be}|\le t \},
\end{equation*}
which is obviously convex and symmetric about the origin.
Moreover,
\begin{equation*}
\volume(\CA_t)= t (2\pi)^{d-1}r^{-d+1}\w_{d-1}.
\end{equation*}
By \cite{Cas},
\S III.2.2 Theorem II, under the condition
$\volume(\CA_t)\ge \dc(\SG) 2^d$ the set
$\CA_t$ contains at least two points $\pm\bg\in\SG$.
The above condition is satisfied if
$t = \dc(\SG)\w_{d-1}^{-1} 2^{d} (2\pi)^{1-d} r^{d-1}$.

By definition of $\CA_t$,
for any $\bm\in \T_r\cap \GV$ we have $|\bg\cdot\normal(\bm)|< 2\pi r^{-1}$, and hence
$|\bg\cdot\bm|< 2\pi r^{-1} |\bm| \le 2\pi$. Since $\bg\in\SG$ and
$\bm\in\SG^\dagger$,
the latter inequality  implies that $\bg\cdot \bm = 0$, so $\bg_{\GV} = 0$.
It is also clear that $|\bg|\le t = 2\dc(\SG)\w_{d-1}^{-1}  \pi^{1-d} r^{d-1}$.

To prove the second statement, we find a subspace
$\GV\in \CV(r, d-1)$ such that $\bnu\notin \GV$
and $\GW\subset \GV$, and denote by $\bg\in\SG$ the vector orthogonal to $\GV$,
found at the first step of the proof. Since
\begin{equation*}
\dist(\bnu, \GW)\ge \dist(\bnu, \GV)
= \frac{1}{|\bg|} |\bnu\cdot\bg|\ge \frac{2\pi}{|\bg|},
\end{equation*}
the required inequality follows from the first part of the Lemma.
\end{proof}

\begin{lem}\label{projections:lem}
  Let $\GV\in\CV(r, m)$, $m \le d-1$.
\begin{enumerate}
\item
If $\bth\in\T_r$ and $\bth\notin \GV$,
then for any $\bxi\in \GV+\GV(\bth)$ we have
\begin{equation*}
|\bxi|\ll r^{d}(|\bxi_{\bth}| + |\bxi_{\GV}|).
\end{equation*}

\item
For any subspace $\GW\in \CV(r, n)$, $n\le d-1$,
\begin{equation*}
|\bxi_{\GV+\GW}|\ll r^{nd}(|\bxi_{\GV}| + |\bxi_{\GW}|).
\end{equation*}
\end{enumerate}
\end{lem}

\begin{proof} Let $\bth\notin \GV$, and let $\bxi\in \GV+\GV(\bth)$.
By Lemma \ref{cassels:lem},
\begin{equation}\label{e:eq}
|\bth_{\be}| = \dist(\bth, \GV) \gg r^{1-d},
\end{equation}
where $\be\in \GV+\GV(\bth)$ is a unit vector orthogonal to $\GV$.
We have:
\begin{equation}\label{ex:eq}
\bxi = \bxi_\be + \bxi_{\GV},
\end{equation}
so that
\begin{equation*}
 \bxi_{\bth} \cdot \bth = \bxi\cdot \bth =
\bxi_{\be}\cdot\bth_{\be} + \bxi_{\GV}\cdot \bth_{\GV}.
\end{equation*}
Since $|\bxi_{\be}\cdot\bth_{\be}| = |\bxi_{\be}|\ |\bth_{\be}|$,
using \eqref{e:eq} we obtain
\begin{equation*}
|\bxi_\be|\le \frac{1}{|\bth_\be|}\bigl(
|\bxi_{\GV}| + |\bxi_{\bth}|
\bigr)  |\bth|\ll r^d (|\bxi_{\GV}| + |\bxi_{\bth}|).
\end{equation*}
Together with \eqref{ex:eq} this implies
the required bound.

To prove the second statement we may assume that $\bxi\in \GV + \GW$.
Let $\bth_1, \bth_2, \dots, \bth_n\in\T_r\cap \GW$ be a linearly independent
set of lattice vectors.
Denote by $\GW^{(j)}$, $j = 1, 2, \dots, n$
the subspace spanned by $\bth_1, \bth_2, \dots, \bth_j$.
Applying the first part of the Lemma repeatedly, we get
\begin{align*}
|\bxi|\ll &\ r^d(|\bxi_{\bth_n}| + |\bxi_{\GW^{(n-1)}+\GV}|)\\
\ll &\ r^d\bigl(|\bxi_{\bth_n}| + r^d(|\bxi_{\bth_{n-1}}| + |\bxi_{\GW^{(n-2)}+\GV}|)\bigr)
\ll \dots \ll r^{nd}\biggl(\sum_{j=1}^n |\bxi_{\bth_j}| + |\bxi_{\GV}|\biggr).
\end{align*}
Noticing that $|\bxi_{\bth}|\le |\bxi_{\GW}|$ for any
$\bth\in\GW$,
we get the proclaimed estimate.
\end{proof}

\subsection{Congruent vectors}
For a non-zero vector $\bth$ we
define \textsl{the resonant layer}
corresponding to $\bth$ by
\begin{equation}\label{layer:eq}
\L(\bth) : = \{\bxi\in\R^d,\, |\bxi\cdot \bth|  < \rho^{\al_1} |\bth|\}.
\end{equation}
Here and below, $\a_1\in (0, 1)$ is a fixed number which will be specified later, and $\rho\ge 1$.
For the sake of uniformity of the notation, in case $\bth = \mathbf 0$ we denote
$\L(\mathbf 0) = \R^d$.

\begin{defn}
\label{reachability:defn}
Let $\bth, \bth_1, \bth_2, \dots, \bth_m$ be some vectors from $\T^0_r$,
which are not necessarily distinct.
\begin{enumerate}
\item \label{1}
We say that two vectors
$\bxi, \boldeta\in\R^d$ are \textsl{$\bth$-resonant congruent}
if both $\bxi$ and $\boldeta$ are inside $\L(\bth)$
and $(\bxi - \boldeta) =l\bth$ with $l\in\Z$.
In this case we write $\bxi \lra \boldeta \mod \bth$.
In particular, each $\bxi\in\R^d$ is $\mathbf 0$-resonant congruent to itself.
\item\label{2}
For each $\bxi\in\R^d$ we denote by
$\BUps_{\bth}(\bxi)$ the set of all points which are
$\bth$-resonant congruent to $\bxi$.
For $\bth\not = \mathbf 0$ we say that
$\BUps_{\bth}(\bxi) = \varnothing$
if $\bxi\notin\L(\bth)$.
\item\label{3}
We say that $\boldeta$
is \textsl{$\bth_1, \bth_2, \dots, \bth_m$-resonant congruent}
to $\bxi$, if there exists a sequence $\bxi_j\in\R^d, j=0, 1, \dots, m$ such that
$\bxi_0 = \bxi$, $\bxi_m = \boldeta$,
and $\bxi_j\in\BUps_{\bth_j}(\bxi_{j-1})$ for $j = 1, 2, \dots, m$.
\item
We say that $\boldeta $ and
$\bxi $ are \textsl{resonant congruent}, if
$\boldeta$ is  $\bth_1, \bth_2, \dots, \bth_m$-resonant congruent to $\bxi$
with some
 $\bth_1, \bth_2, \dots, \bth_m\in\T_r$.
 The set of \textbf{all} points, resonant congruent
to $\bxi$, is denoted by $\BUps(\bxi)$.
We use the notation
$\SN(\bxi) = \{ \bm\in\SG^\dagger: \bxi+\bm\in\BUps(\bxi)\}$,
i.e. $\SN(\bxi) = \BUps(\bxi) - \bxi\subset \SG^\dagger$.
For points $\boldeta\in\BUps(\bxi)$ we write $\boldeta\lra\bxi$.
\end{enumerate}
\end{defn}

\begin{rem}\label{reachability:rem}
\begin{enumerate}
\item
Note that according to the above
definition every point in $\R^d$ is resonant congruent to some point. In particular,
every vector $\bxi$ is $\mathbf 0$-resonant congruent to itself,
i.e. $\BUps_{\mathbf 0}(\bxi) = \{\bxi\}$.
\item
It is clear that the resonant congruence defines
an equivalence relation, so that if $\boldeta\lra\bxi$ then we have
$\BUps(\bxi) = \BUps(\boldeta)$.
\item \label{3:rem}
Note also that if $\boldeta$ is \textsl{$\bth_1, \bth_2, \dots, \bth_m$-resonant congruent}
to $\bxi$, then $\bxi$ is \textsl{$\bth_m, \bth_{m-1}, \dots, \bth_1$-resonant congruent}
to $\boldeta$,  and
$\bxi$ is
\textsl{$\bth_1, \bth_2, \dots, \bth_{m-1},
\bth_m, \bth_{m-1}, \dots, \bth_1$-resonant congruent}
to itself.
\end{enumerate}
\end{rem}

Let us establish some immediate properties of congruent points:

\begin{lem}\label{trans:lem}
For any vectors $\bxi\in\R^d$, $\bth\in\T_r$, and any $\bnu\perp\bth$ we have
$\BUps_{\bth}(\bxi)+\bnu\subset \BUps_{\bth}(\bxi+\bnu)$.

As a consequence, if $\SN(\bxi)\subset \GV$, then
for any vector $\bnu\in\GV^\perp$ we have
$\BUps(\bxi)+\bnu\subset \BUps(\bxi+\bnu)$
and $\SN(\bxi)\subset\SN(\bxi+\bnu)$.
\end{lem}

\begin{proof} If $\BUps_{\bth}(\bxi)$ is empty, the result is obvious.
For a non-empty $\BUps_{\bth}(\bxi)$
the first statement is an immediate consequence of
Definition \ref{reachability:defn} in view of the orthogonality of $\bth$ and $\bnu$.
Under the conditions $\SN(\bxi)\subset\GV$, $\bnu\in\GV^\perp$,
this leads to the inclusion $\SN(\bxi)\subset\SN(\bxi+\bnu)$.
\end{proof}

The next Lemma is the first of many results, establishing some inequalities for
congruent points and/or lattice subspaces.
To avoid unnecessary repetitions in their formulations, we adopt
the following convention.

\begin{convention}\label{uniform:conv}
\begin{enumerate}
\item
Each inequality (e.g. $\ll$, $\gg$ or $\asymp$),
involving points of the Euclidean space
and/or lattice subspace(s), is assumed to be uniform in those objects.
\item
The constants
in the inequalities are allowed  to depend on the dimension $d$ and the lattice $\SG^\dagger$
only.
\item
We say that a certain statement, involving points of the Euclidean space
and/or lattice subspace(s), holds for sufficiently large $\rho$, if
there is a number $\rho_0>0$, independent on the points and/or subspace(s) at hand,
such that the statement holds for all $\rho \ge \rho_0$.
\end{enumerate}
\end{convention}

\begin{lem}\label{lem:reachability1}
If $\boldeta\in\BUps(\bxi)$, then
\begin{equation}\label{couple:eq}
|\bxi-\boldeta|\ll\rho^{\al_1}r^{(d-1)d}.
\end{equation}
In particular,
\begin{equation}\label{sn:eq}
\max_{\SN(\bxi)}|\bm|\ll \rho^{\a_1}r^{d(d-1)},\ \
N(\bxi):=\card \SN(\bxi)\ll \rho^{d\a_1}r^{d^2(d-1)},
\end{equation}
uniformly in $\bxi$.
\end{lem}

The proof of this Lemma relies on the following result:

\begin{lem}\label{lem:reachability11}

(i) Let $\boldeta\in\BUps(\bxi)$
be $\bth_1, \bth_2,\dots, \bth_m$-resonant congruent to $\bxi$, and let
\[
\textup{span}(\bth_1, \bth_2,\dots, \bth_m)=:\GV\subset \CV(n)
\]
with some $n\le d$. Then
\begin{equation}\label{couple10:eq}
|\bxi_\GV|\ll\rho^{\al_1}r^{(n-1)d}.
\end{equation}

(ii) Suppose that
\begin{equation*}
\textup{span}\ \SN(\bxi) = \GV\subset \CV(n),
\end{equation*}
with some $n\le d$. Then
\begin{equation}\label{couple1:eq}
|\bxi_\GV|\ll\rho^{\al_1}r^{(n-1)d}.
\end{equation}
\end{lem}

\begin{proof}[Proof of Lemma \ref{lem:reachability1}]
As the relation $\bxi\lra\boldeta$ is symmetric, Lemma \ref{lem:reachability11}
 implies that $|\bxi_{\GV}| + |\boldeta_{\GV}|\ll \rho^{\a_1} r^{(d-1)d}$. 
As $\bxi-\boldeta\in\GV$, we have
\begin{equation*}
|\bxi-\boldeta|=|(\bxi-\boldeta)_\GV|
\le |\bxi_{\GV}|+|\boldeta_\GV|\ll\rho^{\al_1}r^{(d-1)d},
\end{equation*}
which is \eqref{couple:eq}.
The first estimate in \eqref{sn:eq} follows from
the inequality
\begin{equation*}
\max_{\bm\in\SN(\bxi)}|\bm|\le \sup_{\boldeta\in\BUps(\bxi)}
|\bxi-\boldeta|
\end{equation*}
and from \eqref{couple:eq}. The bound for $N(\bxi)$ is simply
an estimate for the number of lattice points in the ball of radius
$\ll \rho^{\a_1} r^{d(d-1)}$.
\end{proof}

\begin{proof}[Proof of Lemma \ref{lem:reachability11}]
The proof of the first part is by induction.
For $n=1$ the statement is an
immediate consequence of the definitions.
Assume that \eqref{couple10:eq} holds for all $n\le k$, $k\ge 1$, and let us prove
\eqref{couple10:eq} for $n=k+1$.

Let $\boldeta\in\BUps(\bxi)$
be $\bth_1, \bth_2,\dots, \bth_m$-resonant congruent to $\bxi$,
and let $\bxi_0, \bxi_1,
\dots, \bxi_m$ be the vectors from Definition  \ref{reachability:defn}(\ref{3}).
Let $l\ge k, l < m$,  be a number, uniquely defined by the following conditions:
$\text{span}(\bth_1,\dots,\bth_{l+1}) = \GV\subset \CV(k+1)$,
and $\GW:=\text{span}(\bth_1,\dots,\bth_{l})\not=\GV$.
Clearly, $\dim\GW = k$.
Since $\bxi_{l+1} \lra \bxi_l \mod \bth_{l+1}$, by Definition
\ref{reachability:defn}(\ref{1}) we have $|(\bxi_l)_{\bth_{l+1}}|<\rho^{\a_1}$.
At the same time, since $\bxi_l\lra\bxi$, we also have
$|(\bxi_l)_{\GW}|\ll \rho^{\a_1} r^{(k-1)d}$ by the induction assumption, as
$\GW\in\CV(k)$.
Thus, according to Lemma \ref{projections:lem}(1),
\begin{equation}\label{eq:1}
|(\bxi_l)_{\GV}|\ll r^d (|(\bxi_l)_{\GW}| + |(\bxi_l)_{\bth_{l+1}}|)
\ll \rho^{\a_1} r^{kd}.
\end{equation}
Noticing that $\bxi - \bxi_l\in\GW$, we infer from the induction assumption again that
\begin{equation*}
|\bxi-\bxi_l|=|(\bxi-\bxi_l)_\GW|\le |\bxi_\GW|+|(\bxi_l)_\GW|
\ll\rho^{\al_1}r^{(k-1)d},
\end{equation*}
Together with \eqref{eq:1} this leads
to $|\bxi_\GV|\ll \rho^{\al_1}r^{ (n-1)d}$.
Thus by induction,  \eqref{couple10:eq} is proved.

In order to prove part (ii), we
use the observation made in Remark \ref{reachability:rem} (\ref{3:rem}) to conclude that
under the assumption $\textup{span}\ \SN(\bxi) = \GV$ there are vectors $\bth_1,\dots,\bth_m$
such that $\bxi$ is $\bth_1,\dots,\bth_m$-resonant congruent to itself and
$\textup{span}(\bth_1,\dots,\bth_m) = \GV$. It remains to use the first part of the Lemma.
\end{proof}

One important conclusion of Lemma \ref{lem:reachability1}
is that the numbers $\card\BUps(\bxi)$ are bounded uniformly
in $\bxi\in\R^d$.

We need further notions related to congruency:

\begin{defn}\label{chain:defn}
We say that a point $\bxi\in\R^d$ is $\bth$-non-critical, if
for a sufficiently small $\e>0$ we have
$\BUps_{\bth}(\bxi+\bmu) = \BUps_{\bth}(\bxi)+ \bmu$ for all $\bmu: |\bmu|<\e$.
Otherwise we call $\bxi$ $\bth$-critical.

If $\bxi$ is $\bth$-non-critical for all $\bth\in\T_r$, we call it
non-critical.

 We say that the set $\BUps(\bxi)$ is non-critical,
if it consists of non-critical points. In other words,
for a sufficiently small $\e>0$, we have
$\BUps(\bxi+\bmu) = \BUps(\bxi)+\bmu$ if $|\bmu| < \e$. Otherwise
$\BUps(\bxi)$ is said to be critical.
\end{defn}

If $\BUps(\bxi)$ is non-critical, then the set $\SN(\cdot)$
remains constant in a neighbourhood of $\bxi$.

It is clear that if $\bth\not = 0$, then the set
$$
\{\bxi\in\R^d:  \bxi\cdot \normal(\bth)  \not \equiv \pm \rho^{\a_1} \mod |\bth|
\},
$$
consists of $\bth$-non-critical points. Each $\bxi\in\R^d$ is obviously
$\mathbf 0$-non-critical. Thus,
the set of $\bxi\in\R^d$ for which  $\BUps(\bxi)$ is non-critical,
is open and of full measure.

\subsection{Resonant sets and their properties}
\label{resonant:subsect}

Our aim is to construct a collection of sets
$\Xi(\GV)\subset\R^d$ parametrised by subspaces
$\GV\in\CV(r, n)$ for $n = 0, 1, \dots, d$, and depending on the parameter
$\rho >0$,
satisfying the following properties:
\begin{equation}\label{exhaust:eq}
\R^d = \bigcup_{\GV\in\mathcal W(r)} {\Xi(\GV)},
\end{equation}
\begin{equation}\label{disjoint:eq}
\Xi(\GV_1)\cap\Xi(\GV_2)=\varnothing \ \ \ \textup{for}\ \ \ \  \GV_1\ne\GV_2,
\end{equation}
\begin{equation}\label{span:eq}
\textup{For each}\ \  \bxi\in\Xi(\GV), \textup{one has}\ \
\BUps(\bxi) \subset\Xi(\GV) \ \ \
\textup{and}\ \ \  \SN(\bxi)\subset\GV.
\end{equation}
The required sets will depend on the arbitrarily chosen real
parameters $\a_0, \a_1, \dots, \a_d$ and $\varkappa$, satisfying the conditions
$0=\a_0<\a_1< \a_2< \dots< \a_d <1$ and
\begin{equation}\label{eq:condition1}
\begin{cases}
r\le \rho^{\varkappa},\\[0.2cm]
\a_{n+1} > \a_n + 2\varkappa d^2,\ n = 0, 1, \dots, d-1.
\end{cases}
\end{equation}
>From now on these conditions are always assumed to hold.
Under these conditions the inequalities \eqref{sn:eq} give
for sufficiently large $\rho$:
\begin{equation}\label{card:eq}
\begin{cases}
|\bm|\le \rho^{\frac{\a_1+\a_2}{2}}, \bm\in\SN(\bxi);\\[0.2cm]
\card\BUps(\bxi)\le \rho^{\frac{ \a_1+\a_2 }{2} d},
\end{cases}
\end{equation}
uniformly in $\bxi$.
The next Lemma is a straightforward consequence of the above inequalities:

\begin{lem}\label{bups:lem}
Let $|\bxi|\ge \rho/2$. Then
\begin{equation*}
\max_{\BUps(\bxi)}|\boldeta|
\asymp \min_{\BUps(\bxi)}|\boldeta|\asymp |\bxi|,
\end{equation*}
uniformly in $\bxi$.
\end{lem}

For each lattice subspace $\GV\in\CV(n)$, $n=0, 1, \dots, d$ we introduce the (open)
sets
\begin{equation}\label{xi1:eq}
\Xi_1(\GV):=\{\bxi\in\R^d,\,|\bxi_\GV|<\rho^{\a_n}\},
\ \ \ \Xi_2(\GV) := \bigcup_{\bxi\in\Xi_1(\GV)} \BUps(\bxi),
\end{equation}
and
\begin{equation}\label{xi:eq}
\Xi(\GV):=
\begin{cases}
\Xi_2(\GV)\setminus\bigcup\limits_{m=n+1}^d\bigcup
\limits_{\GW\in\CV(m)}\Xi_2(\GW),\ n < d;\\[0.3cm]
\Xi_2(\GV),\ n = d.
\end{cases}
\end{equation}
The set $\Xi(\GV), \GV\not=\GX$,
is referred to as the \textsl{resonant set, associated with
the lattice subspace  $\GV$}.
By definition of $\Xi_2(\GV)$, we have $\bxi\in\Xi_2(\GV)$ if and only if
$\BUps(\bxi)\subset\Xi_2(\GV)$, so that
\begin{equation}\label{xi2:eq}
\bxi\in\Xi(\GV) \ \ \textup{if and only if}\ \ \ \BUps(\bxi)\subset\Xi(\GV).
\end{equation}
This observation immediately leads to the natural
representation of $\Xi(\GV)$ as a union of non-intersecting equivalence classes:
\begin{equation}\label{equiv:eq}
\Xi(\GV) = \bigcup_{\bxi\in\Xi(\GV)/{\lra}} \BUps(\bxi).
\end{equation}
Note that
 $\Xi_1(\GV(\bth)) = \Xi_2(\GV(\bth)) =\L(\bth)$ for each $\bth\in\T_r$.
For $\GV = \GX$ we have
$\Xi_1(\GX)=\Xi_2(\GX)=\R^d$. The sets
\begin{equation}\label{(non)resonant:eq}
\CB: = \Xi(\GX)=\R^d\setminus
\bigcup_{m\ge 1}\bigcup_{\GW\in\CV(m)}\Xi_2(\GW)
\ \ \ \textup{and}\ \ \ \ \CD = \R^d\setminus \CB
\end{equation}
are called the \textsl{non-resonant set} and \textsl{resonant set} of $\R^d$  respectively.
The sets introduced above obviously depend on the parameter $\rho$. Whenever necessary,
the dependence on $\rho$ is reflected in the notation, e.g. $\Xi(\GV; \rho)$,
$\CB(\rho), \CD(\rho)$.

Note an immediate consequence of the definition \eqref{xi:eq} and Lemma
\ref{lem:reachability1}:

\begin{lem}\label{lem:properties3}
Assume \eqref{eq:condition1}.
Let $\GV\in\CV(n)$,\ $n=1, 2, \dots, d$, and $\bxi\in\Xi_2(\GV)$.
Then for sufficiently large $\rho$ we have:
\begin{equation}\label{proj:eq}
|\bxi_\GV|\le \begin{cases}
2\rho^{\frac{\a_1+\a_2}{2}}, n = 1,\\[0.2cm]
2\rho^{\a_n},\ n = 2, \dots, d.
\end{cases}
\end{equation}
\end{lem}

\begin{proof}
By definition \eqref{xi1:eq},
$\bxi\in \BUps(\boldeta)$ for some $\boldeta\in\Xi_1(\GV)$.
Thus, by \eqref{card:eq},
\begin{equation*}
|\bxi_{\GV}|\le |\boldeta_{\GV}| + \max_{\bm\in\BUps(\boldeta)} |\bm|
< \rho^{\a_n}+ \rho^{\frac{\a_1+\a_2}{2}}.
\end{equation*}
In view of monotonicity of $\a_j$'s, this proves \eqref{proj:eq}.
\end{proof}

The previous Lemma shows that the resonant sets $\Xi(\GV), \GV\not=\GX$,
are ``small"
relative to the non-resonant set $\CB = \Xi(\GX)$. More precisely, we show
that the resonant set $\CD$  has a small angular measure.
To this end for each $\bth\in\T_r$ define
$$
\tilde\L(\bth) = \{\bxi\in\R^d: |\bxi\cdot\normal(\bth)| < 2\rho^{\a_{d-1}}\},
$$
cf. \eqref{layer:eq}.
By Lemma \ref{lem:properties3}, for any $\GV\in\CV(n), n\le d-1$ we have
$$
\Xi(\GV)\subset \bigcup_{\bth\in\GV\cap\T_r} \tilde\L(\bth)
$$
so that
$$
\CD\setminus B(\rho/8)
\subset  \bigcup_{\bth\in\T_r} \tilde\L(\bth)\setminus B(\rho/8).
$$
An elementary calculation shows that
\begin{equation*}
\tilde\L(\bth)\setminus B(\rho/8)
\subset S(\bth; \rho)\times [\rho/8, \infty),\
S(\bth; \rho):= \{\boldsymbol \Om\in\mathbb S^{d-1}:
|\boldsymbol\Om\cdot \normal(\bth)| < 16\rho^{\a_{d-1}-1}\},
\end{equation*}
for all sufficiently large $\rho$.
Let
\begin{equation}\label{surface:eq}
S(\rho) = \bigcup_{\bth\in\T_r} S(\bth; \rho), \ \
T(\rho) = \mathbb S^{d-1}\setminus S(\rho).
\end{equation}

\begin{lem}\label{surface:lem}
Let the sets $S(\rho)\subset\mathbb S^{d-1}$, $T(\rho)\subset\mathbb S^{d-1}$
be as defined above.  Then
\begin{equation}\label{surface1:eq}
\volume_{\mathbb S^{d-1}} S(\rho)\ll \rho^{\a_d - 1},\
\volume_{\mathbb S^{d-1}} T(\rho)\asymp 1,
\end{equation}
for sufficiently large $\rho$.

\end{lem}

\begin{proof}
The elementary bound
\begin{equation*}
\volume_{\mathbb S^{d-1}} S(\bth; \rho)
= \int_{\mathbb S^{d-2}} \int_{|\cos\om|\le 16\rho^{\a_{d-1}-1}} \sin^{d-2}\om d\om
d\hat{\boldsymbol\Om}\ll \rho^{\a_{d-1}-1},
\end{equation*}
together with the observation that the number of the sets $\tilde\L(\bth)$
is bounded above by $\card \T_r\ll r^d$, gives the estimate
\begin{equation*}
\volume_{\mathbb S^{d-1}} S(\rho)\ll r^d \rho^{\a_{d-1}-1}\ll \rho^{\a_d-1}.
\end{equation*}
Here we have used \eqref{eq:condition1}.

The second bound in \eqref{surface1:eq} immediately follows from the first one by
definition \eqref{surface:eq}.
\end{proof}

In what follows, apart from the non-resonant set $\CB(\rho)$,
the set
\begin{equation}\label{angular:eq}
\tilde\CB(\rho) = T(\rho)\times [\rho/8, \infty)
\end{equation}
will play an important role. Since $\CD\subset S(\rho)\times [\rho/8, \infty)$,
and $\CB = \R^d\setminus\CD$ (see \eqref{(non)resonant:eq}), we have
\begin{equation*}
 \tilde\CB(\rho)\subset \CB(\rho).
\end{equation*}

Let us now proceed with our study of the resonant sets $\Xi(\GV)$.
Introduce a new notion:

\begin{defn}\label{strdis:defn}
Two subspaces $\GV_1$, $\GV_2$ are said to be
\textsl{strongly distinct}, if they are distinct,
and neither of $\GV_1$, $\GV_2$ is a subspace of the other.

An equivalent definition of this property is that
$\dim(\GV_1+\GV_2)> \max(\dim \GV_1, \dim \GV_2)$.
\end{defn}

Here is a simple relation between the sets $\Xi_1, \Xi_2$:

\begin{lem}\label{xi12:lem}
Assume \eqref{eq:condition1}. Then
for any strongly distinct lattice subspaces
$\GV$ and $\GW$ we have
\begin{equation}\label{strongly:eq}
\Xi_2(\GV)\cap\Xi_2(\GW)\subset \Xi_1(\GV + \GW),
\end{equation}
for sufficiently large $\rho$.
\end{lem}

\begin{proof}
Let $\GV\in \CV(n)$, $\GW\in\CV(m)$, and assume without loss of generality that
$m\ge n$.

Suppose first that $m\ge 2$.
By Lemma \ref{lem:properties3}, for any
$\bxi\in \Xi_2(\GV)\cap\Xi_2(\GW)$ we have
\begin{equation*}
|\bxi_{\GV}|< 2 \rho^{\a_m},\ \
|\bxi_{\GW}|< 2\rho^{\a_m}.
\end{equation*}
Therefore, by Lemma \ref{projections:lem},
for the projection onto the subspace $\GA = \GV + \GW$ we have
\begin{equation*}
|\bxi_{\GA}| \ll \rho^{\a_m} r^{d^2}.
\end{equation*}
Since $\GV$ and $\GW$ are strongly distinct,
we have $p:=\dim \GA > m$, and hence, by \eqref{eq:condition1},
for sufficiently large $\rho$ the
right hand side is bounded above by $\rho^{\a_p}$,
which implies the proclaimed inclusion
in view of the definition \eqref{xi1:eq}.

Suppose that $n=m=1$, so that by Lemma \ref{lem:properties3}
\begin{equation*}
|\bxi_{\GV}|,\ |\bxi_{\GW}|< 2 \rho^{\frac{\a_1+\a_2}{2}}.
\end{equation*}
Lemma \ref{projections:lem} gives again that
$|\bxi_{\GA}| \ll \rho^{\frac{\a_1+\a_2}{2}} r^{d^2}$.
Since  $p:=\dim \GA= 2$, by \eqref{eq:condition1},
for sufficiently large $\rho$ the
right hand side is bounded above by $\rho^{\a_2}$,
which implies \eqref{strongly:eq} again.
\end{proof}

Lemma \ref{xi12:lem} has a number of very useful consequences.
First of all note that in the definition \eqref{xi:eq} we could have written
$\GV\subset\GW$ without changing the set $\Xi(\GV)$. Precisely, we have the following
lemma.

\begin{lem}\label{equiv:lem}
Let
\begin{equation}\label{xi11:eq}
\Xi'(\GV):=\Xi_2(\GV)\setminus\bigcup
\limits_{m>n}\bigcup\limits_{\GW\in\CV(m), \GV\subset\GW}\Xi_2(\GW).
\end{equation}
Then for sufficiently large $\rho$, we have
$\Xi(\GV)=\Xi'(\GV)$ for all $\GV\in\CV(n)$,
$n\le d-1$.
\end{lem}

\begin{proof} Let $\GV\in \CV(r, n)$.
Clearly, $\Xi(\GV)\subset\Xi'(\GV)$.
Let us prove the opposite inclusion. On the r.h.s. of
the formula \eqref{xi:eq} one can replace $\Xi_2(\GW)$ by $\Xi_2(\GW)\cap\Xi_2(\GV)$.
For strongly distinct $\GW$ and $\GV$, from Lemma \ref{xi12:lem} we obtain that
$$
\Xi_2(\GW)\cap\Xi_2(\GV)\subset\Xi_2(\GW+\GV).
$$
For $m > n$ this entails the inclusion
\begin{align*}
\bigcup\limits_{\GW\in\CV(m)}\Xi_2(\GW)\cap\Xi_2(\GV)
\subset &\
\biggl(\bigcup\limits_{\GW\in\CV(m), \GV\not\subset\GW}\Xi_2(\GW+\GV)\biggr)
\bigcup
\biggl(
\bigcup_{\GW\in\CV(m), \GV\subset\GW} \Xi_2(\GW)\biggr)\\[0.2cm]
\subset &\ \bigcup_{p\ge m}\bigcup_{\GW\in\CV(p), \GV\subset\GW} \Xi_2(\GW),
\end{align*}
so that $\Xi(\GV)\supset \Xi'(\GV)$,
as required.
\end{proof}

The next lemma shows that \underline{all} resonant sets $\Xi(\GV)$ are non-empty.
It is given here for the sake of completeness, and it
will not be used in the subsequent argument.

\begin{lem}
 Let $\GV\in\CV(n), n\le d$.
For any $R\gg \rho$ and sufficiently large $\rho$, we have
\begin{equation}\label{nonempty:eq}
\volume \bigl(\Xi(\GV)\cap B(R)\bigr)\gg \rho^{n\a_n} R^{d-n}.
\end{equation}
\end{lem}

\begin{proof} Let us fix a subspace $\GV\in\CV(n)$.
By definition \eqref{xi1:eq},
\[
V_0(\rho, R):=
\volume\bigl(\Xi_2(\GV)\cap B(R)\bigr)
\ge \volume \bigl(\Xi_1(\GV)\cap B(R)\bigr)
\gg \rho^{n\a_n} R^{d-n}.
\]
For  $n=d$ we have $\GV = \R^d$ and $\Xi_2(\GV) = \Xi(\GV)$, so the above lower bound
yields \eqref{nonempty:eq}.

Suppose that $n \le d-1$.
By Lemma \ref{lem:properties3}, for any $\GV\in\CV(n)$ and  any
$\GW\in \CV(m), m = n+1, \dots, d$, such that $\GV\subset\GW$,
we have
\begin{equation*}
\volume \bigl(\Xi_2(\GW)\cap \Xi_2(\GV)\cap B(0, R)\bigr)
\ll \rho^{ n\a_n} \rho^{(m-n)\a_m}R^{d-m}.
\end{equation*}
If $n = d-1$, then definition \eqref{xi:eq} ensures that
\begin{align*}
\volume \bigl(\Xi(\GV)\cap B(R)\bigr)
\gg &\ V_0(\rho, R) - C\rho^{(d-1)\a_{d-1} +\a_d}\\[0.2cm]
\gg &\ \rho^{(d-1)\a_{d-1}} R \bigl( 1 -C \rho^{\a_d - 1}\bigr)
\gg \rho^{(d-1)\a_{d-1}} R.
\end{align*}
For $n\le d-2$ we recall that
the number of lattice subspaces $\GW\in\mathcal V(m), m\le d-1,$
is less than $Cr^{d^2}$ with
a universal constant $C$. This leads to the estimate
\begin{align*}
V(\rho, R):=\volume \bigcup_{m\ge n+1}
&\ \bigcup_{\substack{\GW\in\CV(m)\\ \GV\subset\GW}}
\bigl(\Xi_2(\GW)\cap\Xi_2(\GV)\cap B(R)\bigr)\\
\ll &\  r^{d^2} \sum_{m=n+1}^{d-1}\rho^{ n\a_n} \rho^{(m-n)\a_m}R^{d-m}
+  \rho^{n\a_n} \rho^{(d-n)\a_d}\\[0.2cm]
\ll &\ \rho^{ n\a_n} R^{d-n}  \biggl[\sum_{m=n+1}^{d-1}
r^{d^2}\rho^{(m-n)\a_m}R^{n-m} + \rho^{(d-n)\a_d} R^{n-d}\biggr].
\end{align*}
Therefore, under the condition $R\gg \rho$, using \eqref{eq:condition1}, we obtain
that
\begin{equation*}
\frac{V(\rho, R)}{V_0(\rho, R)}
\ll \sum _{m=n+1}^{d} \rho^{(m-n)(\a_d-1)}\to 0, \rho\to\infty.
\end{equation*}
Now the required estimate \eqref{nonempty:eq} follows
from Lemma \ref{equiv:lem}.
\end{proof}

\begin{lem}\label{properties2:lem}
Let $\GV\in\CV(n)$ and $\GW\in\CV(m)$ be two strongly distinct
lattice subspaces. Then for sufficiently large $\rho$
we have $\Xi(\GV)\cap \Xi_2(\GW) = \varnothing$.
In particular, if $\bth\notin\GV$, then $\Xi(\GV)\cap \L(\bth) = \varnothing$.
\end{lem}

\begin{proof}
By Lemma \ref{xi12:lem}, for sufficiently large $\rho$,
$$
\Xi(\GV)\cap \Xi_2(\GW)\subset \Xi_2(\GV)\cap \Xi_2(\GW)\subset\Xi_2(\GU),\
\GU = \GV+\GW.
$$
Since $\dim\GU > n$, by definition \eqref{xi:eq}, the left hand side
is empty, as required.

As $\L(\bth) = \Xi_1(\GV(\bth))\subset \Xi_2(\GV(\bth))$, the second statement
follows immediately.
\end{proof}

\begin{cor}\label{properties2:cor}
Let $\GV\in\CV(n)$.
If $\bxi\in\Xi(\GV)$, then $\BUps(\bxi)\subset \Xi(\GV)$, and
for sufficiently large
$\rho$ also $\SN(\bxi)\subset\GV$.
\end{cor}

\begin{proof} The inclusion $\BUps(\bxi)\subset\Xi(\GV)$ immediately follows
from \eqref{xi2:eq}.

To prove the second statement, it suffices to show that
if $\boldeta\in\BUps_{\bth}(\bxi)$ with some $\bth\in\T_r$, then
$\bth\in\GV$. The assumption $\boldeta\in\BUps_{\bth}(\bxi)$
means, in particular, that
 $\bxi\in\L(\bth)$. By Lemma \ref{properties2:lem} the latter inclusion
is possible only if $\bth\in \GV$.
\end{proof}

Now we are in position to prove that the constructed sets $\Xi(\GV)$ satisfy
the conditions \eqref{exhaust:eq}, \eqref{disjoint:eq} and \eqref{span:eq}:

\begin{thm}\label{partition:thm}
Let $\a_0, \a_1, \dots, \a_d$ and $\varkappa$ be some numbers satisfying
\eqref{eq:condition1}.
Then for sufficiently large $\rho$  the collection of  sets
$\Xi(\GV)\subset\R^d, \GV\in \mathcal W(r)$, defined by the equalities
\eqref{xi1:eq}, \eqref{xi:eq}, satisfies the properties
\eqref{exhaust:eq}, \eqref{disjoint:eq} and \eqref{span:eq}.
\end{thm}

\begin{proof} Proof of \eqref{exhaust:eq}.
Let us prove that
\begin{equation}\label{xixi1:eq}
\Xi^{(n)} : = \bigcup_{m\ge n} \bigcup_{\GV\in\CV(m)} \Xi(\GV)
= \bigcup_{m\ge n} \bigcup_{\GV\in\CV(m)} \Xi_2(\GV)=: \Xi_2^{(n)},
\end{equation}
for all $n = 0, 1, 2, \dots, d$.
If $n=d$, then $\Xi^{(d)} = \Xi(\R^d) = \Xi_2(\R^d) = \Xi_2^{(d)}$.
Suppose that $0\le n \le d-1$.
By definition \eqref{xi:eq} and backward induction,
\begin{equation*}
\Xi_2^{(n)} = \bigcup_{\GV\in\CV(n)} \Xi_2(\GV)\bigcup \Xi_2^{(n+1)}
= \bigcup_{\GV\in\CV(n)} \Xi(\GV)\bigcup \Xi_2^{(n+1)}
= \bigcup_{\GV\in\CV(n)} \Xi(\GV)\bigcup \Xi^{(n+1)} = \Xi^{(n)}.
\end{equation*}
Therefore $\R^d = \Xi_2(\GX)\subset \Xi_2^{(0)} = \Xi^{(0)}$, as claimed.

Proof of \eqref{disjoint:eq}.
Let $\GV_1$ and $\GV_2$ be distinct lattice subspaces, $\dim\GV_j=p_j$,
$j = 1, 2$.
If $\GV_1\subset \GV_2$, then $p_1<p_2$, and it follows from Definition \eqref{xi:eq}
that
\begin{equation*}
\Xi(\GV_1)\subset \Xi_2(\GV_1)\setminus\Xi(\GV_2),
\end{equation*}
so that $\Xi(\GV_1)\cap\Xi(\GV_2)=\varnothing$.
If $\GV_1$ and $\GV_2$ are strongly distinct, then the required result follows
from Lemma \ref{properties2:lem}.

Proof of \eqref{span:eq}. See Corollary \ref{properties2:cor}.
\end{proof}

 \subsection{Scaling properties of the sets $\BUps(\bxi)$}

\begin{lem}\label{lem:upsilon}
Suppose that  $\bxi\in\Xi(\GV)$ and $\bxi + \bnu\in\Xi(\GV)$
with some $\bnu\in\GV^\perp$.
Then for sufficiently large $\rho$,
\begin{enumerate}
\item\label{upsilon:item1}
$\SN(\bxi) = \SN(\bxi+\bnu)$,
\item\label{upsilon:item2}
For $\bxi\in \Xi(\GV)$
the set $\SN(\bxi)$ depends only on the vector $\bxi_{\GV}$.
\end{enumerate}
\end{lem}

\begin{proof} By Corollary \ref{properties2:cor}
$\SN(\bxi), \SN(\bxi+\bnu) \subset\GV$, so Part \ref{upsilon:item1}
follows from Lemma \ref{trans:lem}.
Statement \ref{upsilon:item2} is a rephrased statement \ref{upsilon:item1}.
\end{proof}

\begin{lem}\label{perpendicular:lem}
Suppose that $\bxi\in\Xi(\GV)$ with some $\GV\in\CV(n), n\le d$, and
$\bxi_t = \bxi + t\bnu, \bnu = \bxi_{\GV^\perp}$ with $t\ge 0$.
Then for sufficiently large $\rho$ the vector
$\bxi_t\in \Xi(\GV)$ for all $t\ge 0$.
\end{lem}

We precede the proof with another Lemma:

\begin{lem}\label{coupl:lem}
Suppose that $\bxi\in\Xi(\GV)$ with some $\GV\subset\CV(n), n\le d$,  and
$\bxi_t = \bxi + t\bnu, \bnu = \bxi_{\GV^\perp}$ with $t\ge 0$.
 Then for sufficiently large $\rho$ we have $\SN(\bxi) = \SN(\bxi_t)$ for all $t\ge 0$,
and, in particular, $\SN(\bxi_t)\subset \GV$.
\end{lem}

\begin{proof} For $n = d$ the result is trivial, so we assume that
$n\le d-1$. We consider the case $n>0$; the case $n=0$ is similar, and we leave it to the reader.

Let us fix a $t >0$.
By Corollary \ref{properties2:cor}, $\SN(\bxi)\subset\GV$, so that
$\SN(\bxi) \subset \SN(\bxi_t)$ by Lemma \ref{trans:lem}.

In order to prove the opposite inclusion,
it suffices to prove that
$\BUps_{\bth}( \bxi_t) - t\bnu\subset\BUps_{\bth}(\bxi)$
for each $\bth\in\T_r^0$ such that
$\BUps_{\bth}(\bxi_t)\not=\varnothing$.
If $\bth = \mathbf 0$, then the above assertion is obvious.
Suppose that $\bth\not = \mathbf 0$.
Let $\bmu_t\in\BUps_{\bth}(\bxi_t)$,
so we need to show that  $\bmu :=  \bmu_t - t\bnu\subset \BUps_{\bth}(\bxi)$.

First, we prove that $\bth\in\GV$.
Suppose the converse, i.e. that $\bth\notin\GV$,
and define $\GW = \GV + \GV(\bth)\in\CV(n+1)$.
We may assume that $|\bxi_{\GW}|\ge \rho^{\a_{n+1}}$. Indeed, otherwise
we would have
$|\bxi_{\GW}|<\rho^{\a_{n+1}}$,
which would imply that $\bxi\in\Xi_1(\GW)\subset\Xi_2(\GW)$, but
the latter is impossible in view of the condition $\bxi\in\Xi(\GV)$. Since
\begin{equation}\label{aux:eq}
|\bxi_{\GV}|\le 2\rho^{\frac{\a_n+\a_{n+1}}{2}},\
\end{equation}
see Lemma \ref{lem:properties3}, we can
claim by virtue of Lemma \ref{projections:lem} and
conditions \eqref{eq:condition1} that
\begin{equation*}
|\bxi_{\bth}|\ge r^{-d} |\bxi_{\GW}| - |\bxi_{\GV}| \gg \rho^{\a_{n+1}} r^{-d}.
\end{equation*}
Consequently,
\begin{equation*}
|\bnu_{\bth}| = |\bxi_{\GV^\perp}\cdot \normal(\bth)|\ge |\bxi_{\bth}| - |\bxi_{\GV}|
\gg \rho^{\a_{n+1}} r^{-d}.
\end{equation*}
Therefore,
\begin{align*}
| (\bxi_t)_{\bth}| =  &\ |\bxi_{\bth} +
t \bnu_{\bth}|
= |(\bxi_{\GV})_{\bth} + (t+1) \bnu_{\bth}| \\[0.2cm]
\ge &\ (t+1)|\bnu_{\bth}| - |\bxi_{\GV}|
\ge \rho^{\a_{n+1}}r^{-d}
\ge \rho^{\a_1},\ \ t\ge 0,
\end{align*}
so that $ \bxi_t\notin\L(\bth), t\ge 0$,
which contradicts
the fact that $\bxi_t\lra\bmu_t \mod \bth$. Thus, $\bth\in\GV$.

As $\bnu\in\GV^\perp$, the inclusion $\bmu_t\in\BUps_{\bth}(\bxi_t)$ implies
$\bmu\in\BUps_{\bth}(\bxi)$ by Lemma \ref{trans:lem}. Therefore,
$\BUps_{\bth}(\bxi_t)=\BUps_{\bth}(\bxi)+t\bnu$, which in turn implies
$\BUps(\bxi_t)=\BUps(\bxi)+t\bnu$ and $\SN(\bxi_t)=\SN(\bxi)$.
The inclusion $\SN(\bxi_t)\subset\GV$ follows from $\SN(\bxi)\subset\GV$.
\end{proof}

\begin{proof}[Proof of Lemma \ref{perpendicular:lem}]
As in the previous proof assume that $n\le d-1$.
Let us fix a $t >0$.
Since $\bxi\in\Xi(\GV)\subset\Xi_2(\GV)$, one can find a vector
$\boldeta\in\Xi_1(\GV)\cap\BUps(\bxi)$.
As $\bnu\in\GV^\perp$,
this implies that $ \boldeta_t := \boldeta + t\bnu\in\Xi_1(\GV)$ and
by Lemma \ref{trans:lem},
$\bxi_t\in\BUps(\boldeta_t)\subset\Xi_2(\GV)$.
Thus, in view of Lemma \ref{equiv:lem}
it remains to prove that $\bxi_t\not\in\Xi_2(\GW)$
for any $\GW\supsetneq\GV$.
Suppose the contrary, i.e. for some $\GW\in\CV(m)$, $m>n$, $\GW\supset\GV$,
we have $\bxi_t\in\Xi_2(\GW)$.
We show that under this condition
we have $\bxi\in\Xi_2(\GW)$, which would contradict the
assumption $\bxi\in\Xi(\GV)$.

Denote $\GU:=\GW\ominus\GV$ ($\GU$ does not have to
be a lattice $r$-subspace).
By Definition \eqref{xi1:eq} there is a vector
$\bmu_t\in\Xi_1(\GW)$, resonant congruent to $\bxi_t$.
By Lemma \ref{coupl:lem}, $\SN(\bxi_t) = \SN(\bxi)\subset\GV$,
so that $(\bmu_t)_{\GU} = (\bxi_t)_{\GU}$.
Therefore, for the vector  $\bmu = \bmu_t - t\bnu$
we also have $ \bmu_{\GU} =  \bxi_{\GU}$.
This implies that
\begin{align*}
| (\bmu_t)_{\GW}|^2 = &\ |\bmu_{\GW} + t\bnu_{\GW}|^2 =
|\bmu_{\GW} +  t \bxi_{\GU}|^2 = |\bmu_{\GW} +  t \bmu_{\GU}|^2\\[0.2cm]
= &\  |\bmu_{\GV} + (t +1)\bmu_{\GU}|^2
= |\bmu_{\GV}|^2 + (t+1)^2 |\bmu_{\GU}|^2
\ge  \ |\bmu_{\GW}|^2, t\ge 0.
\end{align*}
By definition \eqref{xi1:eq}, $|(\bmu_t)_{\GW}|< \rho^{\a_m}$, and hence
$|\bmu_{\GW}|< \rho^{\a_m}$,
i.e. $\bmu\in\Xi_1(\GW)$.
On the other hand, since $ \bmu_t\in\BUps(\bxi_t)$, Lemma \ref{coupl:lem}
entails that $\bmu\in\BUps(\bxi)$, and thus $\bmu\in\Xi(\GV)$. This contradicts
the inclusion $\bmu\in\Xi_1(\GW)$, by definition of $\Xi(\GV)$.
Consequently, $\bxi_t\notin\Xi_2(\GW)$ for any
$\GW\supset\GV$, and the proof of the Lemma is complete.
\end{proof}


\section{Invariant subspaces for the ``gauged" operator}

The resonant sets $\Xi(\GV)$ are designed to describe the invariant subspaces
of the periodic PDO's having the form
\begin{equation}\label{model:eq}
A = H_0 + B^o + B^\flat
\end{equation}
with the symbols $h_0(\bxi) = |\bxi|^{2m}$ and $b\in \BS_{\a}(w)$, where
 $\a \in\R$, $w(\bxi) = \lu \bxi\ru^{\b}$, $\b\in (0, 1)$.
 By \eqref{decay:eq} and \eqref{subord:eq},
 \begin{equation}\label{bm:eq}
 |\BD_{\bxi}^{\mathbf s} b(\bth, \bxi)|
 + |\BD_{\bxi}^{\mathbf s} \hat b^o(\bth, \bxi)|
 + |\BD_{\bxi}^{\mathbf s} b^{\flat}(\bth, \bxi)| \ll \1 b\1^{(\a)}_{l, s}
 \lu\bxi\ru^{(\a - |\mathbf s|)\b}\lu \bth\ru^{-l},
 \end{equation}
 for all $\mathbf s$.
 We always assume that \eqref{alm1:eq} is satisfied, so that
 \begin{equation}\label{alm:eq}
 2m > \a\b,\ 2m-1 > \b(\a-1), 2m-2 > \b(\a-2).
 \end{equation}
 This guarantees that the symbol $b$ and its first two derivatives
grow slower than the principal symbol $h_0$ and its corresponding derivatives
respectively.

In order to use the resonant sets $\Xi(\GV)$ constructed previously, set
\begin{equation*}
\a_1 = \b,
\end{equation*}
and assume that the condition \eqref{eq:condition1} is satisfied.
In addition to  the symbol \eqref{flat:eq}, for any lattice subspace $\GV\in\CV(n)$,
$n=1, 2, \dots, d$ we define
\begin{equation}\label{bflatv:eq}
b_{\GV}^{\flat}(\bx, \bxi; \rho) = \frac{1}{\sqrt{\dc(\SG)}}
\sum_{\bth\in\T_r\cap\GV}
\hat b(\bth, \bxi; \rho)\z_{\bth}(\bxi; \rho^\b)
e_{\bth}(\bxi) e^{i\bth \bx}.
\end{equation}
It is clear that the above symbol retains from $b^{\flat}$ only the Fourier coefficients
with $\bth\in\GV$. Introduce also the notation for the  appropriate reduced version of the
model operator \eqref{model:eq}:
\begin{equation*}
A_{\GV} = H_0 + B^o + B_{\GV}^{\flat},\ B_{\GV}^{\flat} = \op(b_{\GV}^{\flat}).
\end{equation*}
Recall (see Subsect. \ref{fibre:subsect})
that for any set $\CC\subset\R^d$ we denote by $\CP(\CC)$ the operator
$\chi(\BD; \CC)$, where $\chi(\ \cdot\ ; \CC)$ is the characteristic function
of the set $\CC$. Accordingly, for the operators in the Floquet decomposition
acting on the torus, we define $\CP(\bk; \CC)$ to be $\chi(\BD+\bk; \CC)$.

In what follows we still apply Convention \ref{uniform:conv}, and add to
it one more rule: the estimates which we obtain are also uniform in the symbol $b$, satisfying
the condition $\1 b\1^{(\a)}\ll 1$.

\begin{lem}
Let $b$ be as above. Then
for sufficiently large $\rho$, and any $\GV\in\CV(n)$ we have
\begin{equation}\label{inv:eq}
B^\flat\CP\bigl(\Xi(\GV)\bigr) = B^{\flat}_{\GV}\CP\bigl(\Xi(\GV)\bigr)
= \CP\bigl(\Xi(\GV)\bigr)B^{\flat}_{\GV}\CP\bigl(\Xi(\GV)\bigr),
\end{equation}
and
\begin{equation}\label{inv0:eq}
B^\flat(\bk) \CP(\bk; \BUps(\bmu))
= \CP(\bk; \BUps(\bmu))B_{\GV}^\flat(\bk) \CP(\bk; \BUps(\bmu)),
\end{equation}
for any $\bk\in\CO^\dagger$ and any $\bmu\in\Xi(\GV)$ with $\{\bmu\} = \bk$.
\end{lem}

\begin{proof}
Assume without loss of generality that
$b^\flat$ has only one non-zero Fourier coefficient, i.e.
\begin{equation}\label{odin:eq}
b^{\flat}(\bx, \bxi; \rho) = \frac{1}{\sqrt{\dc(\SG)}}
\hat b(\bth, \bxi; \rho)\z_{\bth}(\bxi; \rho^\b)
e_{\bth}(\bxi) e^{i\bth \bx},
\end{equation}
so that $b_{\GV}^{\flat} = 0$ if $\bth\notin\GV$. According to \eqref{floquet:eq},
\begin{align}
(B^{\flat}(\bk) u)(\bx) = &\ \frac{1}{\dc(\SG)}
\sum_{\bm\in\SG^\dagger}
\hat b^\flat (\bth, \bm+\bk) e^{i(\bm+\bth)\bx} \hat u(\bm),\notag\\[0.2cm]
\bigl(B^\flat(\bk) \CP(\bk; \BUps(\bmu))u\bigr)(\bx)
=   &\ \frac{1}{\dc(\SG)}
\sum_{\bm: \bm+\bk\in\BUps(\bmu)}
\hat b^\flat (\bth, \bm+\bk) e^{i(\bm+\bth)\bx} \hat u(\bm), \label{floquetups:eq}
\end{align}
for any $u\in\plainL2(\Td)$.
Observe that by virtue of \eqref{phizeta:eq}
for any $\bxi:=\bm+\bk\in\supp\z_{\bth}(\ \cdot\ ; \rho^\b)$ we have
\begin{equation*}
|\bxi\cdot\bth|
\le |\bth(\bxi+\bth/2)| + |\bth|^2/2\le (\rho^{\b}/2+ r/2)|\bth|<
\rho^\b |\bth|,
\end{equation*}
that is $\bxi\in\L(\bth)$,
and a similar calculation shows that
$\bxi+\bth\in\L(\bth)$ as well.
By Lemma \ref{properties2:lem},
$\BUps(\bmu)\cap \L(\bth)  = \varnothing$, if $\bth\notin\GV$, so
it follows from \eqref{floquetups:eq}
that
\[
B^\flat(\bk) \CP(\bk; \BUps(\bmu)) = B^\flat_{\GV}(\bk) \CP(\bk; \BUps(\bmu)) =0,
\ \ \  \textup{if} \ \ \bth\notin\GV.
 \]
In the case $\bth\in\GV$, by Definition \ref{reachability:defn}, the points
$\bxi:=\bm+\bk\in\BUps(\bmu)$ and $\bxi+\bth$ are $\bth$-resonant congruent,  so that
$\bxi+\bth\in\BUps(\bmu)$. This completes the proof of \eqref{inv0:eq}.

Using \eqref{equiv:eq} we get from \eqref{inv0:eq}:
\begin{equation*}
B^\flat(\bk) \CP(\bk; \Xi(\GV))
= \CP(\bk; \Xi(\GV))B_{\GV}^\flat(\bk) \CP(\bk; \Xi(\GV)).
\end{equation*}
Taking the direct integral in $\bk$ yields \eqref{inv:eq}.
\end{proof}

\subsection{Operator $A$ in the invariant subspaces}
Due to the properties \eqref{exhaust:eq} and
\eqref{disjoint:eq}, the formulas \eqref{inv:eq} and \eqref{inv0:eq}
imply the following orthogonal decomposition for the Floquet fibres $A(\bk)$:
\begin{equation}\label{orhogonal:eq}
A(\bk) = \underset{\GV\in \mathcal W(r)}{\bigoplus} A(\bk; \Xi(\GV))
= \underset{\GV\in \mathcal W(r)}{\bigoplus}
\bigoplus_{\substack{\bmu \in \Xi(\GV)/{\lra}\\ \{\bmu\} = \bk}}
A_{\GV}(\bk; \BUps(\bmu)).
\end{equation}
Since $\card\BUps(\bmu)<\infty$, see \eqref{card:eq}, for each $\bmu\in \Xi(\GV)$
the operator
$A_{\GV}(\bk; \BUps(\bmu))$
is finite dimensional. In the basis
$$
E_{[\bmu]+\bm}(\bx),\ \bm \in \SN(\bmu)
$$
(see definition \eqref{exp:eq}),
of the subspace $\GH\bigl(\bk; \BUps(\bmu)\bigr)$,
the operator $A(\bk; \BUps(\bmu))$ reduces to the matrix $\CA(\bmu)$ with
the entries
\begin{equation}\label{ca:eq}
\CA_{\bm, \bn}(\bmu) = \frac{1}{\sqrt{\dc(\SG)}}\
\hat a (\bm  - \bn, \bmu + \bn; \rho),\ \bm, \bn\in  \SN(\bmu).
\end{equation}
Denote by $\l_j(\CA(\bmu))$, $j = 1, 2, \dots, N(\bmu) = \card \SN(\bmu)$ the
eigenvalues of the matrix $\CA(\bmu)$, arranged in descending order.
It is easy to check that the matrices $\CA(\bmu)$ and $\CA(\bmu')$
with $\bmu'\in\BUps(\bmu)$, are unitarily equivalent,
so that the eigenvalues do not depend
on the choice of $\bmu$, but only on the set $\BUps(\bmu)$.

\begin{lem}\label{twosided:lem}
Let $\l_j(\CA(\bmu))$
be the eigenvalues introduced above. Then for sufficiently large $\rho$,
for all $|\bmu|\gg \rho$, and for all $ j = 1, 2, \dots, N(\bmu)$ one has
\begin{equation*}
\l_j(\CA(\bmu))\asymp \min_{\boldeta\in\BUps(\bmu)}|\boldeta|^{2m}
\asymp
\max_{\boldeta\in\BUps(\bmu)}|\boldeta|^{2m}
\asymp|\bmu|^{2m},
\end{equation*}
uniformly in $\bmu$.
\end{lem}

\begin{proof}
The operator $A$ has the form $H_0+ b^{o, \flat}$, and since
$\a\b < 2m$ (see \eqref{alm:eq}),
by Lemma
\ref{formbound:lem}, the perturbation
$b^{o, \flat}$ is infinitesimally $H_0$-bounded,
so that $cH_0 - \tilde C\le A\le CH_0 + \tilde C$ with some
positive constants $C, c, \tilde C$.
Therefore, the same bounds hold for the
fibers $H_0(\bk)$ and $A(\bk)$. As a consequence,
the restriction of both operators to
the subspace $\GH(\bk, \BUps(\bmu))$
satisfy the same inequalities:
$$
c H_0(\bk; \BUps(\bmu)) - \tilde C\le A (\bk; \BUps(\bmu))
\le C H_0 (\bk; \BUps(\bmu)) + \tilde C.
$$
Now the claimed inequalities follows from Lemma \ref{bups:lem}.
\end{proof}

If $\BUps(\bmu)$ is non-critical (see Definition \ref{chain:defn}),
the set $\SN(\bmu)$
remains constant in a neighbourhood of $\bmu$.
Since the entries of the matrix $\CA$ depend continuously on $\bmu$, we conclude that
the eigenvalues $\l_j(\CA(\ \cdot\ ))$
are continuous in a neighbourhood of such a point $\bmu$.
Moreover, by virtue of Lemma \ref{lem:upsilon}, for any $\bmu\in\Xi(\GV)$
the set $\SN(\bmu)$
remains constant if $\bmu_\GV$ is kept constant, and
hence it makes sense  to study the eigenvalues as functions of the component
$\bnu = \bmu_{\GV^{\perp}}$. Define the matrix
\begin{equation*}
\tilde\CA(t) = \CA(\bmu_{\GV} + t\normal(\bnu)),\ \normal(\bnu) = \frac{\bnu}{|\bnu|},
\end{equation*}
with a real-valued parameter $t\ge t_0 := |\bnu|$. By \eqref{ca:eq}
the entries of this matrix are
\begin{equation}\label{cb:eq}
\tilde\CA_{\bm, \bn}(t) = \frac{1}{\sqrt{\dc(\SG)}}
\hat a(\bm - \bn, \bmu_{\GV} + \bn + t\normal(\bnu); \rho),\ \bm, \bn\in\SN(\bmu).
\end{equation}
By Lemma \ref{perpendicular:lem} the
matrix $\tilde \CA$ is well-defined on the interval $[t_0, \infty)$.

\begin{lem}\label{radial:lem}
Let \eqref{alm:eq} be satisfied.
Suppose that  $\bmu\in\Xi(\GV)$ and $|\bmu|\asymp\rho$.
Then
\begin{equation}\label{radial1:eq}
\l_j(\tilde\CA(t_2)) - \l_j(\tilde\CA(t_1))\asymp
 \rho^{2m-1} (t_2-t_1),
\end{equation}
for any $t_1, t_2 \asymp t_0,\  t_0 \le t_1 <t_2$,
uniformly in $j = 1, 2, \dots, \SN(\bmu)$, $\bmu $ and $\GV$.
\end{lem}

\begin{proof} By Lemma \ref{lem:properties3}, $|\bmu_{\GV}|\le 2\rho^{\a_{d}}$, so
that $t_0=|\bnu|\asymp |\bmu|\asymp \rho$.
By the elementary perturbation theory, it would suffice to establish for the matrix
\begin{equation*}
\tilde\CA(t_1, t_2) = \tilde\CA(t_2) - \tilde\CA(t_1)
\end{equation*}
the relation
\begin{equation*}
(\tilde\CA(t_1, t_2) u, u)\asymp \rho^{2m-1} (t_2 - t_1)\|u\|^2,\ \
t_1, t_2\asymp \rho, t_2 > t_1\ge t_0,
\end{equation*}
for all $u\in\GH$.
The entries of this matrix are
\begin{equation*}
\int_{t_1}^{t_2} \mathcal Y_{\bm, \bn} (t) dt,\ \
\mathcal Y_{\bm, \bn} (t) = \frac{d}{dt} \tilde\CA_{\bm, \bn}(t).
\end{equation*}
We show that the matrix $\mathcal Y(t)$
satisfies
\begin{equation}\label{caly:eq}
(\mathcal Y(t) u, u)\asymp\rho^{2m-1} \|u\|^2,\ t\asymp\rho,
\end{equation}
for all $u\in \GH$, uniformly in $\bmu$,\ $\GV$, and the symbol $b$.
Denote
\begin{equation*}
\bmu_t = \bmu_{\GV} + t \normal(\bnu), \bnu = \bmu_{\GV^\perp}.
\end{equation*}
Corollary \ref{properties2:cor} implies that $\SN(\bxi)\in\GV$. Therefore,
$\mathcal Y(t)$ is the sum of the matrix with diagonal entries
\begin{equation*}
 \frac{d}{dt} h_0( \bmu_t + \bm)
= 2m |\bmu_t + \bm|^{2m-2} t,\ \bm\in\SN(\bmu),
\end{equation*}
 and the matrix $\mathcal Z(t)$ with the entries
\begin{equation*}
\mathcal Z_{\bm, \bn}(t) = \frac{1}{\sqrt{\dc(\SG)}}
\left. \nabla_{\bxi}
 \hat b^{o, \flat}(\bm-\bn, \bxi; \rho)\cdot \normal(\bnu)
\right|_{\bxi = \bmu_t+\bn}.
\end{equation*}
By \eqref{card:eq},  $|\bm|\le  \rho^{\a_{d}}$, and hence
\begin{equation}\label{both:eq}
 \rho\ll t\le |\bmu_t+\bm|\ll \rho,\ \bm\in \SN(\bmu).
\end{equation}
Thus
\begin{equation}\label{radial:eq}
\frac{d}{dt} h_0( \bmu_t + \bm)\asymp \rho^{2m-1}.
\end{equation}
Also, by \eqref{bm:eq},
\begin{equation*}
|\mathcal Z_{\bm, \bn}(t) |\ll
\lu\bm-\bn\ru^{-l} |\bmu_t + \bn|^{(\a-1)\b} \ll \lu\bm-\bn\ru^{-l}  \rho^{(\a-1)\b},
\end{equation*}
for any $l>0$.
Assuming that $l >d$, from here we get:
\begin{align*}
\| \mathcal Z(t)\|\le &\
 \max_{\bn}\sum_{\bm}
|\mathcal Z_{\bm, \bn}(t)|
 \\[0.3cm]
\le &\ C_l \rho^{(\a-1)\b} \sup_{\bn}\max_{\bm} \lu \bm-\bn\ru^{-l}
\ll C \rho^{(\a-1)\b}.
\end{align*}
This, together with \eqref{alm:eq} and \eqref{radial:eq},
leads to \eqref{caly:eq}, which implies \eqref{radial1:eq}, as required.
\end{proof}


\section{Global description of the eigenvalues of operator $A(\bk)$}
\label{global:sect}

In this section we continue the study of the
discrete spectrum of the fibres $A(\bk)$. Our aim is to construct a
function $g: \R^d\to \R$, which establishes a one-to-one correspondence between the
points of $\R^d$ and the eigenvalues of $A(\bk)$. More precisely,
we seek a function $g$ such that
\begin{enumerate}
\item
for every $\bxi\in\R^d$ the value $g(\bxi)$ is an eigenvalue
of the operator $A(\bk),\ \bk = \{\bxi\}$,   and
\item
for every $j\in\mathbb N$
there exists a uniquely defined point $\bxi$ with $\{\bxi\} = \bk$ such that
$g(\bxi) = \l_j(A(\bk))$.
\end{enumerate}
In other words, we intend to label the eigenvalues of $A(\bk)$
by the points of the lattice $\SG^\dagger$, shifted by $\bk$. The construction of the
convenient function $g$ is conducted using the decomposition
\eqref{orhogonal:eq}, individually in the invariant subspaces generated by the sets
$\Xi(\GV)$.

We begin with the non-resonant set $\CB = \Xi(\GX), \GX = \{0\}$.
On the subspace $\CH(\Xi(\GX))$
the symbol of the operator $A$ is $\bx$-independent, and it is
$a^o(\bxi) = h_0(\bxi) + b^o(\bxi)$.
Therefore the eigenvalues
of the operator $A(\bk)$ are given by
$a^o(\bmu+\bk), \bmu\in\SG^\dagger, \bmu+\bk\in\Xi(\GX)$.
Clearly, it is natural to label the eigenvalues
by lattice points. Let us define
$$
g(\bxi) = a^o(\bxi),\ \bxi\in\Xi(\GX) = \CB.
$$
According to \eqref{bm:eq},
\begin{equation}\label{ginb:eq}
g(\bxi) = h_0(\bxi) + b^o(\bxi),\ \
\biggl|\frac{\partial}{\partial |\bxi|}b^o(\bxi)\biggr|
\le C\lu\bxi\ru^{(\a-1)\b},\ \ \
\biggl|\frac{\partial^2}{\partial |\bxi|^2}b^o(\bxi)\biggr|
\le C\lu\bxi\ru^{(\a-2)\b}
\end{equation}
for all $\bxi\in\CB$.

Suppose now that $\bxi\in\Xi(\GV)$
with some non-trivial lattice subspace $\GV$.
In view of \eqref{orhogonal:eq},
it suffices to define the function $g$ on the sets $\BUps(\bxi)$.
Let us label all numbers
$ \boldeta \in\BUps(\bxi)$, $\bxi\in\Xi(\GV)$
in the increasing order of their length $|\boldeta|$
by natural numbers from the set $\{1, 2, \dots, N(\bxi)\}$;
if there are two different vectors
$\boldeta, \tilde\boldeta\in\BUps(\bxi)$
with $|\boldeta|=|\tilde\boldeta|$, we
label them in the lexicographic order of
their coordinates, i.e. we put $\boldeta = (\eta_1, \eta_2, \dots, \eta_d)$
before $\tilde\boldeta = (\tilde\eta_1, \tilde\eta_2, \dots, \tilde\eta_d)$
if either $\eta_1<\tilde\eta_1$, or $\eta_1=\tilde\eta_1$
and $\eta_2<\tilde\eta_2$, etc. Such a labeling
associates in a natural way with each point $\boldeta\in\BUps(\bxi)$
a positive integer $\ell=\ell(\boldeta)\le N(\bxi)$. Clearly, this number does
not depend on the choice of the point $\bxi$ as long as $\bxi$ remains
within the same $\lra$-equivalence class.
In particular,
\begin{equation}\label{newindices}
|\boldeta|^{2m} = \l_{\ell(\boldeta)}( H_0(\bk; \BUps(\boldeta))).
\end{equation}
Now for every $\boldeta\in\R^d$ we define
\begin{equation*}
g(\boldeta):=\l_{\ell(\boldeta)}(\CA(\boldeta)),
\end{equation*}
where $\CA(\ \cdot\ )$ is the matrix defined in \eqref{ca:eq}.
Note that in view of Lemma \ref{twosided:lem}
\begin{equation}\label{twosided:eq}
g(\boldeta)\asymp |\boldeta|^{2m},\ |\boldeta|\gg \rho,
\end{equation}
for sufficiently large $\rho$.

In order to analyse the continuity of $g(\ \cdot\ )$, we assume that $\boldeta$
 is a non-critical point, i.e. the set $\SN(\ \cdot\ )$ remains
constant in a neighbourhood of $\boldeta$,
see Definition \ref{chain:defn}. Furthermore,
in the non-critical set, $\ell(\boldeta)$ remains constant,
if the point $\mathbf 0$ stays away from the Voronoi
hyper-planes associated with
 pairs of points from the set $\BUps(\boldeta)$. Recall that
 the Voronoi hyper-plane for a pair $\boldeta_1, \boldeta_2\in\R^d$
 is the set of all points $\bz\in\R^d$ such that $|\boldeta_1-\bz| = |\boldeta_2-\bz|$.
 Thus, the function $g(\ \cdot\ )$ is
 continuous on an open set of full measure in $\R^d$.

In each set $\Xi(\GV)$ the labeling function $\ell$
possesses the following important property.

\begin{lem}\label{lem:new1}
Let $\boldeta, \tilde\boldeta\in\Xi(\GV)$ satisfy
$\bnu := \tilde\boldeta-\boldeta\perp\GV$. Then
$\ell(\tilde\boldeta)=\ell(\boldeta)$.
\end{lem}

\begin{proof}
Recall that by Lemma \ref{lem:upsilon} $\BUps(\boldeta)+ \bnu = \BUps(\tilde\boldeta)$.
 Let $\boldeta\lra\bmu$, so that
$(\boldeta+\bnu)\lra(\bmu+\bnu)$.
By Corollary \ref{properties2:cor},
$\bmu_{\GV^\perp} = \boldeta_{\GV^\perp}$, and hence
the inequality $|\bmu|\le |\boldeta|$ holds or does not hold simultaneously
with $|\bmu+\bnu|\le |\boldeta + \bnu|$.
The same property is true for the coordinates
of the vectors involved. Therefore, $\ell(\boldeta) = \ell(\boldeta+\bnu)$.
\end{proof}

The next lemma allows us to establish smoothness of the function $g$
with respect to the variable $\boldeta_{\GV^\perp}$.

\begin{lem}
Let $\GV\in \CV(n), 1\le n\le d-1$ and let \eqref{alm:eq}
be satisfied.
Suppose that  $\boldeta\in\Xi(\GV)$ and $|\boldeta|\asymp\rho$, and let
$\bnu = \boldeta_{\GV^\perp}$.
Then for sufficiently large $\rho$, on the interval $[t_0, \infty),\ t_0 := |\bnu|$,
the function $\tilde g(t; \boldeta) := g(\boldeta_{\GV} + t\normal(\bnu))$  satisfies
\begin{equation}\label{radial2:eq}
 \tilde g(t_2, \boldeta) - \tilde g(t_1, \boldeta)\asymp \rho^{2m-1}(t_2-t_1),
\end{equation}
for any $t_1, t_2\in[ t_0, \infty)$ such that  $t_1 < t_2$ and  $t_1, t_2\asymp t_0$,
uniformly in $\boldeta\in\Xi(\GV)$ and $\GV$.
\end{lem}

\begin{proof}
Recall that in view of Lemma \ref{perpendicular:lem},
$\boldeta_{\GV} + t\normal(\bnu)\in\Xi(\GV)$
for all $t>t_0$.
Thus by Lemma \ref{lem:new1}, $\ell(\boldeta) = \ell(\boldeta_{\GV}+t\normal(\bnu))=:\ell$
for  $t\in [t_0, \infty)$. Therefore,
\begin{equation*}
\tilde g(t; \boldeta) = \l_\ell(\CA(\boldeta_{\GV} + t\normal(\bnu))).
\end{equation*}
It remains to apply Lemma \ref{radial:lem}.
\end{proof}

In order to study the global properties of
 $g$, note that by the above construction there is a bijection
 $J: \SG^\dagger + \bk\to \mathbb N$
such that
\begin{equation*}
g(\boldeta) = \l_{J(\boldeta)}(A(\bk)),\ \bk = \{\boldeta\},
\end{equation*}
for all $\boldeta\in\R^d$.
 For the following Lemma recall that the distance on the torus
$|\ \cdot\ |_{\Torus}$ is defined in \eqref{distorus:eq}.

\begin{lem}\label{partialofs2}
Let $\ba, \bb\in\R^d$ be such that $|\ba|\asymp \rho$.
Then there exists a vector $\bn\in\SG^\dagger$ such that
\begin{equation}\label{lip:eq}
|g(\bb+\bn) - g(\ba)|\ll \rho^{2m-1}|\bb-\ba|_{\Torus},
\end{equation}
for sufficiently large $\rho$.

Suppose in addition, that $\bm\in\SG^\dagger, \bm\not=\mathbf 0$,
is a vector such that $|\ba+\bm|\asymp \rho$.
Then there exists a $\tilde \bn\in\SG^\dagger$, such that $\bn\not=\tilde\bn$ and
\begin{equation}\label{lip1:eq}
|g(\bb+\tilde\bn) - g(\ba+\bm)|\ll \rho^{2m-1}|\bb-\ba|_{\Torus},
\end{equation}
for sufficiently large $\rho$.
\end{lem}

\begin{proof} As $\l_{J(\ba)}(A(\bk))= g(\ba)$,  by \eqref{twosided:eq},
we have
$$
|\l_{J(\ba)}(A(\bk))| \asymp \rho^{2m}.
$$
Denote $\bk=\{\ba\}$, $\bk_1 = \{\bb\}$. Recall that
the condition \eqref{beta:eq} is satisfied due to \eqref{alm:eq}, so
by Theorem \ref{unperturbed:thm}
\begin{equation}\label{lip2:eq}
|\l_{J(\ba)}( A(\bk)) - \l_{J(\ba)}( A(\bk_1))|
\ll \rho^{2m-1}|\bk-\bk_1|_{\Torus}=
\rho^{2m-1}|\bb-\ba|_{\Torus}.
\end{equation}
Let $\bp\in\R^d$ be a vector such that $\{\bp\} = \bk_1$ and
$g(\bp) = \l_{J(\ba)}(A(\bk_1))$. Now \eqref{lip2:eq} implies \eqref{lip:eq}
with $\bn = \bp-\bb$.

In order to prove \eqref{lip1:eq}, we use \eqref{lip2:eq} with $\ba+\bm$ instead of
$\ba$. Then, as above, one can find a vector $\tilde \bp$ such that
$\{\tilde \bp\} = \bk_1$ and $g(\tilde \bp) = \l_{J(\ba+\bm)}(A(\bk_1))$.
Since $J$ is one-to-one, we have $J(\ba+\bm)\not = J(\ba)$, and hence
$\bp \not =\tilde\bp$. As a consequence, $\tilde \bn  = \tilde \bp - \bb\not = \bn$,
as required, and \eqref{lip2:eq} again leads to \eqref{lip1:eq}.
\end{proof}


\section{Estimates of volumes}

In this section we continue the investigation of the operator of the form
\eqref{model:eq}, with a symbol $b\in \BS_{\a}(w)$,
$w(\bxi) = \lu \bxi\ru^\b$,
with parameters $\a, \b$,
satisfying the conditions \eqref{alm:eq}. Let $g:\R^d\to\R$ be the function
defined in the previous section, and let $\CB(\rho), \CD(\rho)$ and  $\tilde\CB(\rho)$
be the sets introduced  in \eqref{(non)resonant:eq} and \eqref{angular:eq} respectively.

Let $\de\in (0, \l/4]$, $\l = \rho^{2m}$, and let
\begin{equation}\label{ABD:eq}
\begin{cases}
\CA(\rho, \de)=\CA(g; \rho, \de):= g^{-1}([\la-\de,\la+\de]),\\[0.2cm]
\CB(\rho, \de)=\CB(g; \rho,\de):= \CA(\rho, \de)\cap\CB(\rho),\\[0.2cm]
\CD(\rho, \de)=\CD(g; \rho, \de): = \CA(\rho, \de)\cap\CD(\rho),\\[0.2cm]
\tilde\CB(\rho, \de)=\tilde\CB(g; \rho, \de):= \CA(\rho, \de)\cap\tilde\CB(\rho).
\end{cases}
\end{equation}
The estimates for the volumes of the above sets are very important for our argument.

\begin{lem}\label{lem:volumeCA} Let $A$ be the operator
\eqref{model:eq}, and let $\a, \b$ satisfy the conditions
\eqref{alm:eq}.
Then for any $\d \in (0,  \rho^{2m}/4]$ and for all sufficiently large $\rho$,
the following estimates hold
\begin{equation}\label{volume_tildeCB:eq}
\volume \tilde\CB (\rho, \d) \asymp \d \rho^{d-2m},
\end{equation}
and
\begin{equation}\label{eq:volumeCD}
\volume(\CD(\rho, \de))\ll \d\rho^{d-1-2m+\alpha_d}.
\end{equation}
Here $\a_d\in (0, 1)$
is the number defined together with $\a_1, \a_2, \dots, \a_{d-1}$
at the beginning of Subsection \ref{resonant:subsect}(see \eqref{eq:condition1}).
\end{lem}

Before proving the above lemma we find a convenient representation
of the set $\tilde \CB(\rho, \delta)$.
Since $\tilde\CB(\rho, \delta)\subset\CB(\rho)$, for all $\bxi\in \tilde \CB(\rho, \d)$
the function $g$ is defined by the
formula \eqref{ginb:eq}, and in particular, it is continuous.
For all $\boldsymbol\Om\in T(\rho)$ (see \eqref{surface:eq} for definition), we
introduce the subsets of the real line defined as follows:
\begin{equation}\label{iomd:eq}
I(\boldsymbol\Om; \rho, \d)
= \{t >0: \rho^{2m}-\d\le g(t\boldsymbol\Om) \le \rho^{2m}+\d \}.
\end{equation}
By \eqref{ginb:eq},  for $t\in I(\boldsymbol\Om; \rho, \d)$ we
have $\rho/2 <t < 2\rho$,
and hence $t\boldsymbol\Om\in \tilde\CB(\rho, \d)$.
If $\rho$ is sufficiently large, by virtue of \eqref{ginb:eq},
for these  values of $t$ the function $g(t\boldsymbol\Om)$ is strictly increasing,
and hence $I(\boldsymbol\Om; \rho, \d)$ is a closed interval.
Moreover, \eqref{ginb:eq}
implies the relation
\begin{equation}\label{iomd1:eq}
|I(\boldsymbol\Om; \rho, \d)|\asymp \d \rho^{1-2m}
\end{equation}
for its length, uniformly in $\boldsymbol\Om$. By construction,
\begin{equation}\label{tildecb:eq}
\tilde\CB(\rho, \d) = \underset{\boldsymbol\Om\in T(\rho)}
\bigcup I(\boldsymbol\Om; \rho, \d)\boldsymbol\Om.
\end{equation}

\begin{proof}[Proof of Lemma \ref{lem:volumeCA}]
In view of  Lemma
\ref{surface:lem} and of the bound \eqref{iomd1:eq},
we obtain from \eqref{tildecb:eq}:
\begin{equation*}
\volume \tilde\CB(\rho, \d)
= \int_{T(\rho)}\int_{I(\boldsymbol\Om; \rho, \d)}
t^{d-1} dt d\boldsymbol\Om \asymp \d \rho^{d-2m}.
\end{equation*}
This proves \eqref{volume_tildeCB:eq}.

Proof of \eqref{eq:volumeCD}. By definition \eqref{(non)resonant:eq}
and relation \eqref{exhaust:eq},
$$
\CD(\rho) = \underset{\GV\subset \CV(n), 1\le n\le d}{\bigcup}\Xi(\GV; \rho).
$$
Let us estimate the volume of each intersection $\Xi(\GV; \rho)\cap \CA(\rho, \d)$.
Since Lemma \ref{lem:properties3} implies that
$\Xi(\R^d)\cap\CA(\rho, \d) = \varnothing$,
we assume that $n\le d-1$.

For all $\boldsymbol\chi\in\GV$ and
$\boldsymbol\Om\in\GV^\perp, |\boldsymbol\Om| = 1$ denote
\begin{equation*}
S(\boldsymbol\chi, \boldsymbol\Om; \rho)
= \{t\ge 0: \boldsymbol\chi + t\boldsymbol\Om\in\Xi(\GV; \rho)\}.
\end{equation*}
According to Lemma \ref{perpendicular:lem}, this set
is either empty, or it is a half-line of the form $[t_0, \infty)$
or $(t_0, \infty)$ with some $t_0 \ge 0$.
Due to the estimate \eqref{proj:eq}, $S(\boldsymbol\chi, \boldsymbol\Om; \rho)$
is empty if $|\boldsymbol\chi|\ge 2\rho^{\a_{d-1}}$. From now on we assume that
$S(\boldsymbol\chi, \boldsymbol\Om; \rho)\not=\varnothing$,
so that $|\boldsymbol\chi|<2\rho^{\a_{d-1}}$. Consider the subset
\begin{equation}\label{S:eq}
S(\boldsymbol\chi, \boldsymbol\Om; \rho, \d)
= \{ t\in S(\boldsymbol\chi, \boldsymbol\Om; \rho):
\rho^{2m} - \d \le g(\boldsymbol\chi + t\boldsymbol\Om)
\le \rho^{2m} + \d\}.
\end{equation}
In view of \eqref{twosided:eq}, $t\asymp \rho$.
By \eqref{radial2:eq}, the function
$\tilde g(t) = g(\boldsymbol\chi + t\boldsymbol\Om)$ is strictly increasing and continuous,
and hence, $S(\boldsymbol\chi, \boldsymbol\Om; \rho, \d)$
is an interval.  The bound \eqref{radial2:eq} also guarantees the upper
bound
\begin{equation*}
|S(\boldsymbol\chi, \boldsymbol\Om; \rho, \d)| \ll \d\rho^{1-2m},
\end{equation*}
for the length of this interval, uniformly in $\boldsymbol\chi$ and
$\boldsymbol\Om$.
Now we can estimate the volume of the intersection:
\begin{align*}
\volume(\Xi(\GV; \rho)\cap \CA(\rho, \d))
= &\ \underset{|\boldsymbol\chi|<2\rho^{\a_{d-1}}}\int \ \
\underset{\mathbb S^{d-n-1}} \int\ \
\underset{ S(\boldsymbol\chi, \boldsymbol\Om; \rho, \d)}
\int t^{d-n-1} dt
d\boldsymbol\Om \ d\boldsymbol\chi\\[0.2cm]
\ll &\ \rho^{d-n-1}\underset{|\boldsymbol\chi|<2\rho^{\a_{d-1}}}\int \ \
\underset{\mathbb S^{d-n-1}} \int\ |S(\boldsymbol\chi, \boldsymbol\Om; \rho, \d)|
d\boldsymbol\Om \ d\boldsymbol\chi\\[0.2cm]
\ll  &\ \d\rho^{1-2m} \rho^{d-n-1} (\rho^{\a_{d-1}})^{n}
\ll  \d \rho^{d-1 - 2m + \a_{d-1}}.
\end{align*}
Recall that the number of distinct subspaces $\GV\subset\mathcal W(r)$
does not exceed $Cr^{d^2}$ with some universal constant $C$, so that
\begin{equation*}
\volume \CD(\rho, \d)\ll \d\rho^{d-1 - 2m + \a_{d-1}} r^{d^2}\ll \d\rho^{d-1 - 2m + \a_d},
\end{equation*}
where we have used the conditions \eqref{eq:condition1}.
Now \eqref{eq:volumeCD} is proved.
\end{proof}

The next estimate is more subtle:

\begin{thm}\label{volumeCB2a}
Let
$\CB(g; \rho, \d)$, $\d \in (0, \rho^{2m}/4]$,
be as defined in \eqref{ABD:eq}. Let $\ep >0$ be a fixed number.
If $\d\rho^{2-2m+2\ep}\to 0$ as
 $\rho\to\infty$, then
\begin{equation}\label{nokiss:eq}
\volume \Bigl(\CB(g; \rho, \d)\cap \bigl(\CB(g; \rho, \d) + \bb\bigr)\Bigr)
\ll \d^2 \rho^{4-4m + d+6\ep} + \d \rho^{1-2m - \ep(d-1)},
\end{equation}
uniformly in $\bb, |\bb|\gg 1$.
\end{thm}

This Theorem will be proved in the next section.
As an immediate consequence, we can write the following estimate:
\begin{equation}\label{sobranie:eq}
\volume \bigcup_{\bn\in\SG^\dagger\setminus\{0\}}
\Bigl(\CB(g; \rho, \d)\cap \bigl(\CB(g; \rho, \d) + \bn\bigr)\Bigr)
\ll \d^2 \rho^{4-4m + 2d+6\ep} + \d \rho^{1-2m+d - \ep(d-1)},
\end{equation}
valid under the condition $\d \rho^{2-2m+2\ep}\to 0$, $\rho\to\infty$.
Indeed, to get \eqref{sobranie:eq} from \eqref{nokiss:eq}
one notices that the union in the above estimate does not extend to
the lattice points $\bn$ such that $|\bn|\ge 3\rho$.

Another important ingredient is the following estimate on the volumes:

\begin{lem}
Let
$\CB(g; \rho, \d)$, $\CD(g; \rho, \d)$, $\d \in (0, \rho^{2m}/4]$,
be as defined in \eqref{ABD:eq}.
Let $\ep>0$ be some number. If $\d \rho^{2-2m+2\ep}\to 0$
as $\rho\to\infty$, then
\begin{align}\label{BD:eq}
\volume \bigcup_{\bn\in\SG^\dagger\setminus\{0\}}
\Bigl(\CB(g; \rho, \d)\cap &\ \bigl(\CD(g; \rho, \d) + \bn\bigr)\Bigr)\notag\\[0.2cm]
\ll &\ \d^2 \rho^{4-4m + 2d+6\ep} + \d \rho^{1-2m+d - \ep(d-1)}
 + \d\rho^{d-1-2m+\a_{d}}.
\end{align}

\end{lem}

\begin{proof} Let us split $\CD(\rho, \d)$ in three
disjoint sets:
\begin{align*}
\CD_0(\rho, \d) = &\ \{\bxi\in\CD(\rho, \d): \bxi+\bn\notin\CB(\rho, \d), \
\textup{for all}\ \ \bn\in\SG^\dagger\setminus\{\mathbf 0\}\},\\[0.2cm]
\CD_1(\rho, \d) = &\ \{\bxi\in\CD(\rho, \d): \\
&\ \textup{there exists a \underline{unique}}\ \
\bn=\bn(\bxi)\in\SG^\dagger\setminus\{\mathbf 0\}
\ \textup{such that} \
 \bxi+\bn\in\CB(\rho, \d)\},\\[0.2cm]
 \CD_2(\rho, \d) = &\ \CD(\rho, \d)\setminus
 \Bigl(\CD_0(\rho, \d)\bigcup\CD_1(\rho, \d)\Bigr).
\end{align*}
The definition of $\CD_0(\rho, \d)$ immediately implies that
 \begin{equation}\label{CD0:eq}
\CB(\rho, \de)\bigcap\Bigl(\bigcup_{\bn\in\SG^\dagger\setminus\{\mathbf 0\}}
\bigl(\CD_0(\rho, \de)+\bn\bigr)\Bigr)=\varnothing.
\end{equation}
For the set $\CD_2(\rho, \d)$ we have the inclusion
\begin{equation}\label{CD2:eq}
 \bigcup_{\bn\in\SG^\dagger\setminus\{\mathbf 0\}}
\bigl(\CD_2(\rho, \de)+\bn\bigr) \subset
\bigcup_{\bn\in\SG^\dagger\setminus\{ \mathbf 0\}}
\bigl(\CB(\rho, \de)+\bn\bigr).
\end{equation}
Indeed,for each $\bxi\in\CD_2(\rho, \d)$
there are at least two distinct lattice vectors $\bn_1, \bn_2\not = 0$
such that $\bxi + \bn_1\in\CB(\rho, \d)$ and $\bxi+\bn_2\in\CB(\rho, \d)$, so that
any lattice vector $\bm\not = 0$ is distinct either from $\bn_1$
or from $\bn_2$. Thus,  assuming for definiteness that $\bm\not = \bn_1$, we get
\begin{equation*}
\bxi+\bm = \bxi + \bn_1 + (\bm - \bn_1) \in \bigl(\CB(\rho, \d)+ \bm-\bn_1\bigr)
\subset
\bigcup_{\bn\in\SG^\dagger\setminus\{\mathbf 0\}}
\bigl(\CB(\rho, \de)+\bn\bigr).
\end{equation*}
This proves \eqref{CD2:eq}.

Now observe that by definition of
$\CD_1(\rho, \d)$
the sets $\CD_1(\rho, \d)\cap\bigl(\CB(\rho, \d)+\bn\bigr)$ are disjoint  for different
$\bn\in\SG^\dagger\setminus\{\mathbf 0\}$. Therefore
\begin{align*}
\volume \bigcup_{\bn\in\SG^\dagger\setminus\{\mathbf 0\}}
\Bigl(\bigl(\CD_1(\rho, \de)+\bn\bigr)\cap\CB(\rho, \d)\Bigr)
= &\ \sum_{\bn\in\SG^{\dagger}\setminus\{\mathbf 0\}}
\volume \Bigl(\CD_1(\rho, \de)\cap\bigl(\CB(\rho, \d)+\bn\bigr)\Bigr)\\[0.2cm]
\le &\ \volume \CD_1(\rho, \d)\le \volume\CD(\rho, \d).
\end{align*}
Together with \eqref{CD0:eq} and \eqref{CD2:eq} this produces the bound
  \begin{equation*}
  \volume \bigcup_{\bn\in\SG^\dagger\setminus\{\mathbf 0\}}
\Bigl(\CB(\rho, \d)\cap \bigl(\CD(\rho, \d) + \bn\bigr)\Bigr)
\le  \volume \CD(\rho, \d) +
\volume \bigcup_{\bn\in\SG^\dagger\setminus\{\mathbf 0\}}
\Bigl(\CB(\rho, \d)\cap \bigl(\CB(\rho, \d) + \bn\bigr)\Bigr).
  \end{equation*}
The estimate \eqref{BD:eq} follows from \eqref{sobranie:eq} and \eqref{eq:volumeCD}.
\end{proof}

The next section is devoted to the proof of Theorem \ref{volumeCB2a}.


\section{Estimates of volumes: part two}

\subsection{Results and preliminary estimates for the intersection volume}
Consider two continuous functions $g_j: \R^d\to\R$ such that $g_j(\bxi)\to\infty$,
$|\bxi|\to\infty$, $j = 1, 2$.
Our objective is to establish
upper bounds for the measure of the set
\begin{equation*}
\CX(g_1, g_2;\rho; \d; \bb_1, \bb_2)
: = \bigl(\CA(g_1; \rho; \d) + \bb_1\bigr)
\cap \bigl(\CA(g_2;\rho; \d) + \bb_2\bigr),
\end{equation*}
for arbitrary vectors $\bb_1, \bb_2\in\R^d$ such that $|\bb_1-\bb_2|\gg 1$,
and $\d\in (0, \rho^{2m}/4]$.
Clearly,
\[
\volume\bigl( \CX(g_1, g_2;\rho; \d; \bb_1, \bb_2) \bigr)
 = \volume\bigl( \CX(g_1, g_2;\rho; \d; \mathbf 0, \bb_2-\bb_1) \bigr),
\]
so that it suffices to study the set
\begin{equation}\label{bigx:eq}
\CX: = \CX(g_1, g_2;\rho; \d; \mathbf 0, \bb)
\end{equation}
with some $\bb\in\R^d$, $1\ll |\bb|\ll \rho$. Note that the condition
$|\bb|\ll \rho$ does not restrict generality, since for $|\bb|\ge 3\rho$
the set $\CX$ is empty.

Let us make more precise assumptions about the functions $g_1, g_2$.
Suppose that
\begin{equation}\label{gsmall:eq}
g_j(\bxi) = |\bxi|^{2m} + G_j(\bxi), \ G_j\in\plainC2(\R^d), \ j = 1, 2.
\end{equation}
Further conditions are imposed for the following range of values of $\bxi$:
\begin{equation}\label{range:eq}
|\bxi|\asymp\rho,\ \ |\bxi-\bb|\asymp\rho.
\end{equation}
The functions $G_j$'s are  assumed to satisfy the following
conditions:
\begin{equation}\label{cond1}
|G_j(\bxi)|\ll \rho^{\g},
\end{equation}
\begin{equation}\label{cond2}
|\nabla G_j(\bxi)|\ll \rho^{\s},
\end{equation}
\begin{equation}\label{cond3}
|\nabla^2 G_j(\bxi)|\ll \rho^{\omega},
\end{equation}
\begin{equation}\label{cond5}
|G_1(\bxi)-G_2(\bxi-\bb)|\ll |\bb|\rho^{\s},
\end{equation}
and
\begin{equation}\label{cond4}
|\nabla G_1(\bxi)-\nabla G_2(\bxi-\bb)|\ll |\bb|\rho^{\omega},
\end{equation}
for all $\bb, 1\ll |\bb|\ll \rho$,
with some $\g< 2m$, $\s<2m-1$, $\omega<2m-2$, for all $\bxi$ satisfying
\eqref{range:eq}. The constants in these estimates are allowed
to depend on the constant $C$ in \eqref{range:eq}.

\begin{thm}\label{volumeCB2}
Let two functions $g_1, g_2$ be as in \eqref{gsmall:eq}, and suppose
that the conditions \eqref{cond1}, \eqref{cond2}, \eqref{cond3},
\eqref{cond5} and \eqref{cond4} are satisfied. Then for any $\ep>0$,
if $\d\rho^{2-2m+2\ep}\to 0$, $\rho\to\infty$, then
\begin{equation}\label{nokiss1:eq}
\volume \CX(g_1, g_2; \rho, \d; \mathbf 0, \bb)
\ll \d^2 \rho^{4-4m+d+6\ep} + \d \rho^{1-2m - \ep(d-1)},
\end{equation}
uniformly in $\bb, 1\ll |\bb|\ll \rho$.
\end{thm}

Let us show how to derive Theorem \ref{volumeCB2a} from
Theorem \ref{volumeCB2}:

\begin{proof}[Proof of Theorem \ref{volumeCB2a}]
Extend the function $g$ from
the non-resonant set $\CB$ to the entire space $\R^d$ by the formula \eqref{ginb:eq},
and denote the new function by $g_1$. By \eqref{ginb:eq} and \eqref{alm:eq},
the functions $g_1$ and $g_2 = g_1$
satisfy the conditions \eqref{cond1}, \eqref{cond2}, \eqref{cond3},
\eqref{cond5}, \eqref{cond4}
with $\g = \a\b < 2m, \s = (\a-1)\b < 2m-1, \om = (\a-2)\b<2m-2$
for $\bxi$ in the range \eqref{range:eq}.

Then, clearly,
\begin{equation*}
\Bigl(\CB(g; \rho; \d)\cap \bigl(\CB(g;\rho; \d) + \bb\bigr)\Bigr)
\subset \Bigl(\CA(g_1; \rho; \d) \cap \bigl(\CA(g_1;\rho; \d) + \bb\bigr)\Bigr).
\end{equation*}
It remains to use Theorem \ref{volumeCB2} with $g_1 = g_2$.
\end{proof}

Let us concentrate on proving Theorem \ref{volumeCB2}.
Our first observation is that it suffices to
do it for $m = 1$. Indeed,
introducing functions
\begin{equation*}
\check g_j(\bxi) = \bigl(g_j(\bxi)\bigr)^{\frac{1}{m}},
\check G_j(\bxi) = \check g_j(\bxi) - |\bxi|^2,
\end{equation*}
we note that under the condition \eqref{range:eq} the functions
$\check G_j$ satisfy the bounds \eqref{cond1}-\eqref{cond4} with
the parameters
\begin{gather*}
\check \g = \g + 2 - 2m < 2,\ \check \s =
\max(\s+2-2m, \g+1-2m)<1,\\
\check \om = \max(\om+2-2m, \s+1-2m, \g-2m)<0.
\end{gather*}
One checks directly that
\begin{equation*}
\CX(g_1, g_2; \rho, \d; \mathbf 0, \bb)\subset
\CX(\check g_1, \check g_2; \rho, \check\d; \mathbf 0, \bb), \
\check\d = 2m^{-1}\d\rho^{2-2m},
\end{equation*}
for sufficiently large $\rho$. Moreover, the condition $\d\rho^{2-2m+2\ep}\to 0$
becomes $\check\d\rho^{2\ep}\to 0$.

\vskip 0.2cm

\noindent
\underline{Thus, from now on until the end of this section we assume that $m=1$.}

\vskip 0.2cm

\noindent
Due to the condition $\d \rho^{2\ep}\to 0$, we may assume that
$\d \in (0, 1]$. Now,
in view of \eqref{cond1} we have
\begin{equation}\label{inter:eq}
c_1\rho \le |\bxi|\le C_1\rho,\ \
 c_1\rho\le |\bxi-\bb|\le C_1\rho.
\end{equation}
Here, the constants $c_1, C_1, c_1< 1 < C_1$
can be chosen arbitrarily close to $1$, assuming that $\rho$ is sufficiently large.

Below, we denote
by $\phi(\bxi, \boldeta)\in [0, \pi]$ the angle between arbitrary
non-zero vectors $\bxi, \boldeta\in\R^d$.
A central role in the study of the set $\CX$ (see \eqref{bigx:eq})
is played by the angle $\phi(\bmu, \bmu-\bb)$ for the points
$\bmu\in\CX$. Let us establish some general
facts about this angle.
It is convenient to introduce new orthogonal
coordinates in $\R^d$ in the following way:
$\bxi = (\xi_1, \hat\bxi)$ with $\xi_1 = \bxi\cdot\normal(\bb)$ and
$\hat\bxi = \bxi_{\bb^\perp}$, so that $\bxi = \xi_1\normal(\bb) + \hat\bxi$.

\begin{lem}\label{polozh:lem} Let $m=1$, and
suppose that the functions $G_j, j=1, 2$ satisfy
conditions \eqref{cond1}, \eqref{cond5}, and that
$\d\in (0, 1]$, $1\ll |\bb|\ll \rho$.
Then $\phi(\bxi, \bxi-\bb)\gg |\bb|\rho^{-1}$ uniformly in
$\bxi\in\CX$.
\end{lem}

\begin{proof}
Denote for brevity $\phi = \phi(\bxi, \bxi-\bb)$ and
$\CX = \CX(\rho; \d, g_1, g_2; \mathbf 0, \bb)$.
For our purposes
we assume that the constants $c_1, C_1$ in \eqref{inter:eq} satisfy the
bound
\begin{equation}\label{c1:eq}
5c_1^2/4 \ge C_1^2.
\end{equation}
We consider separately two cases: $|\xi_1|<(\sqrt3/2) c_1\rho$ and
$|\xi_1|\ge (\sqrt3/2) c_1\rho$.

\underline{Case 1: $|\xi_1|<\sqrt3/2 c_1\rho$.} Denote
$\tilde\phi = \phi(\bb, \bxi)$.
Since $|\bxi|\ge c_1\rho $ and $|\xi_1|< \sqrt3/2 c_1\rho$, we have
$\tilde\phi\in (\pi/6, 5\pi/6)$, and hence $\sin\tilde\phi > 1/2$.
By the sine rule
\begin{equation*}
\frac{\sin\phi}{|\bb|} = \frac{\sin\tilde\phi}{|\bxi-\bb|},
\end{equation*}
which implies that
\begin{equation*}
\sin\phi = \frac{|\bb|}{|\bxi-\bb|} \sin\tilde\phi
\ge C_1^{-1}\frac{|\bb|}{2\rho}.
\end{equation*}
Thus $\phi\gg |\bb|\rho^{-1}$, as claimed.

\underline{Case 2: $|\xi_1|\ge (\sqrt3/2) c_1\rho$.}
Let us show first that
\bee\label{new*}
\xi_1 >0
\ene
 and
 \bee\label{new**}
 \xi_1 - |\bb|<0.
 \ene

Assume, on the contrary, that $\xi_1 \le 0$, so that
\begin{equation*}
|\bxi-\bb|^2  - |\bxi|^2 = (\xi_1-|\bb|)^2 - \xi_1^2 \ge  -2 \xi_1 |\bb|.
\end{equation*}
Then
\begin{equation*}
g_2(\bxi-\bb) - g_1(\bxi)
\ge  -2\xi_1 |\bb| - |G_1(\bxi) - G_2(\bxi-\bb)|
\end{equation*}
By condition \eqref{cond5},
\begin{equation*}
|G_1(\bxi) - G_2(\bxi-\bb)|\ll |\bb|\rho^{\s}.
\end{equation*}
Together with the assumption  $|\xi_1|\ge (\sqrt3/2) c_1 \rho$ this implies that
\begin{equation*}
g_2(\bxi-\bb) - g_1(\bxi)
\gg
\rho|\bb|.
\end{equation*}
This contradicts the condition $g_2(\bxi-\bb) - g_1(\bxi) < 2\d$, and hence
\eqref{new*} is satisfied.

Assume now that $\xi_1 - |\bb|\ge 0$, Then, similarly to the above argument,
\begin{equation*}
|\bxi|^2 - |\bxi-\bb|^2 = \xi_1^2 - (\xi_1-|\bb|)^2 = 2\xi_1 |\bb| - |\bb|^2
\ge \xi_1 |\bb|.
\end{equation*}
Thus, as above,
\begin{equation*}
g_1(\bxi) - g_2(\bxi-\bb)
\ge  \xi_1 |\bb| - |G_1(\bxi) - G_2(\bxi-\bb)|
\end{equation*}
By condition \eqref{cond5},
\begin{equation*}
|G_1(\bxi) - G_2(\bxi-\bb)|\ll |\bb|\rho^{\s}.
\end{equation*}
Together with the assumption  $|\xi_1|\ge (\sqrt3/2) c_1 \rho$ this implies that
\begin{equation*}
g_1(\bxi) - g_2(\bxi-\bb)
\gg
\rho |\bb|.
\end{equation*}
This contradicts the condition $g_1(\bxi) - g_2(\bxi-\bb) < 2\d$, and hence
\eqref{new**} is satisfied.

The next step is to show that $|\bb|\gg \rho$.
Indeed, it follows from \eqref{inter:eq} that
\begin{equation*}
|\hat\bxi|^2 = |\bxi|^2 - \xi_1^2\le \biggl( C_1^2  - \frac{3}{4} c_1^2\biggr)\rho^2,
\end{equation*}
which implies that
\begin{align*}
|\bb|^2 = &\ |\bxi|^2 +|\bxi-\bb|^2 - 2\bxi\cdot(\bxi-\bb)
\ge |\bxi|^2 +|\bxi-\bb|^2 - 2|\hat\bxi|^2\\[0.2cm]
\ge &\  2 \biggl(c_1^2 - C_1^2 + \frac{3}{4}c_1^2\biggr)\rho^2
= 2 \biggl(\frac{7}{4}c_1^2 - C_1^2\biggr)\rho^2\ge c_1^2\rho^2.
\end{align*}
Here we have used  \eqref{c1:eq} as well as \eqref{new*} and \eqref{new**}.
On the other hand,
\begin{align*}
|\bb|^2 = &\ |\bxi|^2 +|\bxi-\bb|^2 - 2\cos\phi |\bxi| \ |\bxi-\bb|\\[0.2cm]
= &\ \bigl(|\bxi|-|\bxi-\bb|\bigr)^2 + 2(1-\cos\phi) |\bxi| \ |\bxi-\bb|\\[0.2cm]
\le &\ (C_1-c_1)^2\rho^2 + 4 C_1^2\sin^2\bigl(\frac{\phi}{2}\bigr) \rho^2.
\end{align*}
Using this, together with the lower bound $|\bb|\ge c_1\rho$, we arrive at
\begin{equation*}
4C_1^2 \sin^2\bigl(\frac{\phi}{2}\bigr)\ge c_1^2 - (C_1-c_1)^2
= C_1(2c_1 - C_1)\ge \frac{1}{2} c_1 C_1.
\end{equation*}
This means that $\sin(\phi/2)\gg 1\gg |\bb|\rho^{-1}$, which means that
$\phi\gg |\bb|\rho^{-1}$, as claimed.

The proof of the Lemma is complete.
\end{proof}

The next result is proved for those $\bmu\in\CX$
which satisfy the relations
\begin{equation}\label{nuli:eq}
\nabla G_1(\bmu) = \nabla G_2(\bmu-\bb) = 0.
\end{equation}

\begin{lem}\label{kiss:lem}
Let $m=1$, and
let the functions $G_1, G_2\in\plainC2(\R^d)$ satisfy the conditions
\eqref{cond1},
\eqref{cond3}, and let $\d\in (0, 1]$, $1\ll |\bb|\ll \rho$.
Suppose that there exists a point
$\bmu\in \CX$
such that $\pi - \phi(\bmu, \bmu-\bb)\le  l \rho^{-1}$
with $ 0< l \le   1$,
and \eqref{nuli:eq} is satisfied.
Then under the condition $\d l^{-2}\to 0$, $\rho\to\infty$, we have
\begin{enumerate}
\item [1.]
$\CX \subset
\{\bxi\in\R^d: |\hat\bxi| < 4 C_1^2 c_1^{-1} l \}$, and
\item [2.]
$\volume \CX \ll \d l^{(d-1)} \rho^{-1}$.
\end{enumerate}
\end{lem}

\begin{proof} First note some useful inequalities for
$\bmu, \bb$.  Denote
$\phi_0 := \phi(\bmu, \bmu-\bb)$,\ $ \phi_1 := \phi(\bmu, \bb)$.
Since $\pi - \phi_0\le l\rho^{-1}$, we have $\phi_0 \ge \pi/2$,
so that $\cos\phi_0<0$ and $\cos\phi_1 >0$.
Recalling \eqref{inter:eq}, we conclude that
\begin{equation*}
|\bb|^2 = |\bmu|^2+ |\bmu-\bb|^2 - 2\cos\phi_0 |\bmu|\ |\bmu-\bb|\ge 2c_1^2\rho^2.
\end{equation*}
This also gives $|\bmu|< |\bb|$ and $|\bmu-\bb|<|\bb|$.
Furthermore, it follows from the sine rule that
\begin{equation*}
\sin\phi_1 = \sin\phi_0 \frac{|\bmu-\bb|}{|\bb|}
\le \frac{l}{\rho} \frac{C_1\rho}{ c_1 \rho}= \frac{C_1}{ c_1}
\frac{l}{\rho}.
\end{equation*}
This leads to the bounds
\begin{equation*}
|\hat\bmu|= |\bmu|\sin\phi_1 \le c_2l, \  c_2 = C_1^2 c_1^{-1},\ \
\mu_1 = |\bmu|\cos\phi_1> 0.
\end{equation*}
Similarly, by considering $\phi_2:=\phi(\bmu-\bb,\bb)$ we can prove
that $|\bb|-\mu_1>0$.

Now let $\bxi$ be an arbitrary element of $\CX$ and let us prove
that  $|\hat\bxi|< 4 c_2l$. Suppose that, on the contrary,
$|\hat\bxi|\ge 4 c_2l$,
and let $\boldeta = \bxi-\bmu$.  Clearly,
$|\hat\boldeta|\ge 3 c_2 l\ge 3|\hat\bmu|$, and hence
\begin{equation}\label{newtilde}
|\hat\bxi|^2 - |\hat\bmu|^2
= |\hat\boldeta|^2 + 2\hat\bmu\cdot \hat\boldeta
\ge |\hat\boldeta|(|\hat\boldeta|-2|\hat\bmu|)\ge\frac{|\hat\boldeta|^2}{3}.
\end{equation}
Let us now assume that $\eta_1\ge 0$. Then, 
combining \eqref{newtilde} with the identity
\begin{equation}\label{+:eq}
\xi_1^2 - \mu_1^2 = \eta_1^2 + 2\mu_1 \eta_1,
\end{equation}
we obtain
\begin{equation*}
|\bxi|^2 - |\bmu|^2\ge \frac{1}{3}|\boldeta|^2.
\end{equation*}
At the same time, due to \eqref{nuli:eq} and \eqref{cond3},
\begin{equation*}
|G_1(\bxi) - G_1(\bmu)|\ll |\boldeta|^2 \rho^{\om},
\end{equation*}
so that
\begin{equation*}
g_1(\bxi) - g_1(\bmu)\ge |\bxi|^2 - |\bmu|^2 - C|\boldeta|^2 \rho^{\om}
\ge \biggl(\frac{1}{3} - C\rho^{\om}\biggr)|\boldeta|^2
\gg  l^2,
\end{equation*}
and hence $g_1(\bxi)\ge g_1(\bmu)+ Cl^2 \ge \rho^2 - \d + Cl^2$. Since $\d l^{-2}\to 0 $
as $\rho\to\infty$, it follows that
$g_1(\bxi)> \rho^2+ \d$ for large $\rho$.
This means that $\bxi\notin \CA(g_1; \rho; \d)$, so that, by contradiction,
$|\hat\bxi|< 4 c_2l$.

Consider now the case $\eta_1<0$. Then instead of \eqref{+:eq} we use
\begin{equation*}
|\xi_1 - |\bb||^2 - |\mu_1 - |\bb||^2 = |\eta_1|^2 +
2(\mu_1 - |\bb|)\eta_1.
\end{equation*}
Since $\mu_1 - |\bb|<0$, combining this with \eqref{newtilde}, we obtain
\begin{equation*}
|\bxi-\bb|^2 - |\bmu-\bb|^2\ge \frac{1}{3}|\boldeta|^2.
\end{equation*}
At the same time, due to \eqref{nuli:eq} and \eqref{cond3},
\begin{equation*}
|G_2(\bxi-\bb) - G_2(\bmu-\bb)|\ll |\boldeta|^2 \rho^{\om},
\end{equation*}
so that
\begin{equation*}
g_2(\bxi-\bb) - g_2(\bmu-\bb)\ge |\bxi-\bb|^2 - |\bmu-\bb|^2 - C|\boldeta|^2 \rho^{\om}
\gg  l^2,
\end{equation*}
and hence $g_2(\bxi-\bb)\ge g_2(\bmu-\bb)+ Cl^2
\ge \rho^2 - \d + Cl^2$. Since $\d l^{-2}\to 0 $
as $\rho\to\infty$, it follows that
$g_2(\bxi-\bb)> \rho^2+ \d$ for large $\rho$.
This means that $\bxi\notin (\CA(g_2; \rho; \d)+\bb)$, so that, by contradiction,
$|\hat\bxi|< 4 c_2l$.
This completes the proof of Part 1.

Let us fix $\hat\bxi: |\hat\bxi|<c_1\rho/2$.
For all $\bxi=(\xi_1,\hat\bxi), |\bxi| > c_1\rho$
we have
\[
\xi_1^2 = |\bxi|^2 - |\hat\bxi|^2 \ge c_1^2 \rho^2
- \frac{1}{4}c_1^2 \rho^2= \frac{3}{4}c_1^2 \rho^2.
\]
By conditions \eqref{cond3} and \eqref{nuli:eq},
\[
|\nabla G_1(\bxi)| = |\nabla G_1(\bxi) - \nabla G_1(\bmu)|\ll
|\boldeta| \rho^{\om}\ll \rho^{1+\om},
\]
so that
\begin{equation*}
|\partial_{\xi_1} g_1(\xi_1, \hat\bxi)|
= |2\xi_1 + \partial_{\xi_1} G_1(\xi_1, \hat\bxi)| \gg \rho.
\end{equation*}
In particular, the function $g(\ \cdot\ , \hat\bxi)$ is strictly monotone,
and the set
\begin{equation*}
I_{\hat\bxi} = \{\xi_1: |g_1(\xi_1, \hat\bxi) - \rho^2|\le \d\}
\end{equation*}
is a closed interval of length $|I_{\hat\bxi}|\ll \d\rho^{-1}$. By Part 1,
\begin{equation*}
\volume\bigl(\CX(g_1, g_2; \rho, \d, \bb)\bigr)
\le \int_{|\hat\bxi|< 4c_2l} |I_{\hat\bxi}| d\hat\bxi\ll \d l^{d-1}\rho^{-1},
\end{equation*}
as claimed.
\end{proof}

Let us now consider the case when
$\phi_0 = \phi(\bmu, \bmu-\bb)$ is separated away from $\pi$.
The following elementary observation will be useful:

\begin{lem}\label{lem:angle1}
Let $\bn_1, \bn_2\in\R^2$ be two unit vectors. Then for any other unit vector $\bn\in\R^2$
one has
\begin{equation*}
|\bn\cdot\bn_1|^2 + |\bn\cdot\bn_2|^2\ge 1 - |\bn_1\cdot\bn_2|.
\end{equation*}
\end{lem}

\begin{proof} The result follows from the following elementary trigonometric
calculation for arbitrary $\psi, \phi\in\R$:
\begin{align*}
\cos^2\psi + \cos^2(\psi - \phi)
= &\ 1 + \frac{1}{2}\biggl(\cos(2\psi) + \cos(2(\psi-\phi))\biggr)\\[0.3cm]
= &\ 1 + \cos(2\psi - \phi)\cos\phi\ge 1 - |\cos\phi|.
\end{align*}
\end{proof}

\bel\label{volumeCB2.2} Let $m=1$ and $\d\in (0, 1], 1\ll |\bb|\ll \rho$.
Assume that two functions $G_1, G_2\in\plainC2(\R^d)$
satisfy the conditions \eqref{cond1}, \eqref{cond3}, \eqref{cond5}.
Suppose that there exists a point $\bmu\in\CX$
such that  \eqref{nuli:eq} holds and
$\pi - \phi_0\ge l \rho^{-1}, \phi_0 := \phi(\bmu, \bmu-\bb)$, with some
$0<l\le 1$. Then for any $\ep_2 >0$
\bee\label{area1}
\area( \CX\cap B(\bmu, l\rho^{-\ep_2}))\ll l^{d-2} \rho^{-2-\ep_2(d-2)}
\frac{\d^2}{\sin^2\phi_0}.
\ene
\enl

\begin{proof}
First of all, notice that assumptions of this Lemma together with Lemma \ref{polozh:lem}
imply
\bee\label{newhat}
\sin\phi_0\gg l\rho^{-1}.
\ene
Let $\bxi\in B(\bmu, l\rho^{-\ep_2})$, so that $|\bxi|\gg \rho$.
Due to \eqref{nuli:eq} and \eqref{cond3}
we have
\begin{gather*}
|\nabla G_1(\bxi)|
= |\nabla G_1(\bxi) - \nabla G_1(\bmu)|\ll l \rho^{\om-\ep_2},\\[0.2cm]
|\nabla G_2(\bxi-\bb)|
= |\nabla G_2(\bxi-\bb) - \nabla G_2(\bmu-\bb)|\ll l \rho^{\om-\ep_2},
\end{gather*}
for all $\bxi\in B(\bmu, l \rho^{-\ep_2})$, and hence,
by elementary trigonometric argument,
we have the following upper bounds:
\begin{equation}\label{eq:angle2}
\begin{cases}
\phi(\bxi, \bmu) = O(l \rho^{-1-\ep_2 }),\\[0.2cm]
\phi(\bxi-\bb, \bmu-\bb) = O(l \rho^{-1-\ep_2 }),\\[0.2cm]
\phi(\bxi,\nabla g_1(\bxi))=\phi(\bxi,2\bxi+\nabla G_1(\bxi))=
O(l \rho^{-1+\om-\ep_2}),\\[0.2cm]
\phi(\bxi-\bb,\nabla g_2(\bxi-\bb)) =
\phi(\bxi-\bb,2(\bxi - \bb)+\nabla G_2(\bxi-\bb))=
O(l \rho^{-1+\om-\ep_2}).
\end{cases}
\end{equation}
Since $\nabla g_1(\bmu) = 2\bmu, \nabla g_2(\bmu-\bb) = 2(\bmu - \bb)$ and
$\om < 0$, it follows that
\begin{equation*}
\phi(\nabla  g_1(\bmu), \nabla g_1(\bxi)) = O(l\rho^{-1-\ep_2}),\ \
\phi(\nabla  g_2(\bmu-\bb), \nabla g_2(\bxi-\bb)) = O(l\rho^{-1-\ep_2}).
\end{equation*}
The above bounds imply that
\[
|\phi(\bxi, \bxi-\bb) - \phi_0|\ll l\rho^{-1-\ep_2}, \forall
\bxi\in B(\bmu, l\rho^{-\ep_2}).
\]
Since $l\le 1$, together with \eqref{newhat} this means that
\[
\sin \phi(\bxi, \bxi-\bb)\gg \sin\phi_0 \gg l\rho^{-1}.
\]
Thus,  the vectors $\bxi, \bb$
span a two-dimensional space.
>From now on we represent every vector $\bxi\in B(\bmu, l\rho^{-\ep_2})$ as
follows: $\bxi = (z, \Theta_1, \hat{\boldsymbol\Theta})$, where
$z= |\bxi|$, $\Theta_1\in[0, \pi]$ is the angle between $\bxi$ and $\bb$, and
$\hat{\boldsymbol\Theta} = \hat\bxi|\hat\bxi|^{-1}\in\mathbb S^{d-2}$.
We denote the plane spanned by $\bxi$ and $\bb$ by $\GV_{\hat{\boldsymbol\Theta}}$.

Let $\bxi \in \GV_{\hat{\boldsymbol\Theta}}$ with some
$\hat{\boldsymbol\Theta}\in\mathbb S^{d-2}$.
By Lemma \ref{lem:angle1}, for any unit vector $\be\in\GV_{\hat{\boldsymbol\Theta}}$
 we have
\begin{equation*}
|\normal(\bxi)\cdot\be|^2 + |\normal(\bxi-\bb)\cdot\be|^2\ge
1-|\cos\phi(\bxi, \bxi-\bb)|
\ge \frac{1}{2}\sin^2\phi(\bxi, \bxi-\bb)\gg \sin^2\phi_0,
\end{equation*}
which implies that at least one of the following estimates
hold:
\begin{enumerate}
\item[(i)] $|\be\cdot\normal(\bxi)|\gg \sin\phi_0$,\\
or
\item[(ii)] $|\be\cdot\normal(\bxi-\bb)|\gg \sin\phi_0$.
\end{enumerate}
Since $\sin\phi_0\gg l\rho^{-1}$, in view of \eqref{eq:angle2},
we also have that at least one of the following estimates
hold for all $\bxi\in B(\bmu, l\rho^{-\ep_2})\cap\GV_{\hat{\boldsymbol\Theta}}$:
\begin{equation*}
\begin{cases}
| \be\cdot\normal(\nabla g_1(\bxi))|\gg \sin\phi_0,\  \   \textup{for case (i)};\\[0.2cm]
|\be\cdot\normal(\nabla g_2(\bxi-\bb))|\gg \sin\phi_0,\  \   \textup{for case (ii)}.
\end{cases}
 \end{equation*}
Let us fix another vector $\boldeta\in B(\bmu, l\rho^{-\ep_2})
\cap\GV_{\hat{\boldsymbol\Theta}}$ and use (i) or (ii) for
$$
\be = \normal(\boldeta - \bxi) = \normal((\boldeta-\bb) - (\bxi-\bb) ).
$$
If the condition (i) holds, then
\begin{equation}\label{cond(i)}
|g_1(\boldeta) -  g_1(\bxi)|
\gg |\boldeta - \bxi|\ \inf_{\boldsymbol\chi} |
\be \cdot\nabla g_1(\boldsymbol\chi)|\gg
|\boldeta - \bxi|\rho \sin\phi_0,
\end{equation}
where the infimum is taken
over $\boldsymbol\chi \in B(\bmu, l\rho^{-\ep_2})
\cap\GV_{\hat{\boldsymbol\Theta}}$.
Here we have used that $|\nabla g_1(\boldsymbol\chi)|\gg\rho$.
Analogously, if the condition (ii) holds, then
\begin{equation}\label{cond(ii)}
| g_2(\boldeta-\bb) - g_2(\bxi-\bb)|
\gg   |\boldeta - \bxi|\rho \sin\phi_0.
\end{equation}
Suppose in addition that $\bxi,\boldeta\in \CX\cap B(\bmu, l\rho^{-\ep_2})$.
Then, if condition (i) holds, by definition of $ \CX$ we get from
\eqref{cond(i)}:
\begin{equation*}
|\boldeta - \bxi|\ll |g_1(\boldeta)-  g_1(\bxi)| \frac{1}{\rho\sin\phi_0}
\ll \de \frac{1}{\rho\sin\phi_0}.
\end{equation*}
Similarly, if
condition (ii) holds, then \eqref{cond(ii)} implies
\begin{equation*}
|\boldeta - \bxi|\ll | g_2(\boldeta-\bb)- g_2(\bxi-\bb)|  \frac{1}{\rho\sin\phi_0}
\ll \de \frac{1}{\rho\sin\phi_0}.
\end{equation*}
Therefore, for any $\bxi\in  \CX\cap B(\bmu, l\rho^{-\ep_2})
\cap\GV_{\hat{\boldsymbol\Theta}} $ we have $ \CX\cap B(\bmu, l\rho^{-\ep_2})
\cap\GV_{\hat{\boldsymbol\Theta}}
\subset B(\bxi, C\d(\rho\sin\phi_0)^{-1})$, and hence
\begin{equation}\label{2dimarea:eq}
\volume_2( \CX\cap B(\bmu, l \rho^{-\ep_2})
\cap\GV_{\hat{\boldsymbol\Theta}})\ll\de^2\frac{1}{\rho^2\sin^2\phi_0},
\end{equation}
where $\volume_2$ denotes the area on the plane $\GV_{\hat{\boldsymbol\Theta}}$.

Integrating in the coordinates
$(z, \Theta_1,\hat{\boldsymbol\Theta})$, introduced previously,
we can estimate:
\begin{equation*}
\area( \CX\cap B(\bmu, l \rho^{-\ep_2}))
\ll l^{d-2} \rho^{-\ep_2(d-2)} \sup_{
 \hat{\boldsymbol\Theta}}\area_2( \CX\cap B(\bmu, l \rho^{-\ep_2})
\cap\GV_{\hat{\boldsymbol\Theta}}),
\end{equation*}
which completes the proof upon using \eqref{2dimarea:eq}.
\end{proof}

\subsection{Proof of Theorem \ref{volumeCB2} for $m=1$}
In the previous Lemmas the volume of $\CX$ was estimated under the assumptions
that \eqref{nuli:eq} holds.
Although this condition cannot be expected to hold for arbitrary functions $g_1, g_2$,
one can always satisfy \eqref{nuli:eq} locally, by "adjusting" the functions $g_1, g_2$
appropriately.  Afterwards, one can use the above Lemmas. This strategy is implemented
in the proof of Theorem \ref{volumeCB2}.

\begin{proof}[Proof of Theorem \ref{volumeCB2} for $m=1$]
Pick a vector $\bmu\in\CX$, and denote
$\bv_1:=\frac{1}{2}(\nabla G_1(\bmu))$
and $\bv_2:=\frac{1}{2}(\nabla G_2(\bmu-\bb))$.
In view of \eqref{cond4},
\begin{equation}\label{v1v2:eq}
|\bv_1 - \bv_2|\ll |\bb|\rho^{\om}.
\end{equation}
Define $\tilde g_j(\bxi): = g_j(\bxi-\bv_j)$
and $\tilde G_j(\bxi): = \tilde g_j(\bxi) - |\bxi|^2$, so that
\begin{align*}
\tilde G_j(\bxi)= &\ |\bxi-\bv_j|^2-|\bxi|^2
+G_j(\bxi-\bv_j)=-2\bxi\cdot\bv_j +|\bv_j|^2+G_j(\bxi-\bv_j),\\
\nabla \tilde G_j(\bxi) = &\ -2\bv_j + \nabla G_j(\bxi-\bv_j).
\end{align*}
It is easily checked that the functions $\tilde G_j$ satisfy the conditions
\eqref{cond1},
\eqref{cond2} and
\eqref{cond3}
with new parameters $\tilde\g = \max(\g, 1+\s)$,
$\tilde\s = \max(1+\om, \s)$, $\tilde\om = \om$.
We also introduce
\begin{equation*}
\ba := \bb + \bv_1 - \bv_2
\end{equation*}
and notice that
in view of \eqref{v1v2:eq}, we have $ |\ba|\asymp|\bb|$.
Now, writing
\begin{equation*}
\tilde G_1(\bxi) - \tilde G_2(\bxi-\ba)
= -2\bxi\cdot(\bv_1-\bv_2) - 2\ba\cdot\bv_2 + |\bv_1|^2-|\bv_2|^2
+ G_1(\bxi-\bv_1) - G_2(\bxi-\bv_1-\bb),
\end{equation*}
and using \eqref{cond5} and \eqref{cond4}, \eqref{v1v2:eq}, we make the following estimate:
\begin{equation*}
|\tilde G_1(\bxi) - \tilde G_2(\bxi-\ba)|
\ll |\bb| \rho^{1+\om} + |\ba|\rho^\s + |\bb|\rho^{\s+\om} + |\bb|\rho^{\s}
\ll |\ba| \rho^{\tilde\s}.
\end{equation*}
Thus, the condition \eqref{cond5}
is also satisfied with $\bb$ replaced by $\ba$.
Moreover, by definition of $\tilde G_j$,
\begin{equation*}
\nabla\tilde G_1(\bnu) = 0,\ \ \nabla\tilde G_2(\bnu - \ba)=0, \
\bnu = \bmu + \bv_1,
\end{equation*}
so that \eqref{nuli:eq} is fulfilled.
By definition of $\tilde g_1, \tilde g_2$,
\begin{equation*}
\CX(\rho; \d;, g_1, g_2; \mathbf 0,  \bb)
= \CX(\rho; \d; \tilde g_1, \tilde g_2; -\bv_1, -\bv_2+\bb),
\end{equation*}
and consequently,
\begin{equation*}
\volume \bigl( \CX(\rho; \d; g_1, g_2; \mathbf 0,  \bb)\bigr)
=
\volume\bigl(
\CX(\rho; \d; \tilde g_1, \tilde g_2; \mathbf 0,  \ba)
\bigr).
\end{equation*}
Denote $\tilde \CX = \CX(\rho; \d, \tilde g_1, \tilde g_2; \mathbf 0, \ba)$.

Now, depending on the value of $\phi(\bnu, \bnu-\ba)$ we use Lemma
\ref{kiss:lem} or Lemma \ref{volumeCB2.2} with $l = \rho^{-\ep}, \ep >0$
and $\ep_1 = \ep_2 = \ep$.
Note that Lemma \ref{kiss:lem} can be used since
$\d l^{-2} = \d\rho^{2\ep}\to 0$ as $\rho\to\infty$.

If  $\pi-\phi(\bnu, \bnu-\ba)\le \rho^{-\ep-1}$, then by Lemma \ref{kiss:lem}
\begin{equation}\label{kiss:eq}
\volume \tilde\CX\ll \d\rho^{-1-\ep(d-1)}.
\end{equation}
Assume now that for all points $\bmu\in\CX$ we have the bound
$\pi-\phi(\bnu, \bnu-\ba)\ge \rho^{-\ep-1}$.
It follows again from definition of $\tilde g_j$ that
\begin{equation*}
\volume \bigl( \CX(\rho; \d; g_1, g_2; \mathbf 0,  \bb)
\cap B(\bmu, \rho^{-2\ep})\bigr)
=
\volume\bigl(
\CX(\rho; \d; \tilde g_1, \tilde g_2; \mathbf 0,  \ba)
\cap B(\bnu, \rho^{-2\ep})
\bigr).
\end{equation*}
Since $\pi-\phi(\bnu, \bnu-\ba)\ge \rho^{-\ep-1}$,
according to Lemma \ref{volumeCB2.2},
\begin{equation*}
\volume\bigl(\CX\cap B(\bmu, \rho^{-2\ep})\bigr)
= \volume\bigl(\tilde\CX\cap B(\bnu, \rho^{-2\ep})\bigr)
\ll  \d^2 \rho^{-2\ep(d-3)}.
\end{equation*}
As $\CX\subset B(0, 3\rho)$,
one needs $\ll \rho^{d(1+2\ep)}$  balls of radius $\rho^{-2\ep}$
to cover $\CX$. Thus,
\[
\volume\CX\ll \d^2 \rho^{-2\ep(d-3)} \rho^{d(1+2\ep)} = \d^2\rho^{d+6\ep}.
\]
Adding this bound with \eqref{kiss:eq} produces \eqref{nokiss1:eq}.
\end{proof}

As explained earlier, Theorem \ref{volumeCB2} for $m=1$ implies itself for
all $m >0$.


\section{Proof of the Bethe-Sommerfeld Conjecture}

In this section, we prove the Main Theorem \ref{main:thm}.
We do it in a few steps. First we prove it for the model operator
$A$ defined \eqref{model:eq} with conditions \eqref{alm:eq} satisfied.
After that we invoke Theorem \ref{reduction:thm}, which states that the original operator
$H$ can be reduced to the model operator up to controllable error terms.
At the second step we show that these errors do not destroy the spectral
band overlap, obtained for the
model operator.

\subsection{Theorem \ref{main:thm} for the model operator \eqref{model:eq}}
Our proof of the spectral band overlap for the operator $A$ relies on the
following elementary Intermediate Value
Theorem type result for the function $g(\bxi)$ defined
in Section \ref{global:sect}. As before we assume that $\l = \rho^{2m}$.

\begin{lem}\label{lem:simple1}
Let $\bxi=\bxi(t)\subset\CB, t\in [t_1, t_2], t_1 < t_2$, be a continuous path.
Suppose that
$g(\bxi(t_1))\le \la - \d$,
$g(\bxi(t_2))\ge \la+\d$ with some $\d\in (0, \l/4)$,
and for each $t\in [t_1, t_2]$ the number $g(\bxi(t))$ is
a simple eigenvalue
of $A(\bk), \bk = \{\bxi(t)\}$. Then there exists
a $t_0\in (t_1, t_2)$ such that $\l = g(\bxi(t_0))$,
so that $\l\in\s(A)$. Moreover, $\z(\l; A) \ge \d$.
\end{lem}

\begin{proof}
Since $g(\bxi(t))$
is a simple eigenvalue of $A(\bk), \{\bxi(t)\} = \bk$
for each $t\in [t_1, t_2]$, we have $g(\bxi(t))=\l_j(A(\bk))$ with $j$
independent of the choice of $t$.
Since $g$ is continuous on $\CB$, the function $ g(\bxi(t))$ is a continuous
function of $t\in [t_1, t_2]$, and hence
the intermediate value theorem implies that
there is a $t_0\in (t_1, t_2)$ such that $\l_j(A(\{ \bxi(t_0)\})) = \l$.
The bound $\z(\l; A)\ge \d$ follows from the definition \eqref{zeta:eq} of $\z(\l; A)$.
\end{proof}

Our next step is to prove that there is a path
with the properties required in Lemma \ref{lem:simple1}.
In fact we shall prove that the required properties will hold for
an interval
$I(\boldsymbol\Om; \rho, \d)\subset (0, \infty)$ (see \eqref{iomd:eq})
with some $\boldsymbol\Om\in T(\rho)$.

\begin{lem}\label{lem:simple3}
There exists a constant $Z\ge 1$ with
the following property.
Suppose that for some $\boldsymbol\Om\in T(\rho)$ and some
$t\in I(\boldsymbol\Om; \rho,  \d), \d\in (0, \rho^{2m}/4]$,
the number $g(\boldeta)$, $\boldeta = t \boldsymbol\Om$
is a multiple eigenvalue
of $A(\bk), \bk = \{\boldeta\}$. Then for any
$\tau\in  I(\boldsymbol\Om; \rho, \d)$
there exists a vector $\bn\in\SG^\dagger\setminus \{\mathbf 0\}$ such that
$\tau\boldsymbol\Om+\bn\in\CA(\rho, Z \d)$.
\end{lem}

\begin{proof}
Since the number $g(\boldeta), \boldeta = t\boldsymbol\Om,$
is a multiple eigenvalue, by definition of
the function $g(\ \cdot\ )$,
there is a vector $\bp\in\SG^\dagger \setminus\{\mathbf 0\}$ such that
$g(\boldeta) = g(\boldeta + \bp)$. In view of \eqref{twosided:eq},
$ |\boldeta+\bp|\asymp \rho$. Thus by Lemma \ref{partialofs2}, for any
$\tau\in I(\boldsymbol\Om; \rho, \d)$
there exist two vectors $\bm_1, \bm_2\in\SG^\dagger$, $\bm_1\not = \bm_2$
such that, with $\bxi = \tau\boldsymbol\Om$,
\begin{equation}\label{inthezone:eq}
\begin{cases}
|g(\boldeta) - g(\bxi+\bm_1)|\ll \rho^{2m-1}|\boldeta-\bxi|\ll \d,\\[0.2cm]
|g(\boldeta +\bp) - g(\bxi+\bm_2)|\ll \rho^{2m-1}|\boldeta-\bxi|\ll \d.
\end{cases}
\end{equation}
Here we have used the bound $|t-\tau|\ll \d \rho^{1-2m}$, which follows
from \eqref{iomd1:eq}. As $\bm_1\not =\bm_2$,
one of these  vectors is not zero. Denote this vector by $\bn$.
Since $g(\boldeta) = g(\boldeta+\bp)$, it follows from \eqref{inthezone:eq} that
\begin{equation*}
|g(\bxi+\bn) - g(\boldeta)|\ll \d,
\end{equation*}
 so that $\bxi+\bn\in\CA(\rho, Z \d)$ with some constant $Z$ independent
 of $\bxi$ and $\rho$, as required.
\end{proof}

The next Lemma is the cornerstone of our argument: it shows that at least for
one $\boldsymbol\Om\in T(\rho)$ the interval $I(\boldsymbol\Om; \rho, \d)$
consists entirely of the points $t$ such that $g(t\boldsymbol\Om)$ is a
simple eigenvalue.

\begin{lem}\label{lem:simple4}
There exists a vector $\boldsymbol\Om\in T(\rho)$ and a number $c_3>0$ such
that for $\d = c_3 \rho^{2m-4-d-12(d-1)^{-1}}$ and
each $t\in I(\boldsymbol\Om; \rho, \d) $
the number $g(\bxi)$, $\bxi = t\boldsymbol\Om$ is a simple eigenvalue of $A(\{\bxi\})$.
Moreover, $\z(\rho^{2m}; A)\ge \d$.
\end{lem}

\begin{proof}
Suppose the contrary, i.e. if $\rho$ is sufficiently large,
then for any $\boldsymbol\Om\in T(\rho)$ there is a
$t\in I(\boldsymbol\Om; \rho, \d)$ such that $g(t\boldsymbol\Om)$ is a
multiple eigenvalue of $A({t\boldsymbol\Om})$.
Then due to formula \eqref{tildecb:eq}, Lemma \ref{lem:simple3} implies that
\begin{equation}\label{cover0}
\tilde\CB(\rho, \d)\subset
\bigcup_{\bn\in\SG^\dagger\setminus\{0\}}(\CA(\rho, \de_1)+\bn)
\end{equation}
with $\de_1:=Z\de$. Since $\tilde\CB(\rho, \d)\subset\CB(\rho, \d_1)$, we can
re-write \eqref{cover0} as
\begin{equation}\label{cover1}
\begin{split}
\tilde\CB(\rho, \de) &\subset \bigcup_{\bn\in\SG^\dagger\setminus\{\mathbf 0\}}
\Bigl(\bigl(\CA(\rho, \de_1)+\bn\bigr)\bigcap\CB(\rho, \de_1)\Bigr)\\
&=\bigcup_{\bn\in\SG^\dagger\setminus\{\mathbf 0\}}
\Bigl(\bigl(\CB(\rho, \de_1)+\bn\bigr)\bigcap\CB(\rho, \de_1)\Bigr)\bigcup
\bigcup_{\bn\in\SG^\dagger\setminus\{\mathbf 0\}}
\Bigl(\bigl(\CD(\rho, \de_1)+\bn\bigr)\bigcap\CB(\rho, \de_1)\Bigr).
\end{split}
\end{equation}
Let us estimate the volumes of sets on both sides of this inclusion.
For a fixed $\ep>0$, whose value is chosen a few lines down,
assume that $\d \rho^{2-2m+2\ep}\to 0$, $\rho\to\infty$,
we can use \eqref{sobranie:eq} and \eqref{BD:eq} for the volume of the right hand side.
For the left hand side we use \eqref{volume_tildeCB:eq}, so that \eqref{cover1}
results in the estimate
\begin{equation*}
\d\rho^{d-2m}\ll \d^2 \rho^{4-4m + 2d+6\ep} + \d \rho^{1-2m+d - \ep(d-1)}
 + \d\rho^{d-1-2m+\a_{d}},
\end{equation*}
which simplifies to
\begin{equation*}
1 \ll \d \rho^{4-2m + d + 6\ep} + \rho^{1-\ep(d-1)} + \rho^{-1+\a_{d}}.
\end{equation*}
Choose $\ep= 2(d-1)^{-1}$ and $\d = c_3 \rho^{2m-4-d-6\ep}$ with a
suitably small $c_3$. Then for large $\rho$ the right hand side
is less than the left hand side, which produces a contradiction, thus
proving the Lemma.
\end{proof}

\subsection{Proof of the Main Theorem}
We assume that the conditions of Theorem \ref{main:thm} are satisfied.
The proof uses the reduction of the operator $H$ to $A_1$, established in Theorem
\ref{reduction:thm}. The first step is to show that the spectrum of $A_1$ is
well approximated by that of the model operator \eqref{model:eq} with $B$ replaced with $X$,
i.e.
\[
A = H_0+ X^o + X^\flat.
\]
Let numbers $\a_j<1, j = 1, 2, \dots, d$ be as defined in Subsection \ref{resonant:subsect}.

\begin{lem}\label{atoa1:lem}
Suppose that the conditions of Theorem \ref{main:thm} are satisfied.
Let $A_1$ be the operator \eqref{a1:eq}, and let $r = \rho^{\vark}$
with a number $\vark>0$, satisfying \eqref{eq:condition1} and the inequality
\begin{equation}\label{vark:eq}
d^2\vark < (2m-\a\b)\a_d.
\end{equation}
Then for any $L>0$ there exists an $M$ (i.e.
the number of steps in Theorem \ref{reduction:thm}) such that
\begin{equation}\label{loca:eq}
N(\mu- \rho^{-L}, A(\bk))\le N(\mu, A_1(\bk))\le N(\mu + \rho^{-L}, A(\bk))
\end{equation}
for all $\mu\in \bigl((1-c_4)^{2m}\rho^{2m}, (1+c_4)^{2m}\rho^{2m}\bigr)$ with
any $c_4 < 1/32$.
\end{lem}

\begin{proof}
By Theorem \ref{reduction:thm}, $\|R_{M+1}\|\ll \rho^{\b\e_{M+1}}$,
uniformly in $b: \1 b\1^{(\a)}\ll 1$ (see \eqref{sigmaeps:eq} for definition of $\e_{M+1}$).
The condition \eqref{alm1:eq} is equivalent to $\s <1$,
so that $\e_j\to -\infty$ as $j\to\infty$.  Thus
for sufficiently large $M=M(L)$ we have $\|R_{M+1}\|\ll \rho^{-L}/2$. As a consequence,
\begin{gather}
N\bigl(\mu - \rho^{-L}/2, \tilde A_1(\bk)\bigr)
\le N(\mu, A_1(\bk))
\le N\bigl(\mu + \rho^{-L}/2, \tilde A_1(\bk)\bigr),\notag \\
\tilde A_1 = A + X^{\downarrow, \sharp, \uparrow},\label{loca2:eq}
\end{gather}
for all $\mu\in\R$.
Due to \eqref{orhogonal:eq}, the operator
$\tilde A_1$ can be represented in the block-matrix form:
\begin{equation*}
\tilde A_1 = \underset{\GV\in\mathcal W(r)}\bigoplus \CP(\Xi(\GV))
A_{\GV}\CP(\Xi(\GV))
+ \underset{\GV, \GW\in\mathcal W(r),}\bigoplus \CP(\Xi(\GV))
X^{\downarrow, \sharp, \uparrow}\CP(\Xi(\GW)).
\end{equation*}
Since the number of distinct subspaces $\GW\in\mathcal W(r)$  is bounded above by
$ C r^{d^2}$ with some universal constant $C>0$,
the second term satisfies the two-sided estimate
\begin{align*}
- Cr^{d^2} \underset{\GV\in\mathcal W(r)}\bigoplus \CP(\Xi(\GV))
|X|^{\downarrow, \sharp, \uparrow} \CP(\Xi(\GV))
\le &\ \underset{\GV, \GW\in\mathcal W(r),}\bigoplus \CP(\Xi(\GV))
X^{\downarrow, \sharp, \uparrow}\CP(\Xi(\GW))\\[0.2cm]
\le &\
Cr^{d^2} \underset{\GV\in\mathcal W(r)}\bigoplus \CP(\Xi(\GV))
|X|^{\downarrow, \sharp, \uparrow} \CP(\Xi(\GV)).
\end{align*}
Here we have denoted $|X|^{\downarrow, \sharp, \uparrow}
= |X^{\downarrow}| + |X^{\sharp}| + |X^{\uparrow}|$.
Consequently,
\[
\tilde A_-\le \tilde A_1\le \tilde A_{+}
\]
with
\[
\tilde A_\pm = \bigoplus_{\GV\in\mathcal W(r)}\CP(\Xi(\GV))
\bigl(A_{\GV} \pm  Cr^{d^2} |X|^{\downarrow, \sharp, \uparrow}\bigr)\CP(\Xi(\GV)).
\]
Since $\tilde A_{\pm}$ are orthogonal sums,
the problem is reduced to estimating the counting functions of $\tilde A_{\pm}(\bk)$
on each invariant subspace $\GH\bigl(\bk; \Xi(\GV)\bigr)$. From now on we assume that
$\GV$ is fixed and omit it from the notation.

If $\GV\in\CV(r, d)$, i.e. $\GV = \R^d$, then
$\Xi = \Xi(\GV)\subset B(0, 2\rho^{\a_d})$, see Lemma \ref{lem:properties3}.
Clearly, $\| H_0 \CP(\Xi)\|\le \rho^{2m\a_d}$. Also, by \eqref{subord:eq},
\[
\1 x^o\1^{(\a)} + \1 x^{\downarrow}\1^{(\a)}
+ \1 x^{\sharp)}\1^{(\a)}
+ \1 x^{\uparrow}\1^{(\a)}\ll \1 b\1^{(\a)},
\]
and hence, by Lemma \ref{formbound:lem},
\[
\| X^o \CP(\Xi)\| + r^{d^2}\|\CP(\Xi) |X|^{\downarrow, \sharp, \uparrow}
\CP(\Xi)\|\ll r^{d^2}\rho^{\a\b\a_d}.
\]
In view of \eqref{vark:eq}, the right hand side does not exceed $\rho^{2m\a_d}$.
Consequently, $\|\tilde A_{\pm}\CP(\Xi)\|\ll \rho^{2m\a_d}$,
which implies that $N(\mu, \tilde A_\pm(\bk); \Xi) = 0$ for all
$\mu\ge (\rho/2)^{2m}$.

Now, let us fix $\GV\in\CV(r, n), n\le d-1,$ and prove the bounds
\begin{equation}\label{loca1:eq}
N(\mu- \rho^{-L}/2, A_{\GV}(\bk); \Xi)
\le N(\mu, \tilde A_\pm(\bk); \Xi)\le N(\mu + \rho^{-L}/2, A_{\GV}(\bk); \Xi),
\end{equation}
for sufficiently large $\rho$.
Split $\Xi$ into three disjoint sets:
\begin{align*}
\Xi = & \CC_<\cup \CC_0\cup \CC_>, \\[0.2cm]
\CC_0 = & \{\bxi\in\Xi:  7\rho/8\le |\bxi_{\GV^\perp}| \le 17\rho/16\},\\[0.2cm]
\CC_< = & \{\bxi\in\Xi:   |\bxi_{\GV^\perp}| < 7\rho/8\}, \ \
\CC_> = \{\bxi\in\Xi:  17\rho/16 < |\bxi_{\GV^\perp}|\}.
\end{align*}
Note that by definition of the operator $A_{\GV}$ (see \eqref{bflatv:eq})
all three subspaces $\CH(\CC_0), \CH(\CC_<), \CH(\CC_>)$
(see Subsection \ref{fibre:subsect}) are invariant for $A_{\GV}$.
Since $|\bxi_{\GV}|<2\rho^{\a_{d-1}}$ (see Lemma \ref{lem:properties3}),
we have
\begin{gather*}
\Xi\cap B(0, 7\rho/8)\subset\CC_<\subset B(0, 29\rho/32), \\[0.2cm]
\Xi\cap B(17\rho/16)\subset (\CC_<\cup \CC_0)\subset B(0, 9\rho/8).
\end{gather*}
Therefore, by Lemma \ref{smallorthog:lem},
\begin{equation*}
\CP(\Xi)|X^{\downarrow}| \CP(\Xi) = \CP(\CC_<)|X^{\downarrow}| \CP(\CC_<),\
\CP(\Xi)|X^{\sharp}| \CP(\Xi) = \CP(\CC_>)|X^{\sharp}| \CP(\CC_>).
\end{equation*}
Thus $\tilde A_{\pm}\CP(\Xi)$ can be rewritten as
\begin{equation*}
\tilde A_{\pm}\CP(\Xi) = F_{\pm}  \pm Cr^{d^2}
\bigl(\CP(\Xi)|X|^{\uparrow}\CP(\Xi)
- \CP(\CC_>) |X|^{\uparrow}\CP(\CC_>)\bigr),
\end{equation*}
with
\begin{align*}
F_{\pm}
= &\ \CP(\CC_<) \bigl(A_{\GV} \pm Cr^{d^2} |X^{\downarrow}|\bigr) \CP(\CC_<)
 \oplus
\CP(\CC_0 )A_{\GV} \CP(\CC_0)\\
&\ \ \ \ \ \ \ \ \ \ \ \ \ \ \ \ \ \ \oplus
\CP(\CC_> )
\bigl(A_{\GV} \pm Cr^{d^2} |X|^{\uparrow, \sharp}\bigr) \CP(\CC_>).
\end{align*}
By \eqref{shar:eq},
\begin{equation*}
r^{d^2}\| \CP(\CC_<\cup \CC_0)|X^{\uparrow}|\| +
r^{d^2}\| |X^{\uparrow}|\CP(\CC_<\cup \CC_0) \|\ll r^{d^2 + p-l} \rho^{\b\max(\a, 0)},
\end{equation*}
for any $p>d$ and $l \ge p$ uniformly in $b$ satisfying $\1 b\1^{(\a)} \ll 1$.
As $r = \rho^\vark$, $\vark >0$, by choosing a sufficiently large
$l$, we can guarantee that the right hand side is bounded by $\rho^{-L}/2$. This leads
to the bounds
\begin{equation}\label{loca4:eq}
N(\mu- \rho^{-L}/2, F_{\pm}(\bk); \Xi)
\le N(\mu, \tilde A_\pm(\bk); \Xi)\le N(\mu + \rho^{-L}/2,
F_{\pm}(\bk); \Xi),
\end{equation}
for all $\mu\in\R$.
Consequently,
\eqref{loca1:eq} will be proved if we show that
 \begin{equation}\label{loca3:eq}
N(\mu, F_{\pm}; \Xi) = N(\mu, A_{\GV}; \Xi),\
\biggl(\frac{15\rho}{16}\biggr)^{2m}\le \mu \le \biggl(\frac{33\rho}{32}\biggr)^{2m}.
\end{equation}
To this end note first that the definition of $\CC_<$ and $\CC_>$ implies
\begin{equation}\label{partial:eq}
H_0 \CP(\CC_<)\le (29\rho/32)^{2m}\CP(\CC_<),\ \
H_0\CP(\CC_>)\ge (17\rho/16)^{2m}\CP(\CC_>).
\end{equation}
Also, by Lemma \ref{smallorthog:lem},
\[
\|X^o\|+\| X^{\flat}_{\GV}\| + Cr^{d^2}\| X^{\downarrow}\|\ll r^{d^2}\rho^{\b\max(\a, 0)}.
\]
Under the condition \eqref{vark:eq} the right hand side of this estimate is bounded by
$ o(\rho^{2m})$, $\rho\to\infty$ uniformly in $b$.
Together with \eqref{partial:eq}, this entails that
\begin{equation}\label{c<:eq}
 N(\mu, A_{\GV} \pm Cr^{d^2} \CP(\CC_<)|X^{\downarrow}|\CP(\CC_<); \CC_<)
 = N(\mu, A_{\GV}; \CC_<),\ \ \mu\ge \biggl(\frac{15\rho}{16}\biggr)^{2m}.
\end{equation}
Furthermore, in view of \eqref{subord:eq} and \eqref{formbound1:eq},
\begin{equation*}
\CP(\CC_{>})(|X^o| + |X^{\flat}_{\GV}| + C r^{d^2}|X|^{\uparrow, \sharp})\CP(\CC_>)
\ll r^{d^2} (H_0+I)^{\g}\CP(\CC_>),\ \g = \frac{\a\b}{2m}.
\end{equation*}
Using again \eqref{vark:eq} and remembering \eqref{partial:eq},
we conclude that the right hand side
is estimated above by $o(1) H_0\CP(\CC_>)$, $\rho\to\infty$, uniformly in $b$.
Together with \eqref{partial:eq} this implies that
\begin{equation}\label{c>:eq}
 N(\mu, A_{\GV} \pm Cr^{d^2} \CP(\CC_>)|X|^{\uparrow, \sharp}\CP(\CC_>); \CC_>)
 = 0,\ \ \mu\le \biggl(\frac{33\rho}{32}\biggr)^{2m}.
\end{equation}
Putting together \eqref{c<:eq} and \eqref{c>:eq}, we arrive at \eqref{loca3:eq}.
In combination with \eqref{loca4:eq} this leads to \eqref{loca1:eq}.
Together with \eqref{loca2:eq} they yield
\eqref{loca:eq}.
\end{proof}

\begin{proof} [Proof of the Main Theorem]
By Theorem \ref{reduction:thm}, it suffices to prove that
$\z(\rho^{2m}, A_1)> c \rho^{S}$ with some $S$ for sufficiently large $\rho$.
It follows from Lemma \ref{lem:simple4} that $\z(\rho^{2m}; A)\ge c \rho^{S}$
with $S = 2m-4-d-12(d-1)^{-1}$.
Using the bounds \eqref{loca:eq} with $L>-S$, we get the required estimate
$\z(\rho^{2m}, A_1)\gg \rho^S$ from the definition \eqref{zeta1:eq}.
 This completes the proof of Theorem \ref{main:thm}.
\end{proof}


\bibliographystyle{amsplain}

\end{document}